%% file: FV_AT_MW_arxiv.tex
\documentclass[aop,preprint]{imsart}

\usepackage{amsmath,amsthm,amssymb,mathtools}
\usepackage{multicol,multirow,hhline,comment,graphicx,xcolor,url}
\usepackage{mathrsfs}	
\usepackage{cellspace}		
\usepackage{hyperref}

\usepackage{enumitem} 
\setlist[enumerate,1]{leftmargin=1cm}

\usepackage{accents,stackengine,scalerel} 

\input{ADMacros}

\newcommand{\sskewer}{\textsc{sSkewer}}
\def\sskewerP{\widebar{\sskewer}}	

\numberwithin{equation}{section}
\numberwithin{figure}{section}
\numberwithin{table}{section}

\def\umCladeAbar{\overline{\nu}_{\perp {\rm cld}}^{(\alpha)}}

\begin{document}

\begin{frontmatter}

\title{A two-parameter family of\\ measure-valued diffusions with\\ Poisson--Dirichlet stationary distributions\thanksref{T0}}

\runtitle{Measure-valued PD$(\alpha,\theta)$ diffusions}
\runauthor{Forman, Rizzolo, Shi and Winkel}

\thankstext{T0}{This research is partially supported by NSF grants {DMS-1204840, DMS-1308340, DMS-1612483, DMS-1855568}, UW-RRF grant A112251, EPSRC grant EP/K029797/1.}

\begin{aug}
  \author{\fnms{Noah} \snm{Forman}\thanksref{m1}\ead[label=e1]{noah.forman@gmail.com}},
  \author{\fnms{Douglas} \snm{Rizzolo}\thanksref{m3}\ead[label=e3]{drizzolo@udel.edu}},
  \author{\fnms{Quan} \snm{Shi}\thanksref{m4}\ead[label=e4]{quanshi.math@gmail.com}},\\ and 
  \author{\fnms{Matthias} \snm{Winkel}\thanksref{m5}\ead[label=e5]{winkel@stats.ox.ac.uk}}
  
  \affiliation{McMaster University\thanksmark{m1}, 
  University of Delaware\thanksmark{m3}, Universit\"at Mannheim\thanksmark{m4}, and University of Oxford\,\thanksmark{m5}}
  
  \address{Department of Mathematics \& Statistics\\ McMaster University\\ 1280 Main Street West\\ Hamilton, Ontario L8S 4K1\\ Canada\\
   \printead{e1}
  }
  
  \address{Department of Mathematical Sciences\\ University of Delaware\\ Newark DE 19716\\ USA\\
   \printead{e3}
  }
  
  \address{Institut f\"ur Mathematik\\ Universit\"at Mannheim\\ Mannheim, 68159\\ Germany\\
   \printead{e4}
  }
  
  \address{Department of Statistics\\ University of Oxford\\ 24--29 St Giles'\\ Oxford OX1 3LB\\ UK\\
   \printead{e5}
  }
\end{aug} 
  
 \begin{abstract}
  We give a pathwise construction of a two-parameter family of purely-atomic-measure-valued diffusions in which ranked masses of atoms are stationary with the Poisson--Dirichlet$(\alpha,\theta)$ distributions, for $\alpha\in (0,1)$ and $\theta\ge 0$. This resolves a conjecture of Feng and Sun (2010). We build on our previous work on $\left(\alpha,0\right)$- and $\left(\alpha,\alpha\right)$-interval partition evolutions. Indeed, we first extract a self-similar superprocess from the levels of stable processes whose jumps are decorated with squared Bessel excursions and distinct allelic types. We complete our construction by time-change and normalisation to unit mass. In a companion paper, we show that the ranked masses of the measure-valued processes evolve according to a two-parameter family of diffusions introduced by Petrov (2009), extending work of Ethier and Kurtz (1981). These ranked-mass diffusions arise as continuum limits of up-down Markov chains on Chinese restaurant processes. 
 \end{abstract}



\end{frontmatter}

\section{Introduction}

$\!\!\!\!$The well-known \emph{labeled infinitely-many-neutral-alleles model} introduced by Ethier and Kurtz \cite{EthKurtzBook,EthiKurt93} is a one-parameter family of Fleming--Viot processes \cite{FlemViot79}.  They showed that these measure-valued diffusions have stationary distributions that are now known \cite{IshwJame03,PitmYorPDAT,Teh06} in the Bayesian non-parametrics community as the $\alpha\!=\!0$ case of the Pitman--Yor process ${\tt PY}(\alpha,\theta,\mu)$, for a 
parameter $\theta\!>\!0$ and a probability measure $\mu$ on an underlying space of alleles.
For general $\alpha\!\in\![0,1]$, $\theta\!>\!-\alpha$, the ranked atom sizes of ${\tt PY}(\alpha,\theta,\mu)$ form a Poisson--Dirichlet sequence ${\tt PD}(\alpha,\theta)$ and their locations are independent with common distribution $\mu$. 

In this paper, we construct a family of measure-valued diffusions in the two-parameter setting  $\alpha \in (0,1)$ and $\theta\ge 0$. 
The existence of such processes has been conjectured by Feng and Sun \cite{FengSun10}.
Some progress has been made by analytic methods \cite{Petrov09,CdBERS17,FengSun19} to establish processes on spaces of decreasing sequences, on spaces allowing only finitely many atoms, and in the case $\alpha=1/2$. 
Our approach is probabilistic, building on our previous work \cite{Paper1-1,Paper1-2} on interval-partition-valued diffusions. 
This enables us to study the evolution of individual atoms from their creation until reaching zero size again,  corresponding to the creation and extinction of allelic types in the language of Ethier and Kurtz \cite{EthiKurt81,EthiKurt93}. 

Petrov \cite{Petrov09} introduced ${\tt PD}(\alpha,\theta)$ diffusions, which generalize the process of ranked atom sizes \cite{EthiKurt81} in the labeled infinitely-many-neutral-alleles diffusions \cite{EthiKurt93}. 
%
In a companion paper \cite{Paper1-3}, we use analytic arguments to show that projections of our construction onto the space $\nabla_\infty$ of ranked atom sizes evolve as Petrov's diffusions, hence fully resolving Feng and Sun's conjecture.  
Our discussion here does not cover the case $\theta \in (-\alpha, 0)$, which will be explored in a forthcoming work. 
  
In the following, we model the space of alleles by the interval $[0,1]$ and take as $\mu$ the uniform distribution \Unif[0,1], which may be thought of as the distribution of the allelic type of a mutant offspring of a gamete of any allelic type. 
This specialization is only for simplicity; the same results hold for any compact metric space endowed with any atom-free distribution $\mu$. In this context we denote the Pitman--Yor distribution
by $\PDRMAT$, so that
\begin{equation}\label{eqn:PDRM}
  \sum_{i\ge 1} A_i\delta(U_i) \sim \PDRMAT,\ \ \mbox{with}\ 
  \begin{array}{l}
     (A_i)_{i\ge 1}\!\sim\! \PoiDirAT,\\ 
     U_i\!\sim\!\Unif[0,1], i\! \ge\! 1,
  \end{array}\ 
  \mbox{independent.}
\end{equation}
These Poisson--Dirichlet random measures will be the stationary distributions of our measure-valued diffusions. Pitman--Yor processes are also called \emph{three-parameter Dirichlet processes} in \cite{Carlton02}, as  
they are equivalent to what has also been referred to as the Dirichlet processes in Bayesian nonparametric statistics when $\alpha\!=\!0$.  
We also refer to \cite{Pitman96, JLP08} for more studies on this family.

\subsection{Main results}\label{sec:results}

Let $\cM^a\subset\cM$ be the space of all purely atomic finite measures on $[0,1]$, as a subspace of the Prokhorov space $(\cM,d_{\cM})$. 
For any measure $\pi\in \cM$, write $\|\pi\|:= \pi([0,1])$ for its \em total mass\em. 
We denote by $\cM_1^a=\{\pi\in\cM^a\colon\|\pi\|=1\}$ the subspace of atomic probability measures on $[0,1]$. 
Our construction of $\cM_1^a$-valued diffusions is in two steps. 
The first step is to construct self-similar $\cM^a$-valued diffusions with a branching property, with fluctuating total mass. 
The second step is to time-change these diffusions and to renormalise by their total mass. 
In this measure-valued context, this is reminiscent of the skew-product relationship between the measure-valued branching diffusions and probability-measure-valued Fleming--Viot 
process of \cite{KonnoShiga88,Shiga1990,EtheMarc91}. 

For the first step, we define semi-groups $(K_y^{\alpha,\theta},y\ge 0)$ on $(\cM^a,d_{\cM})$ that possess the \em branching property \em under which the state at time $y$ can be seen 
as the sum of a family of independent random measures indexed by the atoms of the initial measure (plus a further random measure -- immigration). Indeed, we can think of a genealogy in which the atoms of each of the 
independent 
random measures are the descendants of the corresponding atom of the initial measure. In the cases $\theta=0$ and $\theta=\alpha$, this definition is motivated by related semi-groups on 
spaces of interval partitions \cite{Paper1-2}. Ingredients there include distributions of random variables $L_{b,r}^{(\alpha)}$ with Laplace transform\vspace{-.2cm}
\begin{equation}
 \EV\left[e^{-q L_{b,r}^{(\alpha)}}\right]=\left(\frac{r+q}{r}\right)^\alpha\frac{e^{br^2/(r+q)}-1}{e^{br}-1}, \quad q\ge0,\ b>0,\ r>0,\label{eqn:LMB}\vspace{-.1cm}
\end{equation}
which we use here for atom sizes. As atom locations, we consider a given type $x\!\in\![0,1]$ and a new type $U_0\!\sim\!\Unif:=\Unif[0,1]$ with mixing probabilities\vspace{-.1cm}
$$p_{b,r}^{(\alpha)}(c)=\frac{I_{1+\alpha}(2r\sqrt{bc})}{I_{-1-\alpha}(2r\sqrt{bc})+\alpha(2r\sqrt{bc})^{-1-\alpha}/\Gamma(1-\alpha)}
\quad\mbox{and}\quad 1-p_{b,r}^{(\alpha)}(c),\vspace{-.1cm}$$
where for $v\in\mathbb{R}$, we denote by $I_v$ the modified Bessel function of the first kind of index $v$. 
Independently of $L_{b,r}^{(\alpha)}$ and $U_0$, scale $\overline{\Pi}\sim\PDRM[\alpha,\alpha]$ by an independent total mass $G\!\sim\!\GammaDist[\alpha,r]$ to obtain $\Pi:=G\overline{\Pi}$.
Then define the distribution $Q_{b,x,r}^{(\alpha)}$ on $\mathcal{M}^a$ of a random measure as \vspace{-.2cm}
\begin{align}\nonumber
 Q_{b,x,r}^{(\alpha)}=e^{-br}\delta_0+(1\!-\!e^{-br})\!&\int_0^\infty\!\Big(p_{b,r}^{(\alpha)}(c)\Pr\!\left\{c\delta(x)\!+\!\Pi\in\cdot\,\right\}\\
\label{eq:Qbxralpha}					&\quad+(1\!-\!p_{b,r}^{(\alpha)}(c))\Pr\!\left\{c\delta(U_0)\!+\!\Pi\in\cdot\,\right\}\!\!\Big)\Pr\{L_{b,r}^{(\alpha)}\!\in\! dc\}.\vspace{-.3cm}
\end{align}
The idea is that we associate with an atom of size $b$ at location $x$, under $Q_{b,x,r}^{(\alpha)}$, either no descendants (the zero measure $0$) with probability $e^{-br}$;
or otherwise, the descendants are the atoms of $\Pi$, and there is also an atom of size $L_{b,r}^{(\alpha)}$, which when $L_{b,r}^{(\alpha)}=c$ is still at $x$ with probability 
$p_{b,r}^{(\alpha)}(c)$, or otherwise is at an independent uniformly distributed location $U_0$. 
\begin{definition}[Transition kernel $K^{\alpha,\theta}_y$]\label{def:kernel:sp}
 Fix $\alpha\in (0,1)$, $\theta\ge 0$ and let $\pi = \sum_{i\ge 1} b_i \delta(x_i) \in\cM^a$.
 For $y>0$, we define $K^{\alpha,\theta}_y(\pi,\,\cdot\,)$ to be the distribution on $\mathcal{M}^a$ of the random measure $G^y\overline{\Pi}_0+\sum_{i\ge 1}\Pi_i^y$ 
 for independent $G^y\sim\GammaDist[\theta,1/2y]$, $\overline{\Pi}_0\sim\PDRMAT$ and $\Pi_i^y\sim Q_{b_i,x_i,1/2y}^{(\alpha)}$, $i\ge 1$.
\end{definition}
One can easily check that there are a.s.\ only finitely many $\Pi_i^y\neq 0$ (cf. \cite[Lemma 6.1]{Paper1-1}) 
and thus $K^{\alpha,\theta}_y(\pi,\,\cdot\,)$ is indeed well-defined as a measure on the Borel sets of $(\cM^a,d_\cM)$. \pagebreak 
Note, however, that it is not clear a priori that $(K_y^{\alpha,\theta},y\ge 0)$ forms a transition semigroup.
We do not attempt to settle in this paper the natural question of what other $L_{b,r}^{(\alpha)}$,
$p_{b,r}^{(\alpha)}(c)$, $\Pi$ etc. lead to transition semigroups. See however Section \ref{sec:construction} for a construction that sheds some light on this question. For now, we
state our theorem for $(K^{\alpha,\theta}_y,y\!\ge\! 0)$:  

\begin{theorem}	\label{thm:sp}
	Fix $\alpha\in (0,1)$ and $\theta\ge 0$. The family $(K^{\alpha,\theta}_y, y\ge 0)$ forms the transition semigroup of a path-continuous Hunt process on $(\cM^a,d_{\cM})$.
\end{theorem}

We shall refer to these measure-valued processes as \emph{self-similar $(\alpha,\theta)$-super\-processes}, or ${\tt SSSP}(\alpha,\theta)$, in view of the following proposition.
We write ${\tt SSSP}_\pi(\alpha,\theta)$ for the distribution of an ${\tt SSSP}(\alpha,\theta)$ starting from $\pi\in\cM^a$.

\begin{proposition}[Self-similarity]\label{prop:scaling:sp}
	Fix $\alpha\!\in\! (0,1)$, $\theta\!\ge\!0$, $\pi\!\in\!\cM^a$ and $c\!>\!0$. Then $(\pi^y,y\!\geq\! 0)\sim{\tt SSSP}_\pi(\alpha,\theta)$ implies $(c \pi^{y/c},y\!\geq\! 0)\sim{\tt SSSP}_{c\pi}(\alpha,\theta)$.
\end{proposition}

\begin{proposition}[Additivity property]\label{prop:branching}
 Fix $\alpha\!\in\! (0,1)$, $\theta_1,\theta_2\!\ge\!0$, and 
 mutually singular measures 
 $\pi_1,\pi_2\!\in\!\cM^a$. 
 If $(\pi_1^y,\,y\!\geq\! 0)\sim{\tt SSSP}_{\pi_1}(\alpha,\theta_1)$ and $(\pi_2^y,\,y\!\geq\! 0)\!\sim\!{\tt SSSP}_{\pi_2}(\alpha,\theta_2)$ are independent then $(\pi_1^y\!+\!\pi_2^y,\,y\!\geq\! 0)$ is an ${\tt SSSP}_{\pi_1+\pi_2}(\alpha,\theta_1\!+\!\theta_2)$.
\end{proposition}

For any $r\!\in\! \mathbb{R}$, $z\!\ge\! 0$ and $B$ a standard one-dimensional Brownian motion, it is well-known that there exists a unique strong solution to the equation 
\[
 Z_t = z + r t + 2 \int_0^t \sqrt{|Z_s|} d B_s, 
\]
which is called an \emph{$r$-dimensional squared Bessel process} starting from $z$ and denoted by $\BESQ_{z}(r)$.  
The Feller diffusion, which is a continuous-state branching process that arises as a scaling limit of critical Galton--Watson processes, is \BESQ[0]. For $r>0$, \BESQ[r] is a Feller diffusion with immigration. The case $r < 0$ can be interpreted as a Feller diffusion with emigration at rate $|r|$. In this case, as when $r=0$, the boundary point 0 is not an entrance boundary for $(0,\infty)$, while exit at 0 (we will then force absorption) happens almost surely. For $r=d\in\BN$, the squared norm of a $d$-dimensional Brownian motion is a \BESQ[d]. See \cite{PitmYor82,GoinYor03,Pal13}.

\begin{theorem}[Total mass]\label{thm:mass:sp}
	For $\alpha\in (0,1)$, $\theta\ge0$ and $\pi\in\cM^a$, the total mass of an ${\tt SSSP}_\pi(\alpha,\theta)$ evolves as a $\BESQ_{\|\pi\|}(2 \theta)$.
\end{theorem}

For the second step, to construct a probability-measure-valued diffusion, let $\boldsymbol{\pi}:= (\pi^y, y\ge 0)$ be an ${\tt SSSP}_\pi(\alpha,\theta)$. 
We define a time-change function by
\begin{equation}\label{eq:tau-pi}
\rho_{\boldsymbol{\pi}}(u):= \inf \left\{ y\ge 0\colon \int_0^y \|\pi^z\|^{-1} d z>u \right\}, \qquad u\ge 0.\pagebreak
\end{equation}

\begin{definition}[$(\alpha, \theta)$-Fleming--Viot process]\label{def:FV}
	Let $\boldsymbol{\pi}:= (\pi^y, y\ge 0)$ be an ${\tt SSSP}_\pi(\alpha,\theta)$ and consider the time-change function $\rho_{\boldsymbol{\pi}}$ of \eqref{eq:tau-pi}. 
	Then the $\cM_1^a$-valued process $\overline{\boldsymbol{\pi}}:= (\ol{\pi}^u,u\ge 0)$ defined by\vspace{-0.1cm}
	\[
	\ol{\pi}^u:= \left\| \pi^{\rho_{\boldsymbol{\pi}}(u)} \right\|^{-1} \pi^{\rho_{\boldsymbol{\pi}}(u)},\qquad u\ge 0,\vspace{-0.1cm}
	\]
	is called an \emph{$(\alpha, \theta)$-Fleming--Viot process}, or ${\tt FV}(\alpha,\theta)$.
\end{definition}

Our main result is the following.

\begin{theorem}\label{thm:dP-sp}
	$\!\!\!$For $\alpha\!\in\! (0,1)$ and $\theta\!\ge\!0$, the ${\tt FV}(\alpha, \theta)$ is a path-continuous Hunt process on $(\cM_1^a,d_\cM)$ and has    
        $\mathtt{PDRM}(\alpha,\theta)$ as a stationary distribution. 
\end{theorem}

This extends the well-known normalisation/time-change connection between the measure-valued branching diffusions and the Fleming--Viot processes of  \cite{KonnoShiga88,Shiga1990,EtheMarc91}. Indeed, we point out that in our setting both ${\tt SSSP}(\alpha,\theta)$ and 
${\tt FV}(\alpha,\theta)$ are time-homogeneous Markov processes and ${\tt SSSP}(\alpha,\theta)$ enjoys the additivity property. 
Our processes can be thought of as not having a spatial motion other than preserving existing allelic types -- the superprocesses evolve via fluctuating atom sizes that are absorbed at zero, and via mutation creating a countable dense set of new types that are independent and identically distributed, as in \cite[(3.1)--(3.2)]{Shiga1990}. 

While our processes are not special cases of \cite[(3.1)--(3.2)]{Shiga1990}, 
the constructions have a lot in common, and 
we explore this further in Section \ref{intro:shiga},
along with some additional useful properties of ${\tt FV}(\alpha,\theta)$. 
Their statements do not require the construction of Section \ref{sec:construction}, 
but it is instructive to see them in the light of our construction, as well.

%

\subsection{Construction from marked L\'evy processes}\label{sec:construction}

We construct ${\tt SSSP}(\alpha,\theta)$ and hence ${\tt FV}(\alpha,\theta)$ from marked stable L\'evy processes. 
Specifically, following \cite{Paper1-1,Paper1-2}, denote by $\nu_{\tt BESQ}^{(-2\alpha)}$ the Pitman--Yor excursion measure of 
$\BESQ(-2\alpha)$ on a space $\Exc$ of excursions away from zero \cite{PitmYor82},
with normalisation so that \vspace{-0.1cm}
  $$\nu_{\tt BESQ}^{(-2\alpha)}\{\zeta>y\}=\frac{\alpha}{2^\alpha\Gamma(1-\alpha)\Gamma(1+\alpha)}y^{-1-\alpha},\vspace{-0.1cm}$$ 
where $\zeta(f)=\inf\{y\!>\!0\colon f(y)\!=\!0\}$ is the \em length/lifetime \em of the excursion $f\!\in\!\Exc$. 
We represent $f$ as a function $f\colon\mathbb{R}\rightarrow[0,\infty)$ vanishing outside $[0,\zeta(f))$.  
Let $\mathbf{V}$ be a Poisson random measure ({\tt PRM}) on $[0,\infty)\times\Exc\times[0,1]$ with intensity measure ${\tt Leb}\otimes\nu_{\tt BESQ}^{(-2\alpha)}\otimes\Unif$, 
where ${\tt Leb}$ denotes one-dimensional Lebesgue measure. 
Mapping all points $(t,f_t,x_t)$ of $\mathbf{V}$ to times $t$ and excursion lengths $\zeta(f_t)$, 
we obtain a {\tt PRM} from which we can naturally define a zero-mean (stable) L\'evy process $\mathbf{X}$ 
whose jumps are $\zeta(f_t)$ at each time $t$ of a point $(t,f_t,x_t)$ of $\mathbf{V}$.
From the perspective of $\mathbf{X}$ we view $(f_t,x_t)$ as an \em atom size evolution \em and \em allelic type \em marking the jump of size 
$\zeta(f_t)$ at time $t$. We refer to $\mathbf{X}$ as \em scaffolding \em to which the marks are attached. See Figure \ref{fig:scaf-marks} for an illustration with finitely many jumps.
\begin{figure}
  \centering
	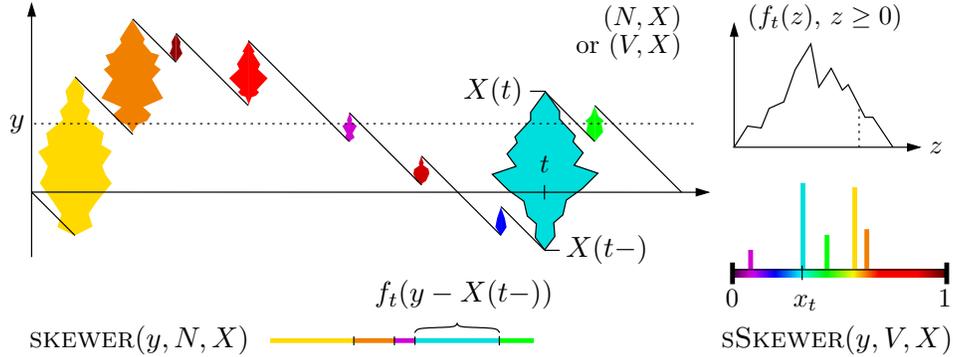
	\caption{A scaffolding with marks (atom size evolutions as spindle-shapes and allelic types from a color spectrum coded by $[0,1]$) and the skewer and superskewer at level $y$, not to scale.\label{fig:scaf-marks}}
\end{figure}
We will refer to the excursions marking the jumps of $\mathbf{X}$ as \em spindles\em.\pagebreak[2]

For $x\in[0,1]$ and $b>0$, denote by $\mathbf{Q}_{b,x}^{(\alpha)}$ the distribution of
\begin{equation}\label{eq:clade_law}
\widehat{\mathbf{V}}:=\delta(0,\mathbf{f},x)\!+\!\mathbf{V}|_{[0,T]\times\Exc\times[0,1]},\ 
   \mbox{for\,}\begin{array}{l}\mathbf{f}\!\sim\!{\tt BESQ}_{b}(-2\alpha)\mbox{ independent,}\\ 
                                                      T:=\inf\{t\ge 0\colon\zeta(\mathbf{f})\!+\!\mathbf{X}(t)\!=\!0\}.
                           \end{array}
\end{equation}
The L\'evy process $\widehat{\mathbf{X}}:=\zeta(\mathbf{f})+\mathbf{X}|_{[0,T]}$ starting from $\zeta(\mathbf{f})>0$ and stopped at its first passage at 0 
is fully determined by $\widehat{\mathbf{V}}$ and called the \em scaffolding associated with $\widehat{\mathbf{V}}$\em. 
This $\widehat{\mathbf{X}}$ provides for each point $(t,f_t,x_t)$ of $\widehat{\mathbf{V}}$ 
a birth level $\widehat{\mathbf{X}}(t-)$ and a death level $\widehat{\mathbf{X}}(t)=\widehat{\mathbf{X}}(t-)+\zeta(f_t)$ for the spindle $f_t$ 
so that $f_t(y-\widehat{\mathbf{X}}(t-))$ is naturally associated with level $y$. See Figure \ref{fig:scaf-marks}.
 
For $\pi\!=\!\sum_{i\ge 1}b_i\delta(x_i)\!\in\!\mathcal{M}^a$, consider independent $\mathbf{V}_i\!\sim\!\mathbf{Q}_{b_i,x_i}^{(\alpha)}$ and associated scaffolding $\mathbf{X}_i$, 
for all $i\!\ge\! 1$ with $b_i\!>\!0$. We write $\mathbf{Q}_\pi^{\alpha,0}$ for the distribution of the point measure $\mathbf{F}\!=\!\sum_{i\ge 1\colon b_i>0}\delta(\mathbf{V}_i,\mathbf{X}_i)$
on a suitable space of point measures on a suitable space of pairs $(V,X)$ of point measures $V$ of spindles and allelic types and scaffoldings $X$. See Section \ref{sec:pointmeas} for details.\pagebreak

\begin{definition}\label{def:superskewer}
          The \em superskewer \em at level $y$ of a pair $(V,X)$ of a point measure $V$ of spindles and allelic types and a scaffolding $X$ is defined as
          \begin{align*}
	\sskewer(y,V,X) := &\int f\big( y-X(t-) \big)\Dirac{x} V(dt,df,dx)\\
                                    =&\sum_{{\rm points}\;(t,f_t,x_t)\;{\rm of}\;V}f_t\big(y-X(t-)\big)\delta(x_t).\nonumber
	\end{align*}
          The superskewer at level $y\ge 0$ of $\mathbf{F}\sim\mathbf{Q}_\pi^{\alpha,0}$ for any $\pi\in\mathcal{M}^a$ is
	\begin{equation}\label{eq:sskewer}
	\sskewer(y,\mathbf{F}) := \sum_{{\rm points}\;(V,X)\;{\rm of}\;\mathbf{F}}\sskewer(y,V,X).
	\end{equation}
\end{definition}

We call this the superskewer as it constructs a superprocess by collecting from each level $y$ of the marked scaffolding(s) and placing onto the type space $[0,1]$ 
the atoms of the superprocess. 
This is a variant of the skewer map introduced in \cite{Paper1-1}, 
which similarly constructs an interval partition whose interval lengths (our atom sizes here) are placed onto $[0,\infty)$ by the skewer map in the left-to-right order of the jump times 
without leaving gaps as if on a skewer that pushes through the marked scaffolding from left to right. This skewer map does not record allelic types and therefore only depends on
$N(dt,df):=V(dt,df,[0,1])$. See Figure \ref{fig:scaf-marks} for an illustration and \cite[Definition 1.2]{Paper1-1} or \cite[Definition 1.7]{Paper1-2} 
for a precise definition of  $\skewer(y,N,X)$. 

\begin{theorem}\label{thm:a0:const}  For $\pi\in\cM^a$ and $\mathbf{F}\sim\mathbf{Q}_\pi^{\alpha,0}$, 
  the measure-valued process $\left(\sskewer(y,\mathbf{F}),y\!\ge\!0\right)$ is an ${\tt SSSP}_\pi(\alpha,0)$ as defined in/after Theorem \ref{thm:sp}.
\end{theorem}  

To similarly construct ${\tt SSSP}_\pi(\alpha,\theta)$ for $\theta>0$, we need to add immigration to $\mathbf{F}$. To this end, consider again $(\mathbf{V},\mathbf{X})$ as above. 
As we will demonstrate in Section \ref{sec:min_cld}, we can build scaffolding for immigration from excursions of $\mathbf{X}$ above its infimum process. 
Specifically, standard fluctuation theory for L\'evy processes yields a $\sigma$-finite excursion measure $\nu_{\perp{\rm stb}}^{(\alpha)}$. 
In our setting we will define a $\sigma$-finite measure $\overline{\nu}_{\perp {\rm cld}}^{(\alpha)}$ 
on a suitable space of pairs $(V,X)$ such that its pushforward onto the scaffolding $X$ is $\nu_{\perp{\rm stb}}^{(\alpha)}$. 

For $\cev{\mathbf{F}}\sim\mathbf{Q}_0^{\alpha,\theta}:=\PRM[\frac{\theta}{\alpha}{\tt Leb}\otimes\overline{\nu}_{\perp {\rm cld}}^{(\alpha)}]$, 
we interpret a point $(z,V_z,X_z)$ as immigration at level $z\ge 0$ and define the superskewer as
\begin{equation}\label{eq:sskewercev}
  \sskewer(y,\cev{\mathbf{F}})=\sum_{{\rm points}\;(z,V_z,X_z)\;{\rm of}\;\cev{\mathbf{F}}}\sskewer(y-z,V_z,X_z).
\end{equation}

\begin{theorem}\label{thm:at:const} Let $\alpha\in(0,1)$, $\theta>0$ and $\pi\in\mathcal{M}^a$. For $\cev{\mathbf{F}}\sim\mathbf{Q}_0^{\alpha,\theta}$ and an independent $\mathbf{F}\sim\mathbf{Q}_\pi^{\alpha,0}$, the measure-valued process\vspace{-0.1cm}
  \begin{equation}\label{timechFV}\big(\sskewer(y,\cev{\mathbf{F}})+\sskewer(y,\mathbf{F}),y\!\ge\!0\big)\vspace{-0.1cm}
  \end{equation} 
 is an ${\tt SSSP}_\pi(\alpha,\theta)$ as defined in/after Theorem \ref{thm:sp}.
\end{theorem}  


\subsection{Further properties of Fleming--Viot processes}\label{intro:shiga}

In this section, we state three further properties of ${\tt FV}(\alpha,\theta)$ relating to the number of atoms at exceptional times, 
the $\alpha$-diversity of the sequence of atom sizes 
as a proxy for the genetic diversity of types in the population modelled by a ${\tt FV}(\alpha,\theta)$,
and on connections to Shiga's construction \cite{Shiga1990}.

For $\pi\in\mathcal{M}^a$ denote by $N(\pi)\in\mathbb{N}\cup\{\infty\}$ the number of atoms of $\pi$. As ${\tt PDRM}(\alpha,\theta)$ has
infinitely many atoms a.s., for all $\theta\ge 0$, stationary ${\tt FV}(\alpha,\theta)$ processes have infinitely many atoms a.s.,
at each time $u\ge 0$. However, the following result states that there are exceptional times where this is not so.

\begin{theorem}\label{thm:property1} Let $1\le n<\infty$. 
  Then a ${\tt FV}(\alpha,\theta)$ process visits the set $\{\pi\in\mathcal{M}^a\colon N(\pi)=n\}$ with positive probability
  if and only if $\theta + n\alpha <1$. 
\end{theorem}

The corresponding result for Petrov's diffusions \cite{Petrov09}, which \cite{Paper1-3} identifies as the evolution of ranked atom sizes of a
${\tt FV}(\alpha,\theta)$, was proved by Feng and Sun \cite[Theorem 2.4]{FengSun10} using Dirichlet form techniques. Our argument is based
on the boundary behaviour of squared Bessel processes.

Let $\alpha\in(0,1)$. For any $\pi\in\mathcal{M}^a$, if the following limit exists, then we say that the \em $\alpha$-diversity of $\pi$ \em is 
$\IPLT_\alpha(\pi)$:\vspace{-0.1cm}
\begin{equation}
\IPLT_\alpha(\pi):= \Gamma(1-\alpha) \lim_{h\downarrow 0} h^{\alpha} \#\{ x\in [0,1]\colon \pi(\{x\})>h \}.\vspace{-0.1cm} 
\end{equation}
It is well-known that a ${\tt PDRM}(\alpha,\theta)$ has an $\alpha$-diversity almost surely, for all $\alpha\in(0,1)$ and $\theta>-\alpha$. This really
is a property of the ranked sequence of atom sizes, which is ${\tt PD}(\alpha,\theta)$. We refer to Pitman \cite[Theorem 3.13]{CSP}. 

\begin{theorem}\label{thm:property2} If $\pi\!\in\!\mathcal{M}^a_1$ has an $\alpha$-diversity, then 
  ${\tt FV}_\pi(\alpha,\theta)$ has a continuously evolving diversity process.
  If $\pi\!\in\!\mathcal{M}^a_1$ does not have an $\alpha$-diversity, then ${\tt FV}_\pi(\alpha,\theta)$ has $\alpha$-diversities evolving continuously at all 
  positive times. 
\end{theorem}

The corresponding result for the $(\alpha,0)$- and $(\alpha,\alpha)$-IP evolutions of \cite{Paper1-1} is a consequence of \cite[Theorems 1.2--1.3]{Paper1-1}, in the special cases $\theta\!=\!0$ and $\theta\!=\!\alpha$. Ruggiero et al.\ \cite{RuggWalkFava13} study a modification of a Poisson--Dirichlet diffusion under which diversities evolve as a diffusion.

Shiga \cite[(3.12) and Theorem 3.6]{Shiga1990} gave a Poissonian construction for a large class of measure-valued processes, which includes the labeled infinitely-many-neutral-alleles models of Ethier and Kurtz \cite{EthiKurt87,EthiKurt93}, which we refer to as ${\tt FV}(0,\theta)$. Shiga's approach is closely related to ours: it is based on a Poisson random measure with intensity equal to a Pitman--Yor excursion measure for \BESQ\ processes -- $\BESQ(0)$ in Shiga's setting corresponding to our $\BESQ(-2\alpha)$ -- followed by a de-Poissonization map, as in Definition \ref{def:FV}, that transforms a branching process into a Fleming--Viot process. Our innovation relative to \cite{Shiga1990} is the introduction of scaffolding-and-spindles, as in Figure \ref{fig:scaf-marks}, to account for $\alpha > 0$. Indeed, Shiga's approach could be understood as spindles-without-scaffolding. As discussed in \cite{Paper0,RogeWink20}, this scaffolding can be interpreted as representing a self-similar genealogy among spindles.

In this genealogical interpretation, any pair
$(\widehat{\mathbf{V}},\widehat{\mathbf{X}})$ as in \eqref{eq:clade_law}
comprises a \emph{clade}: the set of all descendants of a common ancestor.
Indeed, we will use the ``clade'' terminology also for similar pairs $(V,X)$ under
$\overline{\nu}_{\perp\rm cld}^{(\alpha)}$. In Section \ref{sec:property3}, we use Shiga's
construction in \cite{Shiga1990} to show that the ${\tt FV}(\alpha,\theta)$, along with 
some additional type labels, can be
projected down to the ${\tt FV}(0,\theta)$ of Ethier and Kurtz by replacing
each clade of evolving atoms with a single atom with evolving mass equal to
the total mass of the clade.\vspace{-0.1cm}

\subsection{Organization of the paper}

The structure of this paper is as follows. 
In Section \ref{sec:prelim}, 
we recall from \cite{Paper0} and \cite{Paper1-1} relevant preliminaries about point measures of spindles and marked stable L\'evy processes, 
which we enrich by further marking of jumps by independent allelic types. 
In Section \ref{sec:alphazero}, we study the special case $\theta=0$ and prove Theorem \ref{thm:sp} in this case, 
as well as Theorem \ref{thm:a0:const}. 
In Section \ref{sec:min_cld}, we study marked excursions of the stable process above its minimum 
and make precise the notion of the $\sigma$-finite measure $\overline{\nu}_{\perp\rm cld}^{(\alpha)}$. 
This is the key to completing the proof of Theorem \ref{thm:sp} and Theorem \ref{thm:at:const} in Section \ref{sec:theta}, 
where we also establish Propositions \ref{prop:scaling:sp}--\ref{prop:branching} and Theorem \ref{thm:mass:sp}. 
In Section \ref{sec:FV} we turn to Fleming--Viot processes and prove Theorem \ref{thm:dP-sp}. In Section \ref{sec:properties}, we further study Fleming--Viot processes, prove Theorems \ref{thm:property1}--\ref{thm:property2} and develop the connections to Shiga \cite{Shiga1990}.

\section{Preliminaries on marked L\'evy processes and point measures}\label{sec:prelim}

In this section, we recall from \cite{Paper0,Paper1-1,Paper1-2} material about point measures of spindles and scaffoldings
used there for the construction of $\nu_{\tt BESQ}^{(-2\alpha)}$-interval-partition evolutions, 
but which we adapt here at the level of point measures to provide allelic types. 
We also show that the point measures and superskewers of Definition \ref{def:superskewer} are well-defined.

\subsection{Spindles: excursions to describe atom size evolutions}\label{sec:skewer}

Let $\cD$ be the Skorokhod space of real-valued c\`adl\`ag functions $f\colon\mathbb{R}\rightarrow\mathbb{R}$ 
and $\Exc$ the subset of non-negative excursions that are continuous except, possibly, at the beginning and at the end (to include incomplete excursions):
\begin{equation}
 \Exc := \left\{f\in\cD\ \middle| \begin{array}{c}
    \displaystyle \exists\ z\in(0,\infty)\textrm{ s.t.\ }\restrict{f}{(-\infty,0)\cup [z,\infty)} = 0,\\[0.2cm]
    \displaystyle f\text{ positive and continuous on }(0,z)
  \end{array}\right\}.\label{eq:cts_exc_space_def}
\end{equation}
We define the \emph{lifetime} and \emph{amplitude}, $\life,A\colon \Exc \to (0,\infty)$ via
\begin{equation}
 \life(f) = \sup\{s\geq 0\colon f(s) > 0\} \quad\mbox{and}\quad
 	A(f) = \sup\{f(s),\, s\in [0,\zeta(f)]\}.
\end{equation}


\begin{lemma}[Equation (13) in \cite{GoinYor03}]\label{lem:BESQ:length}
 Let $B\!\sim\!\BESQ_z(-2\alpha)$ for some $z \!>\! 0$. Then $1/\zeta(B)\sim\GammaDist[1+\alpha,z/2]$.
\end{lemma}


\begin{lemma}[{\cite[Section 3]{PitmYor82}, \cite[Section~2.3]{Paper1-1}}]\label{lem:zetadist}
 For the purpose of the following, for $m>0$ let $H^m\colon \Exc\to [0,\infty]$ denote the first hitting time of level $m$. 
 Then there exists a measure $\mBxcA$ on $\Exc$ such that, for every $m>0$,
 \[
\mBxcA\{f\in\Exc\colon A(f) > m\} = \frac{2\alpha(\alpha+1)}{\Gamma(1-\alpha)}m^{-1-\alpha},
 \]
 and under the probability measure $\mBxcA(\,\cdot\,|\,A(f)> m)$, the restricted canonical process $(Z(y),\,y\in [0,H^m(Z)])$  is a $\BESQ_0(4+2\alpha)$ stopped upon first hitting $m$, independent of $(Z(H^m(Z) + s),\,s\ge0)\sim\BESQ_m(-2\alpha)$. 
 Moreover,
 \[
 \mBxcA\{f\in\Exc\colon \life(f)>y\} = \frac{\alpha}{2^{\alpha}\Gamma(1-\alpha)\Gamma(1+\alpha)}y^{-1-\alpha},\quad y>0.
 \]
\end{lemma}

\begin{lemma}[Corollary 2.10 in \cite{Paper1-1}] 
  There is a unique $\zeta$-disintegration $\nu_{\tt BESQ}^{(-2\alpha)}(\,\cdot\,|\,\zeta)$ of $\nu_{\tt BESQ}^{(-2\alpha)}$ with the scaling property that 
  for all $z\in(0,\infty)$ and $\mathbf{f}\sim\nu_{\tt BESQ}^{(-2\alpha)}(\,\cdot\,|\,\zeta=1)$, we have $(z\mathbf{f}(s/z),\,s\ge 0)\sim\nu_{\tt BESQ}^{(-2\alpha)}(\,\cdot\,|\,\zeta=z)$.
\end{lemma}

\subsection{Scaffolding: jumps describe births and deaths of atoms/allelic types}\label{sec:pointmeas}

As in Section \ref{sec:construction}, consider a Poisson random measure $\mathbf{V}$ on $[0,\infty)\times\Exc\times[0,1]$ with intensity ${\tt Leb}\otimes\mBxcA\otimes\Unif$, denoted by $\PRM[{\tt Leb}\otimes\mBxcA\otimes\Unif]$. 
Let $\mathbf{N} = \varphi(\mathbf{V})$, for the projection $\varphi(V)(dt,df)\!=\!V(dt,df,[0,1])$. Then $\mathbf{N}\sim\PRM[{\tt Leb}\otimes\mBxcA]$, as studied in \cite{Paper0,Paper1-1}. 
%
Let us make precise how $\mathbf{N}$, hence $\mathbf{V}$, induces the associated scaffolding $\mathbf{X}$. By standard mapping of \PRM s,\vspace{-0.1cm}
$$\mathbf{M}:=\sum_{{\rm points}\;(t,f,x)\;{\rm of}\;\mathbf{V}}\delta(t,\zeta(f))=\sum_{{\rm points}\;(t,f)\;{\rm of}\;\mathbf{N}}\delta(t,\zeta(f))\vspace{-0.1cm}$$
is a \PRM\ on $[0,\infty)\times(0,\infty)$ with intensity ${\tt Leb}\otimes\Lambda$ given by Lemma \ref{lem:zetadist} as\vspace{-0.1cm}
$$\Lambda(dz)=\nu_{\tt BESQ}^{(-2\alpha)}\{\zeta\in dz\}=\frac{\alpha(1+\alpha)}{2^\alpha\Gamma(1-\alpha)\Gamma(1+\alpha)}z^{-2-\alpha}dz.\vspace{-0.1cm}$$ 
Since $\Lambda$ is a stable L\'evy measure of index $1\!+\!\alpha\!\in\!(1,2)$, 
the following limit exists a.s.\ uniformly in $t$ on compact subsets of $[0,\infty)$, for $N\!=\!\mathbf{N}$:\vspace{-0.1cm}
 \begin{equation}
  \xiA_N(t)=\lim_{\varepsilon\downto 0}\left(\int_{[0,t]\times\{g\in\Exc\colon\zeta(g) >\varepsilon\}}\life(f)N(du,df) - t\int_{(\varepsilon,\infty)}z\Lambda(dz)\right).\label{eq:scaff_def}   \vspace{-0.1cm}
 \end{equation}
We write $\mathbf{X}\!:=\!\xiA_\mathbf{V}\!:=\!\xiA_\mathbf{N}\!=\!(\xiA_\mathbf{N}(t),t\!\ge\! 0)$ for this L\'evy process. Then all jumps $\Delta\mathbf{X}(t)\!:=\!\mathbf{X}(t)\!-\!\mathbf{X}(t-)$ of $\mathbf{X}$ are positive of size $\Delta\mathbf{X}(t)\!=\!\zeta(f)$, 
corresponding precisely to the points $(t,f)$ of $\mathbf{N}$. 
We reserve the name ``\Stable[1\!+\!\alpha]'' to refer to spectrally positive stable L\'evy processes with L\'evy measure $\Lambda$. This process has Laplace exponent
\begin{equation}\label{eq:scaff:laplace}
 \psi_\alpha(c) = 2^{-\alpha}\Gamma(1+\alpha)^{-1}c^{1+\alpha}.
\end{equation}

For $(\cS,d_{\cS})$ a Borel subset of a complete and separable metric space, we denote by $\cN(\cS)$ the space of boundedly finite point measures on that space. 
Consider $\mathcal{V}:=\mathcal{N}([0,\infty)\times\Exc\times[0,1])$. We equip $\mathcal{V}$ with the $\sigma$-algebra $\Sigma(\mathcal{V})$ generated by evaluation maps. 
The following result follows by standard marking properties of \PRM s; cf. \cite[Proposition 2.15]{Paper1-1} for the corresponding result with only spindle marks, but no allelic type marks.

\begin{lemma}\label{prop:mark_jumps_clade} 
  For $g\in\mathcal{D}$, denote by $\overline{\mu}_g$ the distribution of a point measure $\mathbf{V}_g=\sum_{t\in\mathbb{R}\colon\Delta g(t)>0}\delta(t,f_t,x_t)$ 
  that independently for each time $t$ of a positive jump associates $(f_t,x_t)\sim\nu_{\tt BESQ}^{(-2\alpha)}(\,\cdot\,|\,\zeta=\Delta g(t))\otimes\Unif$. 
  Then $g\mapsto\overline{\mu}_g$ is a stochastic kernel from $\mathcal{D}$ to $\mathcal{V}$. 
  Moreover, for $g=\mathbf{X}\sim\Stable[1+\alpha]$ we have $\mathbf{V}_\mathbf{X}\sim\PRM[{\tt Leb}\otimes\mBxcA\otimes\Unif]$.
\end{lemma}

Now consider $\widehat{\mathbf{V}}\!:=\!\delta(0,\mathbf{f},x)\!+\!\mathbf{V}|_{[0,T]\times\Exc\times[0,1]}\!\sim\!\mathbf{Q}_{b,x}^{(\alpha)}$ 
and $\widehat{\mathbf{X}}\!:=\!\zeta(\mathbf{f})\!+\!\mathbf{X}|_{[0,T]}$ for some $b>0$, $x\in[0,1]$, 
an independent initial spindle $\mathbf{f}\sim{\tt BESQ}_{b}(-2\alpha)$ of lifetime $\zeta(\mathbf{f})$  
and $T:=\inf\{t\ge 0\colon\zeta(\mathbf{f})+\mathbf{X}(t)=0\}$ as in Section \ref{sec:construction}. 
Then $T<\infty$ a.s.\ since $\mathbf{X}$ has zero mean and is spectrally positive, 
and the limit \eqref{eq:scaff_def} exists for $N=\varphi\big(\widehat{\mathbf{V}}\big)$, uniformly on $[0,T]$, with $\widehat{\mathbf{X}}=\xi_{\widehat{\mathbf{V}}}$, 
where we abuse notation and set $\xi_{\widehat{\mathbf{V}}}(t)=0$ for $t\ge T$. 
We refer to $\widehat{\mathbf{V}}$ as a \em clade starting from type $x$ of size $b$. \em Indeed, $\widehat{\mathbf{V}}$ provides the further size evolution of type $x$ expressed by $\mathbf{f}$, and each atom $(t,f_t,x_t)$ of $\widehat{\mathbf{V}}$, for $t>0$, is interpreted as a descendant allelic type $x_t$ 
first created at level $\widehat{\mathbf{X}}(t-)$ with size evolving continuously according to $f_t$ until its extinction at level $\widehat{\mathbf{X}}(t)\!=\!\widehat{\mathbf{X}}(t-)\!+\!\zeta(f_t)$; then the
size of type $x_t$ at level $y$ is $f_t(y-\widehat{\mathbf{X}}(t-))$. Cf. Figure \ref{fig:scaf-marks}.

\subsection{Point measures of clades} 

  

Recall from Section \ref{sec:construction} our notation $\mathcal{M}$ for the set of finite Borel measures on $[0,1]$ and the Prokhorov distance\vspace{-0.1cm}
$$d_\mathcal{M}(\pi,\pi^\prime)=\inf\{\varepsilon\!>\!0\colon\pi(C)\!\le\!\pi^\prime(C^\varepsilon)\!+\!\varepsilon\mbox{ and }\pi^\prime(C)\!\le\!\pi(C^\varepsilon)\!+\!\varepsilon\ \forall C\!\in\!\mathcal{C}\},\vspace{-0.1cm}$$
where $\mathcal{C}=\{C\subseteq[0,1],C\mbox{ closed}\}$ and $C^\varepsilon=\{x\in[0,1]\colon\min_{y\in C}|x-y|\le\varepsilon\}$ is the $\varepsilon$-thickening of $C$. The state space for our processes is the subspace 
$\mathcal{M}^a\subset\mathcal{M}$ of purely atomic measures, which can be written in the form $\pi=\sum_{i\ge 1}b_i\delta(x_i)$ for some $b_i\ge 0$ and distinct $x_i\in[0,1]$, $i\ge 1$. 
Also recall $\mathcal{M}_1^a:=\{\pi\in\mathcal{M}^a\colon\|\pi\|=1\}$ for the subspace of atomic probability measures.

\begin{lemma}[2.13 in \cite{DubinsFreedman1964}]\label{lm:Lusin} The subsets $\mathcal{M}_1^a\subset\mathcal{M}^a\subset\mathcal{M}$ are Borel subsets of the complete and 
  separable metric space $(\mathcal{M},d_\mathcal{M})$.
\end{lemma}

Denote by $(\mathcal{V}\times\mathcal{D})_{\rm fin}^\circ$ the set of all pairs of a point measure $V\in\mathcal{V}$ and a scaffolding $X\in\mathcal{D}$ with the following additional finiteness properties:
\begin{enumerate}\item[(i)] $V([0,\infty)\times\{f\in\Exc\colon\zeta(f)>z\}\times[0,1])<\infty$ for all $z>0$, 
  \item[(ii)] $\sum_{{\rm points}\;(t,f,x)\;{\rm of}\;V}f(y-X(t-))<\infty$ for all $y\in\mathbb{R}$.
\end{enumerate}
Recall  from \eqref{eq:clade_law} our notation \vspace{-0.1cm}
 $\mathbf{Q}_{b,x}^{(\alpha)}$ for the distribution on $\mathcal{V}$ of a clade $\widehat{\mathbf{V}}$ starting from type $x$ of size $b$. 
 Then the pair $(\widehat{\mathbf{V}},\xi_{\widehat{\mathbf{V}}})$ takes values in $(\mathcal{V}\times\mathcal{D})^\circ_{\rm fin}$ almost surely. 
 On our spaces, $\sigma$-algebras generated by evaluation maps are Borel $\sigma$-algebras of complete separable metrics, 
 so we may consider boundedly finite point measures on spaces such as $(\mathcal{V}\times\mathcal{D})^\circ_{\rm fin}$. 
 See \cite[Supplement A]{Paper1-1} for measure-theoretic details, using the framework of \cite{DaleyVereJones1,DaleyVereJones2}.

Now recall Definition \ref{def:superskewer} of the superskewer.

\begin{lemma}\label{lm:lm8}
 For each $y>0$, the map $(V,X)\mapsto \sskewer(y,V,X)$ is measurable from the space $(\mathcal{V}\times\mathcal{D})_{\rm fin}^\circ$ into the Prokhorov space $(\cM^a,d_\cM)$.
\end{lemma}

\begin{proof}
 Let $y>0$. By property (ii), $\sskewer(y,V,X)$ is a finite atomic measure on $[0,1]$, for all $(V,X)\in(\mathcal{V}\times\mathcal{D})^\circ_{\rm fin}$. 
 It therefore suffices to check that evaluation maps are measurable. Indeed, for any Borel set $A\subset [0,1]$,\vspace{-0.1cm}
 \begin{equation*}
  \sskewer(y,V,X)(A) = \int f(y-X(t-))\cf\{x\in A\}V(dt,df,dx).\vspace{-0.1cm}
 \end{equation*}
 is a measurable function of $(V,X)$.
\end{proof}

\begin{proposition}\label{prop:0:len} Consider $\pi=\sum_{i\ge 1}b_i\delta(x_i)\!\in\!\cM^a$ and independent $\mathbf{V}_i\!\sim\!\mathbf{Q}^{(\alpha)}_{b_i,x_i}$, $i\!\ge\! 1$, with 
  associated scaffolding $\mathbf{X}_i\!:=\!\xi_\mathbf{V_i}$. Let $y\!>\!0$. Then a.s.\vspace{-0.1cm}
 \begin{enumerate}[label=(\roman*), ref=(\roman*)]
  \item[(i)] at most finitely many $\mathbf{X}_i$ have height $\zeta_i^+:=\sup\{\mathbf{X}_i(t),t\in[0,\infty)\}>y$;
   \item[(ii)] $\mathbf{F}_\pi=\sum_{i\ge 1\colon b_i>0}\delta(\mathbf{V}_i,\mathbf{X}_i)$ takes values in $\mathcal{N}((\mathcal{V}\times\mathcal{D})_{\rm fin}^\circ)$; 
  \item[(iii)] $\sskewer(y,\mathbf{F}_\pi)$ is well-defined and an $\mathcal{M}^a$-valued random variable; 
 \item[(iv)] $\pi\mapsto\mathbf{Q}^{\alpha,0}_\pi$, where $\mathbf{Q}_\pi^{\alpha,0}$ is the distribution of $\mathbf{F}_\pi$, is a stochastic kernel.\pagebreak 
 \end{enumerate}
\end{proposition}
\begin{proof} (i) This follows from \cite[Lemmas 5.8 and 6.1]{Paper1-1} 
  since the clades $\mathbf{V}_i$ for initial atom sizes $b_i$ here yield projected clades $\mathbf{N}_i=\varphi(\mathbf{V}_i)$
  for the construction of $\nu_{\tt BESQ}^{(-2\alpha)}$-IP-evolutions starting from single intervals $(0,b_i)$ of length $b_i$ in \cite{Paper1-1}, 
  with the same associated scaffolding $\mathbf{X}_i$. 
 
 (ii) Since $(\mathbf{V}_i,\xi_{\mathbf{V}_i})$ is $(\mathcal{V}\times\mathcal{D})^\circ_{\rm fin}$-valued for all $i\ge 1$ almost surely and the number of points of $\mathbf{F}_\pi$ of height 
 $\zeta_i^+>y$ is finite for all $y>0$, the point measure $\mathbf{F}_\pi$ is boundedly finite almost surely. 

 (iii) Lemma \ref{lm:lm8} yields that each $\sskewer(y,\mathbf{V}_i,\mathbf{X}_i)$ is an $\mathcal{M}^a$-valued random variable. 
  By (i), $\sskewer(y,\mathbf{F}_\pi)$ is an a.s.\ finite sum of these.

  (iv) First, there is a measurable enumeration of atoms $\pi\!\mapsto\!((b_i,x_i),i\!\ge\! 1)$; see e.g. \cite[Lemma 1.6]{Kallenberg2017}. 
  Second,  $(b,x)\mapsto\mathbf{Q}_{b,x}^{(\alpha)}$ is measurable, by construction. 
  Third, as the limit \eqref{eq:scaff_def} exists a.s. uniformly for each $N\!=\!\mathbf{N}_i$, $i\!\ge\! 1$, 
  the point measures $\mathbf{N}_i$ a.s. take their values in a space on which $N\!\mapsto\!\xi_N$ is measurable, by \cite[Proposition 2.14]{Paper1-1}.
  The measurability of $\pi\!\mapsto\!\mathbf{Q}_\pi^{\alpha,0}$ follows as a composition and push-forward under measurable maps.
\end{proof}

\begin{corollary}\label{cor:branchprop} 
  Consider any $\pi_j\in\cM^a$, $j\ge 1$, without common atom locations and $\pi:=\sum_{j\ge 1}\pi_j\in\cM^a$.
  Let $\mathbf{F}_j\sim\mathbf{Q}^{\alpha,0}_{\pi_j}$, $j\ge 1$, independent. 
  Then $\sum_{j\ge 1}\mathbf{F}_j\sim\mathbf{Q}^{\alpha,0}_\pi$.
\end{corollary}

\subsection{Continuity properties in point measures of spindles}

We record two results, one about uniform H\"older continuity of the spindles in a stopped \PRM\ and one about the evolution of total mass of the superskewer.

\begin{lemma}[Proposition 6 of \cite{Paper0}]\label{lem:holder}
	 Let $\mathbf{V}\!\sim\!\PRM[{\tt Leb}\otimes\mBxcA\otimes\Unif]$ and $T\in (0,\infty)$ a random time. 
	 Set $\widetilde{\mathbf{V}}= \mathbf{V}|_{[0,T)}$, $\widetilde{\mathbf{N}}\!=\!\varphi(\widetilde{\mathbf{V}})$ 
  and $\widetilde{\mathbf{X}}\!=\!\xi_{\widetilde{\mathbf{N}}}$. Let $\gamma\!\in\!(0,\min(1-\alpha,1/2))$. 
  Then a.s., the spindles of $\widetilde{\mathbf{N}}$ shifted to their birth levels according to $\widetilde{\mathbf{X}}$ can be partitioned into sequences $(g_j^n,j\!\ge\! 1)$, $n\!\ge\! 1$, 
  such that for each level at most one spindle is alive from each sequence, 
  and $|g_j^n(y)-g_j^n(x)|\le D_n|y-x|^\gamma$ for all $x,y\!\in\!\mathbb{R}$, $j\!\ge\!1$, $n\!\ge\! 1$, for a summable sequence $(D_n,n\!\ge\!1)$ of (random) H\"older constants.
\end{lemma}

\begin{proposition}\label{prop:0:mass} Let $\pi\!\in\!\mathcal{M}^a$.
  Then the total mass of the superskewer of $\mathbf{F}\!\sim\!\mathbf{Q}_\pi^{\alpha,0}$ evolves continuously a.s.. 
  Moreover, $\big(\|\sskewer(y,\mathbf{F})\|,y\!\ge\! 0\big)$ is a ${\tt BESQ}_{\|\pi\|}(0)$.  
\end{proposition}
\begin{proof} Let $\widehat{\mathbf{V}}\sim\mathbf{Q}_{b,x}^{(\alpha)}$ and $\widehat{\mathbf{X}}=\xi_{\widehat{\mathbf{V}}}$. 
  Then \cite[Proposition 3.8]{Paper1-1} shows that the skewer of $\widehat{\mathbf{N}}=\varphi(\widehat{\mathbf{V}})$ and hence its total mass evolve continuously a.s.. 
  By \cite[Theorem 1.4]{Paper1-2}, this total mass process is a ${\tt BESQ}_b(0)$.
  But since interval lengths of $\skewer\big(y,\widehat{\mathbf{N}},\widehat{\mathbf{X}}\big)$ are also the atom sizes $f_t(y-\widehat{\mathbf{X}}(t-))$ of $\sskewer\big(y,\widehat{\mathbf{V}},\widehat{\mathbf{X}}\big)$, 
  the total masses coincide. Cf.\ Figure \ref{fig:scaf-marks}.
  For general $\pi\in\mathcal{M}^a$, the statements follow by additivity of total mass processes and of ${\tt BESQ}(0)$.  
\end{proof}

\subsection{Clade statistics}\label{sec:clade_stat}

Let us introduce a scaling operator for $V\in\mathcal{V}$.
For $c>0$, we define 
\begin{equation}\label{eq:min-cld:xform_def}
c\scaleHA V := \int \Dirac{c^{1+\alpha}t,c\scaleB f,x}V(dt,df,dx), 
\end{equation}
where $c\scaleB f := \big(cf(y/c),\,y\in\BR\big)$.
Recall that for a pair $(V,X)\in\mathcal{V}\times\mathcal{D}$ with $X=\xi_V$, the heights of the scaffolding $X$ correspond to times in the superskewer process $\sskewer(\,\cdot,V,X)$. 
In light of this, we define the \emph{lifetime} of a clade $V\in\mathcal{V}$ for which $\xiA_V$ exists as
\begin{equation}\label{eq:min-cld:lifetime_def}
   \life^+(V) := \life^+(\xi_V):=\sup_{t\in[0,\len(V)]}  \xiA_V(t),
\end{equation}
where
\begin{equation}
   \len(V):=\inf\left\{t\!>\!0\colon V([t,\infty)\!\times\!\Exc\!\times\![0,1])=0\right\}. 
\end{equation}
Note that we have the following relations: 
\begin{equation}\label{eq:min-cld:xform_property}
\life^+(c \scaleHA V) = c\life^+(V) \quad \mbox{and}\quad
\len(c \scaleHA V) = c^{1+\alpha}\len(V).
\end{equation}
\begin{lemma}\label{lm:scaling1}
  Let $\widehat{\mathbf{V}}\sim\mathbf{Q}^{(\alpha)}_{b,x}$. 
  Then $c\scaleHA\widehat{\mathbf{V}}\sim\mathbf{Q}^{(\alpha)}_{cb,x}$  for all $c>0$.
\end{lemma}
\begin{proof} This follows directly from the scale invariance of ${\tt BESQ}(-2\alpha)$ and $\Stable[1\!+\!\alpha]$. Specifically, 
  $\mathbf{f}\!\sim\!{\tt BESQ}_b(-2\alpha)$ yields $c\!\,\scaleB\!\mathbf{f}\!=\!(c\mathbf{f}(y/c),y\!\ge\!0)\!\sim\!{\tt BESQ}_{cb}(-2\alpha)$, 
  and this is inherited by $\nu_{\tt BESQ}^{(-2\alpha)}$ in the form
  $\nu_{\tt BESQ}^{(-2\alpha)}(c\scaleB\,\cdot\,)=c^{-1-\alpha}\nu_{\tt BESQ}^{(-2\alpha)}$.  See e.g.\ \cite[Lemma 2.9]{Paper1-1}.
\end{proof}


\begin{proposition}\label{prop:clade:trans}
Consider a clade $\widehat{\mathbf{V}}\sim{\mathbf{Q}}_{b,x}^{(\alpha)}$ starting from type $x\in[0,1]$ of size $b>0$, 
set $\widehat{\bX}= \xiA_{\widehat{\mathbf{V}}}$ and $\pi^y =\sskewer(y, \widehat{\mathbf{V}},\widehat{\bX})$, $y\ge 0$. 
Then the lifetime $\zeta^+(\widehat{\mathbf{V}})$ of the clade has distribution $\mathtt{InverseGamma}(1, b/2)$, i.e.\ 
		\[\mathbf{P} \{\zeta^+(\widehat{\mathbf{V}})>y\}= \mathbf{P} \{\pi^y \ne 0\} = 1- e^{-b/2y}, \quad y\ge 0.\]
For any $y\ge 0$, the distribution of $\pi^y$ is $Q_{b,x,1/2y}^{(\alpha)}$, as defined in \eqref{eq:Qbxralpha}.
\end{proposition}
\begin{proof}
	In the notation of \cite{Paper1-2}, $\varphi(\widehat{\mathbf{V}})\sim\nu_{{\rm cld}}(\cdot\mid m^0=b)$. 
	Then we obtain the law of its lifetime from \cite[Proposition~3.4]{Paper1-2}. Now fix $y\ge 0$. Recall that $\widehat{\mathbf{V}}$ has a 
	point $(0,f_0, x)$ with $f_0\sim \BESQ_{b}(-2\alpha)$. The mass of the leftmost spindle at level $y$ is given by $m^y := f_r (y- \widehat{\bX}(r-))$, where $r:= \inf \{ t\ge 0\!:$\linebreak $(t,f_t ,x_t) \text{ is a point of }\widehat{\mathbf{V}}, f_t(y- \widehat{\bX}(t-))>0\}$.
	By \cite[Lemma~3.5]{Paper1-2}, the conditional distribution of $m^y$ given $\{\pi^y \ne 0\}$ is the law of $L^{(\alpha)}_{b,1/2y}$.  
	The masses of the other blocks are conditionally independent of the leftmost block, given $\{ \pi^y\ne 0 \}$, and their distribution can be read from \cite[Proof of Proposition~3.4]{Paper1-2}: it is described by $(G A_i, i\ge 1)$, where $G\sim \mathtt{Gamma}(\alpha, 1/2y)$ and $(A_i, i\ge 1)\sim \mathtt{PD}(\alpha,\alpha)$ are independent. 
	
	Recall that the spindle $f_0$ has allelic type $x$ and the others have i.i.d.\@ \Unif\@ types. 
	Let $U_i$, $i\ge 0$, be independent \Unif-distributed, independent of everything else. Then the conditional distribution of $\pi^y$ given $\{\pi^y \ne 0\}$ is the law of $\lambda^y+ G \overline{\Pi}$, where  	
	\[
	\lambda^y:= \mathbf{1}\{\zeta(f_0) > y\} m^y \delta(x) + \mathbf{1}\{\zeta(f_0) \le y\} m^y \delta(U_0) 
	\]
	and $\overline{\Pi} := \sum_{i\ge 1} A_i \delta(U_i)$. Here $\lambda^y$ and $G\overline{\Pi}$ represent respectively the contribution of the leftmost spindle at level $y$ and the others.  
	We have $\overline{\Pi}\sim \mathtt{PDRM}(\alpha,\alpha)$ by definition. 	
	It remains to study the law of $\lambda^y$. Using the \cite[Equation (A.6) and (A.7)]{Paper1-2}, we have the identity
	\[
	\frac{\mathbf{P} \big(\zeta(f_0)>y , m^y \in dc \mid \zeta^+(\widehat{\mathbf{V}})>y \big)}
		 {\mathbf{P} \big(y\in [\zeta(f_0), \zeta^+(\widehat{\mathbf{V}})), m^y \in dc\mid  \zeta^+(\widehat{\mathbf{V}})>y \big)}
		 = \frac{p_{b,1/2y}^{(\alpha)}(c)}{1-p_{b,1/2y}^{(\alpha)}(c)}, \quad c>0. 
	\]
	This implies that the conditional distribution of $\mathbf{1}\{\zeta(f_0)\!>\!y\}$ given $\{m^y \!=\!c\}$ is Bernoulli with parameter $p_{b,1/2y}^{(\alpha)}(c)$. Having already obtained the law of $m^y$, we can therefore identify the law of $\lambda^y$. 
	
	Summarizing, we deduce that the distribution of $\pi^y$ is $Q_{b,x,1/2y}^{(\alpha)}$. 
\end{proof}

\subsection{Markov-like properties of point measures of spindles and clades}

Let $V\in\mathcal{V}$ and $X\in\mathcal{D}$ with $X=\xi_V$. Fix $y\ge 0$. If an atom $(t,f_t,x_t)$ of $V$ satisfies 
$y\in \big( X(t-), X(t)\big)$, i.e.\@ the spindle $f_t$ crosses level $y$, then we define $\hat f^y_{t}$ and $\check f^y_{t}$ to be its broken components split about that crossing. See Figure~\ref{fig:min_cld_y}.

\begin{figure}
	\centering
	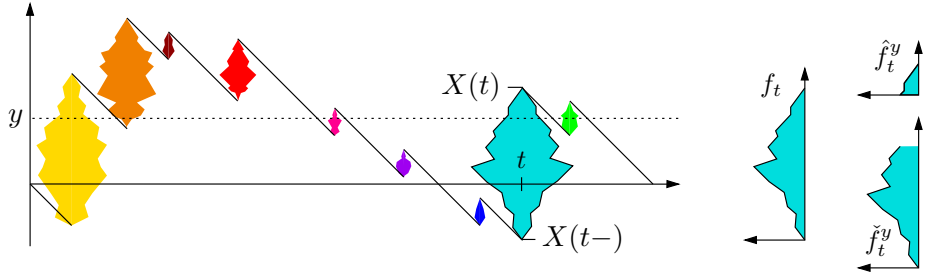
	\caption{Spindle $f_t$ of $V$ at time $t$ split about level $y$ of the scaffolding $X$.\label{fig:min_cld_y}}
\end{figure}

For a point measure $V\in\mathcal{V}$ and an interval $[a,b]$, the \emph{shifted restriction} $\shiftrestrict{V}{[a,b]}\in\mathcal{V}$ is defined as the restriction of $V$ to $[a,b]\times\Exc\times[0,1]$, translated so that it is supported on $[0,b-a]\times\Exc\times[0,1]$. The shifted restriction of $X\in\mathcal{D}$, denoted by $\shiftrestrict{X}{[a,b]}\in\mathcal{D}$, is defined correspondingly:
\begin{equation}
\begin{split}
 \ShiftRestrict{V}{[a,b]}([c,d]\!\times\! A) &= V\big(([c\!+\!a,d\!+\!a]\!\cap\![a,b])\! \times\! A\big),\ \ A\!\in\!\Sigma(\Exc\!\times\![0,1]),\\
 \ShiftRestrict{X}{[a,b]}(t) &= \cf\{t\in [0,b-a]\}X(t+a),\quad t\in\BR.
\end{split}
\end{equation}

\begin{proposition}[Mid-spindle Markov property]\label{MSMP}Fix $y>0$. Let $\mathbf{V}\sim{\tt PRM}\big({\tt Leb}\otimes\nu_{\tt BESQ}^{(-2\alpha)}\otimes\Unif\big)$, 
  $\mathbf{X}=\xi_{\mathbf{V}}$ and $T=T^{\ge y}=\inf\{t\!\ge\! 0\colon\mathbf{X}(t)\!\ge\! y\}$. \linebreak
  Then $\big(\ShiftRestrict{\mathbf{V}}{(T,\infty)},\hat f_T^y,x_T\!\big)$ is conditionally independent of $\big(\ShiftRestrict{\mathbf{V}}{[0,T)},\check{f}^y_T,x_T\!\big)$
  given the mass $f_T(y\!-\!\mathbf{X}(T-))\!=\!a$ and type $x_T\!=\!x$ at level $y$, 
  with regular conditional distribution ${\tt PRM}\big({\tt Leb}\otimes\nu_{\tt BESQ}^{(-2\alpha)}\otimes\Unif\big)\otimes{\tt BESQ}_a(-2\alpha)\otimes\delta_x$.
\end{proposition}
\begin{proof} Let $\mathbf{N}=\varphi(\mathbf{V})$. The Mid-spindle Markov property of $\mathbf{N}$ at $T^{\ge y}$ of \cite[Lemma 4.13]{Paper1-1} states 
  that $\big(\ShiftRestrict{\mathbf{N}}{(t,\infty)},\hat f_T^y\big)$ is conditionally independent of $\big(\ShiftRestrict{\mathbf{N}}{[0,T)},\check{f}^y_T\big)$
  given $f_T(y-\mathbf{X}(T-))=a$, 
  with regular conditional distribution ${\tt PRM}\big({\tt Leb}\otimes\nu_{\tt BESQ}^{(-2\alpha)}\big)\otimes{\tt BESQ}_a(-2\alpha)$. Hence, the proof is completed by
  independently marking all spindles by independent $\Unif$ allelic types, noting in particular that $x_T$ is independent of everything else.
\end{proof}

\begin{figure}
 \centering
 \scalebox{.85}{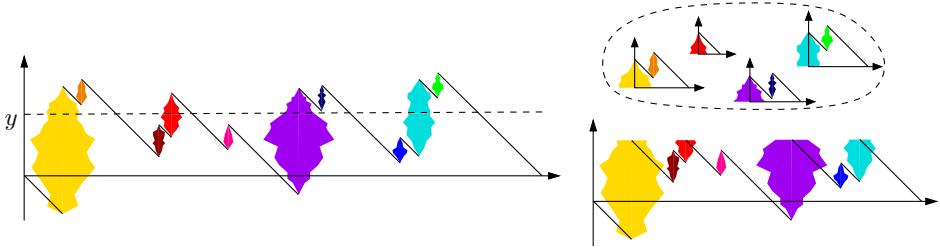}
 \caption{Left: A scaffolding and spindles $(V,X)$. Right bottom: $\textsc{cutoff}^{\le y}(V,X)$ on its natural scaffolding; cf.\  \eqref{eq:skewerscaf}. Right top: $G^{\ge y}_0(V,X)$. This is a point measure of clades, informally represented here as a ``cloud'' of clades, inside the dashed oval.\label{fig:cutoff}}
\end{figure}
For any level $y\ge 0$ and $(V,X)\in\mathcal{V}\times\mathcal{D}$, 
we define the \emph{lower cutoff process} 
\begin{align}
	\label{eq:cutoff}& \\[-13pt]\notag
 {\textsc{cutoff}}^{\le y}(V,X) &:=\!\! \sum_{\text{points }(t,f_t,x_t)\text{ of }V} \!\!\left(\!\!\!\begin{array}{r@{\,}l}
	\cf\big\{y\!\in\! \big( X(t-), X(t)\big)\big\}&\DiracBig{\Theta(t),\check f^y_t,x_t}\\[3pt]
	+\; \cf\big\{X(t) \leq y\big\}&\DiracBig{\Theta(t),f_t,x_t}
   \end{array}\!\!\!\right)\!,
\end{align}
  where $\Theta(t):= {\tt Leb}\{u\le t\colon X(u)\le y\}$. This takes spindles $f_t$ that $X$ places below level $y$ and the lower parts $\check{f}_t^y$ of spindles 
  that $X$ places across level $y$, and moves them (in time) so that intervals of excursions of $X$ above level $y$ are cut out. We also define the \emph{upper point measure of clades}
\begin{equation}\label{eq:upper-clade}
  G^{\ge y}_0(V,X)
  := \sum_{(a,b)} \delta\left(\ShiftRestrict{V}{(a,b)}+\cf\big\{X(a)>y\big\}\delta(0,\hat f^y_a,x_a)\right),\vspace{-0.1cm}
\end{equation}
where the sum is over all excursions of $X$ above level $y$, recording for each excursion a clade consisting of the upper part of any initial spindle that $X$ places across level $y$ and
the shifted point measure of spindles above level $y$.
The cutoff process and upper point measure are illustrated in Figure \ref{fig:cutoff}. 

\begin{proposition}[Markov-like property]\label{prop:a0:markovlike} 
  Let $\widehat{\mathbf{V}}\!\sim\!\mathbf{Q}_{b,x}^{(\alpha)}$, $\widehat{\mathbf{X}}\!=\!\xi_{\widehat{\mathbf{V}}}$ and 
  $y\!\ge\!0$.  
  Then $G^{\ge y}_0(\widehat{\mathbf{V}},\widehat{\mathbf{X}})$ 
  is conditionally independent of ${\textsc{cutoff}}^{\le y}(\widehat{\mathbf{V}},\widehat{\mathbf{X}})$ 
  given $\sskewer(y,\widehat{\mathbf{V}},\widehat{\mathbf{X}})\!=\!\pi$ 
  and has regular conditional distribution $\mathbf{Q}_{\pi}^{\alpha,0}$.  
\end{proposition}
\begin{proof} We define variants of ${\textsc{cutoff}}^{\le y}(V,X)$ and $G^{\ge y}_0$ without allelic types,\vspace{-0.2cm}
  \begin{align*}
    {\textsc{cutoff}}^{\le y}(N,X) 
      &:= \sum_{\text{points }(t,f_t)\text{ of }N} \left(\!\!\!\begin{array}{r@{\,}l}
	\cf\big\{y\!\in\! \big( X(t-), X(t)\big)\big\}&\DiracBig{\Theta(t),\check f^y_t}\\[3pt]
	+\; \cf\big\{X(t) \leq y\big\}&\DiracBig{\Theta(t),f_t}\end{array}\!\!\!\right)\!,\\
    H^{\ge y}_0(N,X)
       &:= \sum_{(a,b)}\delta\!\left(\!\ShiftRestrict{N}{(a,b)}+\cf\big\{X(a)\!>\!y\big\}\delta(0,\hat f^y_a)\right)\!.
  \end{align*} 
  Then the Markov-like property of $\widehat{\mathbf{N}}:=\varphi(\widehat{\mathbf{V}})$ at level $y$, \cite[Proposition 5.9]{Paper1-1}, yields that
  $\widehat{\mathbf{H}}^{\ge y}_0:=H^{\ge y}_0(\widehat{\mathbf{N}},\widehat{\mathbf{X}})$ is conditionally independent of the completion
  $\overline{\mathcal{F}}^y$ of $\sigma\left({\textsc{cutoff}}^{\le y}(\widehat{\mathbf{N}},\widehat{\mathbf{X}})\right)$, given
  $\skewer\big(y,\widehat{\mathbf{N}},\widehat{\mathbf{X}}\big)$, and a regular conditional distribution is provided by the distribution of\vspace{-0.1cm} 
  $$\sum_{\text{intervals }U\text{ of }\skewer(y,\widehat{\mathbf{N}},\widehat{\mathbf{X}})}\delta(\mathbf{N}_U),$$ 
  where the $\mathbf{N}_U$ are conditionally independent. Moreover, comparing \cite[Definition 4.1]{Paper1-1} with the definition of 
  $\mathbf{Q}_{b,x}^{(\alpha)}$ above Definition \ref{def:superskewer} we see that 
  $\mathbf{N}_U\sim\mathbf{Q}_{b,x}^{(\alpha)}(\varphi\in\cdot)$ for $b={\tt Leb}(U)$ and any $x\in[0,1]$. 

  On the event $\big\{\skewer\big(y,\widehat{\mathbf{N}},\widehat{\mathbf{X}}\big)=\emptyset\big\}
				=\big\{\sskewer\big(y,\widehat{\mathbf{V}},\widehat{\mathbf{X}}\big)=0\big\}$, the conditional independence 
  and distribution trivially extend to the completion $\overline{\mathcal{G}}^y$ of the bigger $\sigma$-algebra 
  $\sigma\big({\textsc{cutoff}}^{\le y}(\widehat{\mathbf{V}},\widehat{\mathbf{X}})\big)$. See e.g.\ \cite[Equation (5.8)]{Paper1-1}.

  On the event $\big\{\skewer\big(y,\widehat{\mathbf{N}},\widehat{\mathbf{X}}\big)\neq\emptyset\big\}
				=\big\{\sskewer\big(y,\widehat{\mathbf{V}},\widehat{\mathbf{X}}\big)\neq 0\big\}$, 
  properties of \Stable[1+\alpha] L\'evy processes (see e.g. \cite[Proposition A.3]{Paper1-1}) yield 
  that there are, up to a null set, infinitely many points $(t_i,f_i,x_i)$ of $\widehat{\mathbf{V}}$ that correspond to jumps of $\widehat{\mathbf{X}}$ across level $y$,
  with non-trivial lower and upper parts $\check{f}_i^y$ and $\hat{f}_i^y$, $i\ge 1$.   
  Let $b_i:=f_i(y-\widehat{\mathbf{X}}(t_i-))$, $i\ge 1$. 
  
  We recall that the conditional distribution of $\widehat{\mathbf{V}}$ given $\widehat{\mathbf{N}}$ is given by the
  marking kernel that independently marks each point $(t,f_t)$ of $\widehat{\mathbf{N}}$ by an independent allelic type $x_t$. By elementary arguments 
  based on the chain rule for conditional independence (e.g. \cite[Proposition 6.8]{Kallenberg}), this yields the conditional independence and distribution as statements
  conditionally given $\skewer\big(y,\widehat{\mathbf{N}},\widehat{\mathbf{X}}\big)$ with intervals of lengths $b_i$ and allelic types $x_i$, $i\ge 1$. But since 
  this conditional distribution depends on $\skewer\big(y,\widehat{\mathbf{N}},\widehat{\mathbf{X}}\big)$ and $(b_i,x_i)_{i\ge 1}$ only via 
  $\pi=\sum_{i\ge 1}b_i\delta(x_i)$, 
  the conditional independence and distribution also hold conditionally given $\sskewer\big(y,\widehat{\mathbf{V}},\widehat{\mathbf{X}}\big)=\pi$.  
\end{proof}

\section{Self-similar $(\alpha,0)$-superprocesses}\label{sec:alphazero} 

$\!\!\!$The construction of ${\tt SSSP}(\alpha,0)$ and the proof of its properties 
follows the corresponding steps for $\nu_{\tt BESQ}^{(-2\alpha)}$-interval partition evolutions in \cite{Paper1-1,Paper1-2}. 
For both processes, a substantial part of the technical work is carried out at the level of point measures in the framework of marked L\'evy processes, as collected in Section \ref{sec:prelim}.

In the present section, we use the superskewer to study the associated $\mathcal{M}^a$-valued processes.
This also bears some similarities with the approach in \cite{Paper1-1,Paper1-2}, but there are significant variations for several reasons. 
On the one hand, (natural) topologies on the state spaces involved are rather different. 
On the other hand, L\'evy processes and interval partitions have an intrinsic left-to-right order, which is not captured in the state space $\mathcal{M}^a$. 

\subsection{Path-continuity}

In addition to $d_{\cM}$, we also consider the metric $d_{\rm TV}$ on $\cM^a$ induced by the total variation norm: for any $\pi,\pi'\in \cM^a$, 
\[
d_{\rm TV} (\pi, \pi'):= \sup_{B\in \cB([0,1])}  (|\pi(B) -\pi'(B)| + |\pi(B^c) -\pi'(B^c)|), 
\]
where $\cB([0,1])$ denotes the Borel sets on $[0,1]$. 
Since the Prokhorov metric $d_{\cM}$, unlike $d_{\rm TV}$, induces a separable topology on $\cM^a$, 
we use $d_\cM$ where separability matters. 
But we know from \cite[p.60--61]{Ghosh2003} that the metric $d_{\rm TV}$ is Borel measurable on the metric space $(\cM^a, d_{\cM})$. 
It is therefore meaningful to study path-continuity for the stronger topology induced by $d_{\rm TV}$. 


\begin{proposition}\label{prop:PRM:cont} 
		 Let $\mathbf{V}\!\sim\!\PRM[{\tt Leb}\otimes\mBxcA\otimes\Unif]$ and $T\in (0,\infty)$ be a random time. 
	Set $\widetilde{\mathbf{V}}= \mathbf{V}|_{[0,T)\times\mathcal{E}\times [0,1]}$
	and $\widetilde{\mathbf{X}}\!=\!\xi_{\widetilde{\mathbf{V}}}$.  
	Then the superskewer process $\Big(\sskewer\big(y,\widetilde{\mathbf{V}},\widetilde{\mathbf{X}}\big),y\!\in\!\mathbb{R}\Big)$ of the stopped \PRM\ is a.s.\ H\"older-$\gamma$ in $(\cM^a, d_{\rm TV})$ for every $\gamma\in (0,\min(1-\alpha,1/2))$.
\end{proposition}
\begin{proof}
	Fix $\gamma \in (0,\min(1-\alpha,1/2))$. 
	Applying Lemma~\ref{lem:holder}, we partition the marked spindles of $\widetilde{\mathbf{V}}$ into sequences $((g_j^n,x_j^n), j\ge 1)$ for $n\ge 1$. 
	For every $y\in \mathbb{R}$, let $\pi^n(y):= \sum_{j\ge 1} g_j^n(y) \delta (x_j^n)$. Since $g_j^n(y)=0$ for all but at most one $j\ge 1$, 
	this captures a single type $x_j^n$ for each $n\ge 1$. Furthermore, there is the identity
	\[
	\sskewer(y, \widetilde{\mathbf{V}},\widetilde{\mathbf{X}}) = \sum_{n\ge 1}\pi^n(y), \quad y\in \mathbb{R}. 
	\]

	Fix $y<z$. Let $A:=\{ n\ge 1\colon \pi^n(r) \ne 0 \text{ for all } r\in [y,z]\}$. That is, $A$
	is the set of indices $n$ for which $\widetilde{\mathbf{X}}$ places a single spindle in the sequence $(g^n_j,j\!\ge\!1)$
	across both levels $y$ and $z$, i.e.\ the type $x_j^n$ corresponding to this spindle has positive mass during the entire interval $[y, z]$. 
	Then we have 
	\begin{equation}\label{eq:Holder}
	\begin{split}
	&	d_{\rm TV} \left(\sskewer(y,\widetilde{\mathbf{V}},\widetilde{\mathbf{X}}) ,\sskewer(z,\widetilde{\mathbf{V}},\widetilde{\mathbf{X}})\right)\\
	&\qquad \le \sum_{n\in A} \big| \|\pi^n(y)\| -\|\pi^n(z)\| \big| 
	+ \sum_{n\not\in A}\|\pi^n(y)\| +\sum_{n\not\in A} \|\pi^n(z)\|.
	\end{split}
	\end{equation}
	
	By Lemma~\ref{lem:holder}, a.s.\ each process $(\|\pi^n(r)\|, r\ge 0)$ is $\gamma$-H\"older continuous with H\"older constant $D_n$. 
	Then we have $\big|\|\pi^n(y)\| -\|\pi^n(z)\|\big| \le D_n |y-z|^{\gamma}$ for each $n\ge 1$. Moreover, for $n\not\in A$, there exists some level $x\in [y,z]$ such that $\pi^n(x)=0$. It follows that 
	\[
	\forall n\not\in A,~ \max( \|\pi^n(y)\|, \|\pi^n(z)\| ) \le D_n |y-z|^{\gamma}. 
	\]
	So we deduce from \eqref{eq:Holder} that 
	$
	d_{\rm TV} (\pi^y ,\pi^z)\le 2 \sum_{n\ge 1} D_n |y-z|^{\gamma}. 
	$
	Since $(D_n, n\ge 1)$ is summable, we deduce the $\gamma$-H\"older continuity. 
\end{proof}

\begin{corollary}\label{cor:PRM:cont}
	For $b>0$ and $x\in [0,1]$, let $\widehat{\mathbf{V}}\!\sim\!\bQ^{(\alpha)}_{b,x}$. 
	Set $\widehat{\mathbf{X}}\!=\!\xi_{\widehat{\mathbf{V}}}$.  
	Then the superskewer process $\Big(\sskewer\big(y,\widehat{\mathbf{V}},\widehat{\mathbf{X}}\big),y\!\ge\! 0\Big)$ is a.s.\ H\"older-$\gamma$ in $(\cM^a, d_{\rm TV})$ for every $\gamma\in (0,\min(1-\alpha, 1/2))$.
\end{corollary}
\begin{proof}
	Let  $\mathbf{V}\!\sim\!\PRM[{\tt Leb}\otimes\mBxcA\otimes\Unif]$ and  $\mathbf{f}\sim \BESQ_{b}(-2\alpha)$ be independent. Set $\mathbf{X}= \xiA_{\mathbf{V}}$ and  
	\[
	\widetilde{\mathbf{V}}= \mathbf{V}|_{[0,T)\times\mathcal{E}\times [0,1]}, 
	 \quad\text{where}~ T:=\inf\{t\ge 0\colon \mathbf{X}(t) \le -\zeta(\mathbf{f})\}. 
	\]
	By the definition of $\bQ^{(\alpha)}_{b,x}$, we can write 	$\widehat{\mathbf{V}}= \widetilde{\mathbf{V}}+ \delta(0,\mathbf{f},x)$. 
	Then we have the identity, for every $y\ge 0$, 
	\[
	\sskewer(y, \widehat{\mathbf{V}},\widehat{\mathbf{X}}) 
	= \sskewer(y-\zeta(\mathbf{f}), \widetilde{\mathbf{V}},\xiA_{\widetilde{\mathbf{V}}}) 
		+\mathbf{1}\{y\le \zeta(\mathbf{f})\} \mathbf{f}(y) \delta(x). 
	\]
	Applying Proposition~\ref{prop:PRM:cont} to $\widetilde{\mathbf{V}}$ and noting e.g.\ by \cite[Corollary 34]{Paper0} that the $\BESQ(-2\alpha)$ process $\mathbf{f}$ is H\"older-$\gamma$ for every $\gamma\in (0,1/2)$, 
	we deduce that $\big(\sskewer(\widehat{\mathbf{V}},\widehat{\mathbf{X}}),y\!\ge\! 0\big)$ is the sum of two processes that are H\"older-$\gamma$ in $(\cM^a,d_{\rm TV})$, for any $\gamma <\min(1-\alpha, 1/2)$. 
	This completes the proof. 
\end{proof}

\begin{proposition}\label{prop:0:continuity}
	Let $\pi\in \mathcal{M}^a$ and $\bF\sim \bQ^{\alpha,0}_{\pi}$.  
	Then the superskewer process $\big(\sskewer(y,\bF),\;y\!\ge\! 0\big)$ a.s.\@ 
	has continuous paths in $(\cM^a, d_{\rm TV})$ starting from $\sskewer(0,\bF)\!=\!\pi$.  
	Moreover, it is a.s.\ H\"older-$\gamma$ in $(\cM^a, d_{\rm TV})$ for every $\gamma\in (0,\min\{1-\alpha,1/2\})$ for all $y\in(0,\infty)$. 
\end{proposition}

Since the topology generated by $d_{\rm TV}$ is stronger than the topology generated by $d_{\cM}$, Proposition~\ref{prop:0:continuity} implies that the process also has 
continuous paths in $(\cM^a,d_{\cM})$. 

\begin{proof}
	Our proof is an adaptation of arguments in the proof of \cite[Theorem~1.4]{Paper1-1}. 
	For completeness, let us sketch it here. 

	For every $z>0$, by Proposition~\ref{prop:0:len}, a.s.\ the post-$z$ process is equal to the sum of a finite number of processes 
	$(\sskewer(y, \widehat{\mathbf{V}},\widehat{\mathbf{X}}), y\ge z)$, with  $\widehat{\mathbf{V}}\sim\mathbf{Q}_{b,x}^{(\alpha)}$
	and $\widehat{\mathbf{X}}=\xi_{\widehat{\mathbf{V}}}$. 
	By Corollary~\ref{cor:PRM:cont}, each of these superskewer processes is a.s.\ $\gamma$-H\"older continuous. 
	Therefore, the sum is also a.s.\ $\gamma$-H\"older continuous. 
	Since $z>0$ is arbitrary, it only remains to prove the continuity at level $0$. 
	
		Write $\pi = \sum_{i\ge 1} b_i \delta(x_i)$. Fix $\epsilon>0$ and take $k\in \BN$ large enough such that $b_0:= 	\|\pi\|- \Big(\sum_{i=1}^k b_i\Big) <\epsilon$. 
	Recall the definition of $\bQ^{\alpha,0}_{\pi}$, the law of $\bF$, then we can write \vspace{-.2cm}
	\[
	\sskewer(y, \mathbf{F})=\lambda^y + \sum_{i=1}^k f_i (y) \delta(x_i) ,\quad y\ge 0, 
	\]
	where $(f_i, i\le k)$ are independent $\BESQ(-2\alpha)$ starting from $b_i$ respectively. 
	
	By the continuity of the total mass process of Proposition \ref{prop:0:mass}, and by the continuity of $\BESQ(-2\alpha)$, there exists $h>0$ such that for every $y<h$: 
	\begin{enumerate}
		\item  $\big|\|\pi\|-( \|\lambda^y\|+ \sum_{i=1}^k |f_i(y)|) \big|<\epsilon$; 
		\item for $i\le k$, $|f_i(y) -b_i|<\epsilon/k$. 
	\end{enumerate}
It follows from the first inequality and the choice of $k$ that \vspace{-.1cm}
\[
\|\lambda^y\|\le \sum_{i=1}^k  |f_i(y) -b_i|  + \epsilon + b_0 < 3\epsilon. \vspace{-.1cm}
\]
We conclude that \vspace{-.2cm}
\[
d_{\rm TV} (\pi, \sskewer(y, \mathbf{F}))\le \sum_{i=1}^k  |f_i(y) -b_i| + \|\lambda^y \| + b_0< 5\epsilon. 
\]
This proves the continuity at level $0$. 	
\end{proof}

\subsection{Markov property}\label{subsec:0:Markov}

Let $\pi\in\cM^a$. 
In Proposition \ref{prop:0:continuity}, we saw that the superskewer process is $\mathbf{Q}^{\alpha,0}_\pi$-a.s.\ $d_{\rm TV}$-path-continuous.   
In particular, it has a distribution on the space $\mathcal{C}([0,\infty),\cM^a)$ of continuous paths in the separable metric space 
$(\cM^a,d_\cM)$, since the total variation topology is stronger than the weak topology induced by the Prokhorov metric $d_\cM$.
We denote this distribution on $\mathcal{C}([0,\infty),\cM^a)$ by $\mathbb{Q}_\pi^{\alpha,0}$. 

\begin{lemma}\label{lm:bbq:kernel} The map $\pi\mapsto\mathbb{Q}_\pi^{\alpha,0}$ is a stochastic kernel.
\end{lemma}
\begin{proof} Consider the subset $(\mathcal{V}\times\mathcal{D})^*\subset(\mathcal{V}\times\mathcal{D})^\circ$ of pairs $(V,X)$ with
  \begin{equation}\label{eq:sskewerP}  \sskewerP(V,X):=(\sskewer(y,V,X),\,y\!\ge\! 0)\in\mathcal{C}([0,\infty),\cM^a). 
  \end{equation}
  Recall that the Borel $\sigma$-algebra on $\mathcal{C}([0,\infty),\cM^a)$
  generated by the topology of uniform convergence on compact subsets of $[0,\infty)$ is also generated by the evaluation maps
  \cite[Theorem 12.5]{Kallenberg}. Hence, Lemma \ref{lm:lm8} implies the measurability of $\sskewerP$ as a function from
  $(\mathcal{V}\times\mathcal{D})^*$ to $\mathcal{C}([0,\infty),\cM^a)$. By Proposition \ref{prop:0:continuity}, 
  $\mathbf{Q}_{\pi}^{\alpha,0}((\mathcal{V}\times\mathcal{D})^*)=1$. 
  Therefore, this lemma follows from the kernel property of 
  $\pi\mapsto\mathbf{Q}^{\alpha,0}_\pi$ established in Proposition \ref{prop:0:len}(iv). 
\end{proof}

Let us now state the Markov property of the superskewer process in terms of the kernel $\pi\mapsto\mathbb{Q}_\pi^{\alpha,0}$. 
We also use notation $\mathbb{Q}_\mu^{\alpha,0}:=\int_{\cM^a}\mathbb{Q}_\pi^{\alpha,0}\mu(d\pi)$ 
for Borel probability measures $\mu$ on $(\cM^a,d_\cM)$.

\begin{proposition}[Markov property] \label{prop:0-MP}
	Let $\mu$ be a probability measure on $\cM^a$ and $(\pi^z,\,z\ge 0)\sim \BQ^{\alpha,0}_{\mu}$. 
	For any $y\ge 0$, the process $(\pi^{y+r},\,r\ge 0)$ is conditionally independent of $(\pi^z,\,z\le y)$ given $\pi^y=\pi$ and has 
	regular conditional distribution $\BQ^{\alpha,0}_{\pi}$. 
\end{proposition}
\begin{proof} Fix $b\!>\!0$, $x\!\in\![0,1]$, let $\widehat{\mathbf{V}}\!\sim\!\mathbf{Q}^{(\alpha)}_{b,x}$ and $\widehat{\mathbf{X}}\!=\!\xi_{\widehat{\mathbf{V}}}$. 
  By the Markov-like property of Proposition \ref{prop:a0:markovlike}, 
  the upper point measure of clades $G^{\ge y}_0(\widehat{\mathbf{V}},\widehat{\mathbf{X}})$ 
  is conditionally independent of the lower cutoff process $V:={\textsc{cutoff}}^{\le y}(\widehat{\mathbf{V}},\widehat{\mathbf{X}})$ 
  given $\sskewer(y,\widehat{\mathbf{V}},\widehat{\mathbf{X}})\!=\!\pi$ 
  and has regular conditional distribution $\mathbf{Q}_{\pi}^{\alpha,0}$. 
  Arguing as in the proof of Proposition \ref{prop:a0:markovlike}, we can apply results from \cite{Paper1-1} in the present context. 
  Specifically, \cite[(5.3) and Lemma 5.4]{Paper1-1} yield 
  \begin{equation}\label{eq:skewerscaf}\sskewer(z,V,\xi_V)=\sskewer(z,\widehat{\mathbf{V}},\widehat{\mathbf{X}})\quad\mbox{for all }z\le y.
  \end{equation}
  Now recall that $\mathbb{Q}_\pi^{\alpha,0}$ is the distribution of $(\sskewer(r,\cdot),r\!\ge\! 0)$ under $\mathbf{Q}_\pi^{\alpha,0}$. 
  Applying $(\sskewer(z,\cdot),z\!\le\! y)$ to $(V,\xi_V)$ and
  $(\sskewer(r,\cdot),r\!\ge\!0)$\linebreak to $G^{\ge y}_0(\widehat{\mathbf{V}},\widehat{\mathbf{X}})$, we find
  $(\sskewer(y+r,\widehat{\mathbf{V}},\widehat{\mathbf{X}}),r\!\ge\! 0)$ is conditionally independent of 
  $(\sskewer(z,\widehat{\mathbf{V}},\widehat{\mathbf{X}}),z\!\le\! y)$ given $\sskewer(y,\widehat{\mathbf{V}},\widehat{\mathbf{X}})\!=\!\pi$,
  with conditional distribution $\mathbb{Q}^{\alpha,0}_\pi$. 
  
  For any $\pi^\prime\!=\!\sum_{j\ge 1}b_j\delta(x_j)\!\in\!\cM^a$, 
  we recall the construction of $\mathbf{F}_{\pi^\prime}\!\sim\!\mathbf{Q}_{\pi^\prime}^{\alpha,0}$ of Proposition \ref{prop:0:len}
  from independent $\mathbf{V}_j\sim\mathbf{Q}_{b_j,x_j}^{(\alpha)}$, $j\ge 1$.
  By this independence, the conditional independence extends to
  $(\sskewer(z,\mathbf{F}_{\pi^\prime}),z\le y)$ and $(\sskewer(y\!+\!r,\mathbf{F}_{\pi^\prime}\!),r\!\ge\! 0)$ given 
  $(\sskewer(y,\!\mathbf{V}_j,\!\xi_{\mathbf{V\!}_j}),j\!\ge\! 1)\!=\!(\pi_j,j\!\ge\! 1)$.\linebreak The conditional distribution $\mathbb{Q}_\pi^{\alpha,0}$ of
  $(\sskewer(y\!+\!r,\mathbf{F}_{\pi^\prime}),r\!\ge\! 0)$ is a consequence of Corollary \ref{cor:branchprop}; here $\pi:=\sum_{j\ge 1}\pi_j$. 
  Specifically, we note that $\pi\in\cM^a$ a.s. by Proposition \ref{prop:0:len}(iii); 
  also the $x_j$, $j\ge 1$, are distinct, and the independence of $(\mathbf{V}_j,\xi_{\mathbf{V}_j})$, $j\ge 1$, 
  each based on a countable family of independent $\Unif$ allelic types entails that the atom locations of $\pi_j$, $j\ge 1$, are distinct a.s..
  Since this conditional distribution only depends on $(\pi_j,\,j\ge 1)$ via $\pi=\sum_{j\ge 1}\pi_j$, 
  the conditional independence holds conditionally given just $\sskewer(y,\mathbf{F}_{\pi^\prime})=\pi$. 
 
  Since $\mathbb{Q}_{\pi^\prime}^{\alpha,0}$ is the distribution of $(\sskewer(z,\cdot),\,z\ge 0)$ 
  under $\mathbf{Q}_{\pi^\prime}^{\alpha,0}$,
  this can be phrased purely in terms of $(\pi^z,\,z\ge 0)\sim\mathbb{Q}_{\pi^\prime}^{\alpha,0}$. The passage from fixed $\pi^\prime\in\cM^a$
  to distributions $\mu$ on $\cM^a$ is straightforward.
\end{proof}

To strengthen the simple Markov property to a strong Markov property, we will use the standard approximation of a general finite stopping time
by a decreasing sequence of stopping times taking values in refining discrete time grids. This argument requires some continuity of
$\pi\mapsto\mathbb{Q}_\pi^{\alpha,0}$ along suitable sequences of states arising along such sequences of stopping times. 

\begin{proposition}[Continuity in the initial state]\label{prop:initial:sp}
	For a sequence $\pi_n \to \pi_{\infty}$ in  $(\cM^a,d_{\rm TV})$, there is the weak convergence 
	$\BQ^{\alpha,0}_{\pi_n} \to \BQ^{\alpha, 0}_{\pi_{\infty}}$ 
	in the sense of finite-dimensional distributions on $(\cM^a,d_{\cM})$.  
\end{proposition}
Note that we cannot expect a similar result under the weaker assumption that $\pi_n \to \pi_{\infty}$ under the Prokhorov metric. 
For example, with $\pi_n = \frac{1}{2} \Dirac{2^{-1} - 2^{-n-1}}+\frac{1}{2} \Dirac{2^{-1} + 2^{-n-1}}$, we have $\pi_n\rightarrow\pi_{\infty}= \Dirac{2^{-1}}$ weakly. 
Let $(\pi_n^y,\,y\ge 0)\sim \BQ^{\alpha,0}_{\pi_n}$ and $(\pi_{\infty}^y,\,y\ge0)\sim \BQ^{\alpha,0}_{\pi_{\infty}}$, and consider 
$m_{n}^y:= \pi_n^y((2^{-1}-\epsilon, 2^{-1}+\epsilon))$, $y\ge 0$, the mass evolution in an open interval with an arbitrarily small $\epsilon>0$: 
for $n$ large enough and $y_0>0$ small enough, 
the law of  $(m_n^y,\,y\le y_0)$ is ``close'' to the law of the sum of two independent $\BESQ_{1/2}(-2\alpha)$. 
But $(m_{\infty}^y,\,y\le y_0)$ ``nearly'' has the law of a $\BESQ_1(-2\alpha)$. 
\begin{proof}
	We fix $y>0$ and shall establish the convergence of the one-dimensional distributions at level $y$; the multi-dimensional version can be 
	proved inductively. See e.g.\ \cite[Corollary 6.16]{Paper1-1} for an instance of this inductive argument that is easily adapted. Indeed, the
	convergence of one-dimensional distributions is also adapted from \cite[Proposition 6.15]{Paper1-1}, but we provide the details so as to be 
	clear about the different topologies.   
	
        Suppose that $\pi_{\infty} = \sum_{i\ge 1} b_{i} \Dirac{x_i}$ for a non-increasing sequence $(b_i)_{i\ge 1}$ of nonnegative numbers
	and distinct $x_i\in [0,1]$, $i\ge 1$. 
	Let us construct, on a suitable probability space, a family of coupled measure-valued processes $\boldsymbol{\lambda}_n \sim \BQ^{\alpha,0}_{\pi_n}, n\in \BN\cup\{\infty\}$.  
	Specifically, let $(\mathbf{V}_{\infty, i},\,i\ge 1)$ be a family of independent random point measures with $\mathbf{V}_{\infty, i}\sim \bQ^{(\alpha)}_{b_i,x_i}$ and $\mathbf{X}_{\infty, i}:=\xi_{\mathbf{V}_{\infty,i}}$ associated scaffolding. 
	 Set the process $( \lambda_{\infty,i}^z, z\ge 0)$ to be a $d_{\rm TV}$-continuous version of $\sskewerP(\mathbf{V}_{\infty,i},\mathbf{X}_{\infty,i})$, whose existence is guaranteed by Corollary~\ref{cor:PRM:cont}. 
	By Proposition \ref{prop:0:continuity}, 
	we may further assume that $\lambda_{\infty}^z := \sum_{i\ge 1} \lambda_{\infty, i}^z$, $z\ge 0$, is $d_{\rm TV}$-continuous. 
	For an arbitrarily small $\epsilon>0$, by Proposition~\ref{prop:0:len} we can choose $\ell>0$ large enough such that
	\begin{equation}\label{eq:lem:initial:sp:1}
	\mathbf{P} \big\{\|\lambda_{\infty}^y\|>\ell\big\} < \epsilon. 
	\end{equation}
	We take the smallest $k\ge 1$ large enough so that 
	$\sum_{i>k} b_i<\epsilon$. 
	By Proposition~\ref{prop:clade:trans}, we have
	\begin{equation}\label{eq:lem:initial:sp:2}
	\mathbf{P} \left\{ \sum_{i>k} \lambda_{\infty, i}^y \ne 0\right\} \le \frac{1}{2y}\sum_{i>k} b_i\le  \frac{\epsilon}{2y}. 
	\end{equation}
	Moreover, due to the path-continuity, there exists (random) $\Delta>0$, such that for any $z\in (y-\Delta, y+\Delta)$, there is 
	\begin{equation}\label{eq:lem:initial:sp:3}
	d_{\rm TV} \left(\lambda_{\infty, i}^z, \lambda_{\infty, i}^y\right)<\frac{\epsilon}{k} \quad\text{ for all } 1\le i\le k. 
	\end{equation}
	
	By the convergence $d_{\rm TV} (\pi_n ,\pi_{\infty}) \to 0$, there exists some $m_0\ge 1$ large enough such that for all $n\ge m_0$, we can write 
	$\pi_n = \sum_{i=1}^k b_{n,i} \Dirac{x_i} + \widetilde{\pi}_n$ with $\widetilde{\pi}_n$ having  
	no mass at any $x_i$, $1\le i\le k$, with $\|\widetilde{\pi}_n\|<\epsilon$, and we can choose (random) $M\ge m_0$ so that for all $1\le i\le k$, $n>M$, we have $b_{n,i}>0$,
	\begin{equation}\label{eq:lem:initial:sp:4}
	\left|c_{n,i} -1\right|<\frac{\epsilon}{\ell} \wedge 1, 
	\text{ and }   \left|\frac{y}{c_{n,i}}  - y\right| <\Delta, \quad\mbox{where }c_{n,i}:= \frac{b_{n,i}}{b_{i}}.   
	\end{equation} 
	
	Next, we set $\mathbf{V}_{n,i}:= c_{n,i}\scaleHA\mathbf{V}_{\infty, i}$ and $\mathbf{X}_{n,i}:=\xi_{\mathbf{V}_{n,i}}$ for $n\ge 1$,
	and we note the identity
	\[
	\sskewer(y, \mathbf{V}_{n,i},\mathbf{X}_{n,i})
	= c_{n,i}\,\sskewer(y/c_{n,i}, \mathbf{V}_{\infty,i},\mathbf{X}_{\infty,i})
	=c_{n,i} \lambda_{\infty,i}^{y/c_{n,i}}. 
	\]
	Moreover, by Lemma~\ref{lm:scaling1}, we have 
	$\mathbf{V}_{n,i}\sim \bQ^{(\alpha)}_{b_{n,i},x_i}$. 
	For $n\ge 1$, let $\widetilde{\boldsymbol{\lambda}}_{n}\sim \BQ^{\alpha,0}_{\widetilde{\pi}_n}$, independent of anything else. 
	By Proposition~\ref{prop:clade:trans}, we have for all $n\ge m_0$
	\begin{equation}\label{eq:lem:initial:sp:5}
	\mathbf{P} \big\{\widetilde{\lambda}_{n}^y\neq 0\big\} \le \frac{\|\widetilde{\pi}_n\|}{2y} \le \frac{\epsilon}{2 y}. 
	\end{equation}
	Then 
	$(\lambda_{n}^y,\,y\ge 0):=\big(\widetilde{\lambda}_{n}^y + \sum_{1\le i\le K}\sskewer(y, \mathbf{V}_{n,i},\mathbf{X}_{n,i}),\, y\ge 0\big)\sim\BQ^{\alpha,0}_{\pi_n}$. 
	
	For any $n\ge M$, we have  by  \eqref{eq:lem:initial:sp:3} and \eqref{eq:lem:initial:sp:4} that 
		\begin{equation*}
	\begin{split}
	d_{\rm TV} (c_{n,i} \lambda_{\infty,i}^{y/c_{n,i}},  \lambda_{\infty,i}^{ y}  )
	&\le 
	d_{\rm TV} (c_{n,i} \lambda_{\infty,i}^{y/c_{n,i}},  c_{n,i}\lambda_{\infty,i}^{ y}  )
	+ d_{\rm TV} (c_{n,i} \lambda_{\infty,i}^{y},  \lambda_{\infty,i}^{ y}  )\\
	&\le 2\frac{\epsilon}{k} + \frac{\epsilon}{\ell} \|\lambda_{\infty,i}^{ y} \|. 
		\end{split}
	\end{equation*} 
	Now let $m\ge 1$ so that $\mathbf{P}\{M>m\}<\epsilon$. Summarizing \eqref{eq:lem:initial:sp:1},\eqref{eq:lem:initial:sp:2}, and \eqref{eq:lem:initial:sp:5}, we deduce that, for any $n\ge m$,
	\begin{equation*}
	\begin{split}
	&\mathbf{P} \big\{d_{\rm TV} (\lambda_n^y ,\lambda_{\infty}^y) >3\epsilon\big\} \\
	&\le 
	\mathbf{P} \big\{\widetilde{\lambda}_{n}^y \ne 0\big)\}
	+ \mathbf{P} \bigg\{ \sum_{i>k} \lambda_{\infty, i}^y \ne 0\bigg\}
	+ \mathbf{P} \{\|\lambda_{\infty}^y\|>\ell\} \\
	&\quad\quad + 
	\mathbf{P} \bigg\{ d_{\rm TV} \bigg(\sum_{i=1}^kc_{n,i}\lambda_{\infty,i}^{y/c_{n,i}},\sum_{i=1}^k\lambda_{\infty,i}^{ y}  \bigg)>3\epsilon, \|\lambda_{\infty}^y\|\le \ell\bigg\} \\
	&\le 2\frac{\epsilon}{2y} + 2\epsilon.
	\end{split}
	\end{equation*} 
	This completes the proof. 
\end{proof}

\begin{proposition}	[Strong Markov property] \label{prop:0-SMP}
	For a probability measure $\mu$ on $\cM^a$, let $(\pi^z,\,z\ge 0)\sim \BQ^{\alpha,0}_{\mu}$. Denote 
	its right-continuous natural filtration by $(\cF^y,\, y\ge 0)$.  
	Let $Y$ be an a.s.\@ finite  $(\cF^y ,\,y\ge 0)$-stopping time. 
	Then given $\cF^Y$, the process $(\pi^{Y+y},\, y\ge 0)$ has conditional distribution 
	$\mathbb{Q}_{\pi^Y}^{\alpha,0}$. 
\end{proposition}
\begin{proof} This is now standard, so we only provide a sketch. Consider $Y_n=2^{-n}\lfloor 2^nY+1\rfloor\downarrow 0$. 
  The strong Markov property at $Y_n$ follows from the simple Markov property by elementary partitioning. 
  Letting $n\rightarrow\infty$, we have $d_{\rm TV}(\pi^{Y_n},\pi^Y)\rightarrow 0$ a.s., by Proposition \ref{prop:0:continuity}. 
  This implies $\mathbb{Q}^{\alpha,0}_{\pi^{Y_n}}\rightarrow\mathbb{Q}^{\alpha,0}_{\pi^Y}$ weakly in the sense of finite-dimensional distributions, by Proposition 
  \ref{prop:initial:sp}. 
\end{proof}

\subsection{Proofs of Theorem \ref{thm:a0:const} and of the $\theta=0$ case of Theorem \ref{thm:sp}}\label{sec:pfa0}
		Let $\pi\in \mathcal{M}^a$ and $\bF\sim \bQ^{\alpha,0}_{\pi}$.  
Then for every $y\ge 0$, it follows from Proposition~\ref{prop:clade:trans} that the superskewer $\sskewer(y, \bF)$ has distribution 
$K^{\alpha,0}_y(\pi,\,\cdot\,)$ defined in Definition~\ref{def:kernel:sp}. 
It remains to check the properties required of a path-continuous Hunt process on $(\cM^a, d_{\cM})$, see e.g.\@ \cite[Definition A.18]{Li11}.

Specifically, we noted in Lemma \ref{lm:Lusin} that the state space $(\cM^a, d_{\cM})$ is a Lusin space (a Borel subset of a complete and separable metric space). In Proposition \ref{prop:0:continuity}, we showed that the superskewer process is a.s.\ path-continuous under 
$\bQ^{\alpha,0}_\pi$, with distribution on $\mathcal{C}([0,\infty),\cM^a)$ denoted by $\mathbb{Q}^{\alpha,0}_\pi$ in Section \ref{subsec:0:Markov}.
The map $\pi\mapsto K^{\alpha,0}_y(\pi,\,\cdot\,)$ is measurable by construction, 
it defines a semi-group on $(\cM^a,d_\cM)$ by Proposition \ref{prop:0-MP}. 
Indeed, by Lemma \ref{lm:bbq:kernel}, $\pi\mapsto\mathbb{Q}^{\alpha,0}_\pi$ is a kernel.   
The strong Markov property with respect to a right-continuous filtration was established in Proposition~\ref{prop:0-SMP}. \qed

\section{Clades of the reflected process}\label{sec:min_cld}

In this section we study clades corresponding to excursions of the reflected scaffolding $\bX(t) - \inf_{u\le t}\bX(u)$, $t\ge0$, which form a key ingredient in the construction of the general
two-parameter family of measure-valued diffusions.

\subsection{Preliminaries on reflected \Stable[1+\alpha] processes}

%


Let $\bN$ denote a $\PRM[{\tt Leb}\otimes\nu_{\tt BESQ}^{(-2\alpha)}]$ on $[0,\infty)\times\Exc$ 
or $\mathbf{V}$ a $\PRM[{\tt Leb}\otimes\nu_{\tt BESQ}^{(-2\alpha)}\otimes\Unif]$ on $[0,\infty)\!\times\!\Exc\!\times\![0,1]$ and $\bN=\varphi(\mathbf{V})$. 
Then $\bX := \xiA_\bN\sim\Stable[1\!+\!\alpha]$. We call 
\begin{figure}
 \centering
 \begin{minipage}{6.1cm}
  \centering
  \raisebox{2.6cm}{(A)} \includegraphics[width=5.4cm,height=5.25cm]{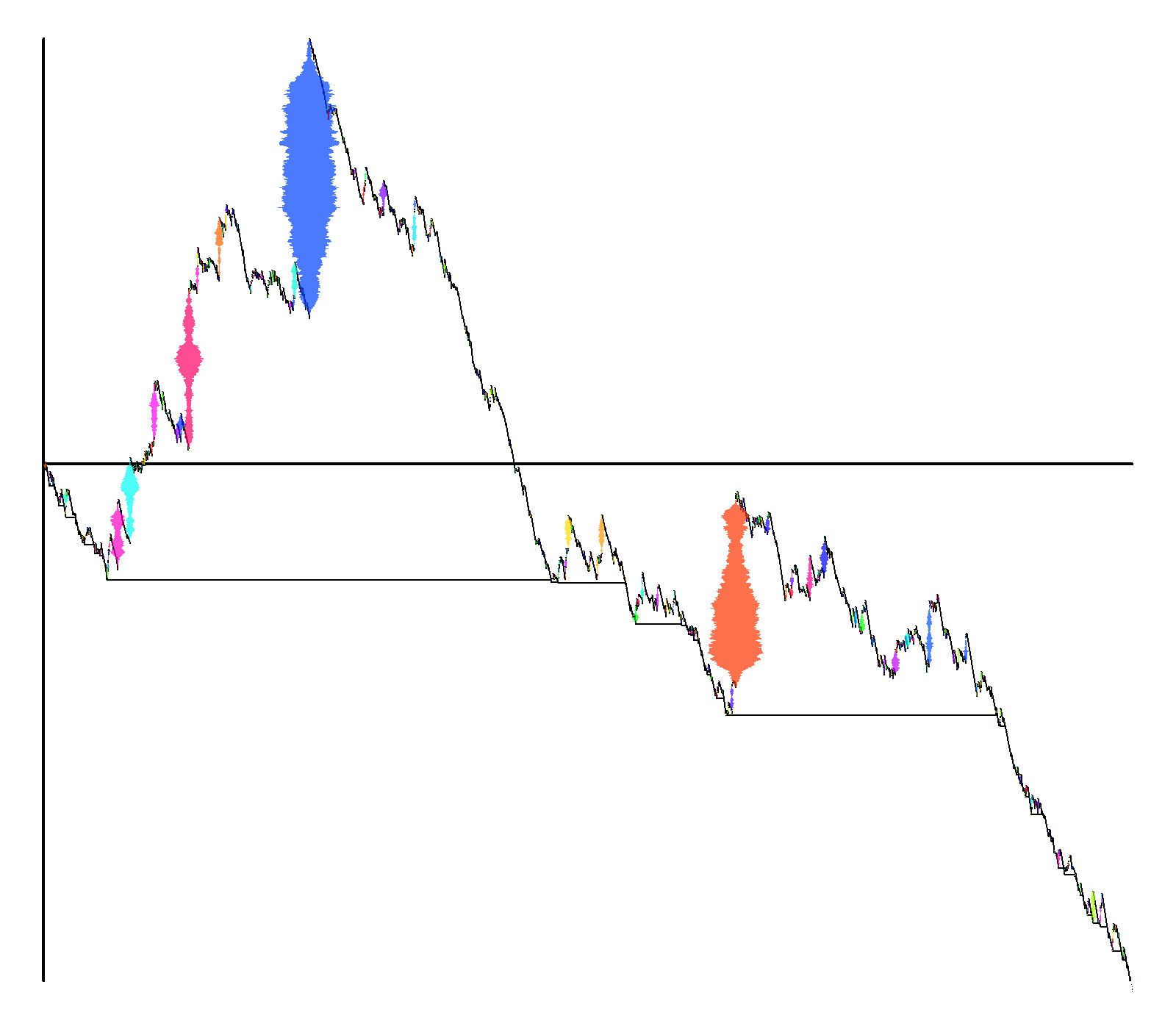} 
 \end{minipage}
 \raisebox{.07cm}{
 \begin{minipage}{6.1cm}
  \centering
  \raisebox{1.5cm}{(B)} \includegraphics[height=3cm,width=5.4cm]{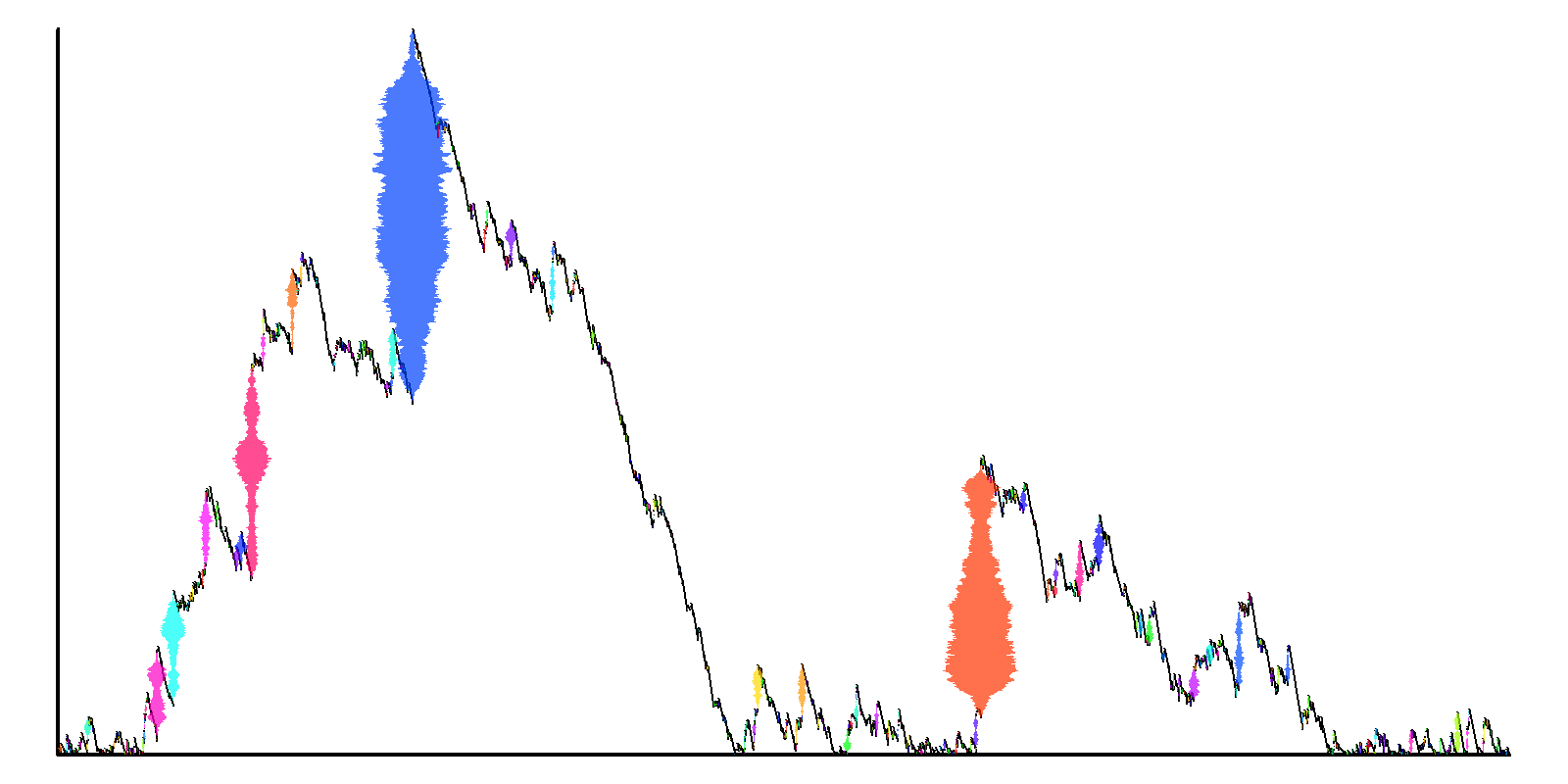}\\[.1cm]
  \raisebox{1cm}{(C)} \includegraphics[height=3cm,width=2.2cm]{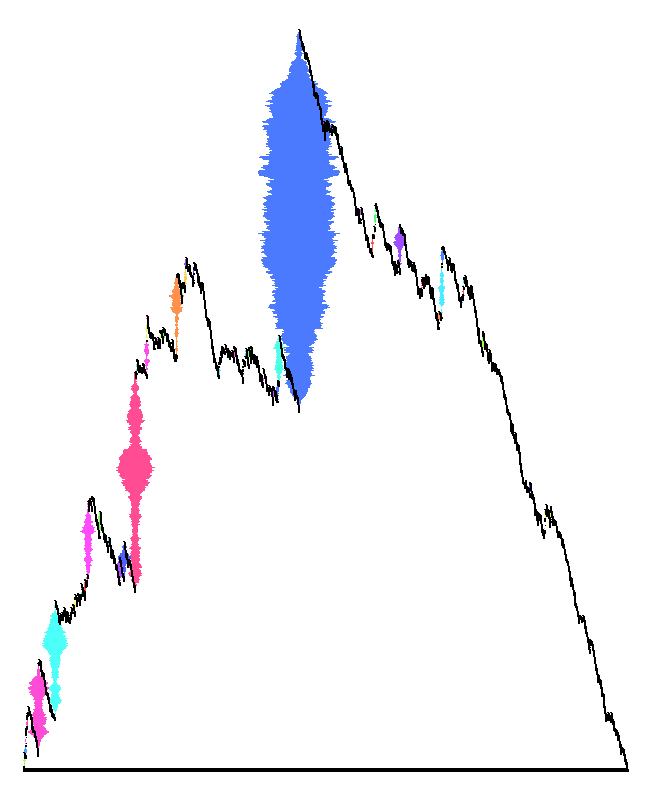}\hfill
  \raisebox{1cm}{(D)} \includegraphics[height=2.7cm,width=2.04cm]{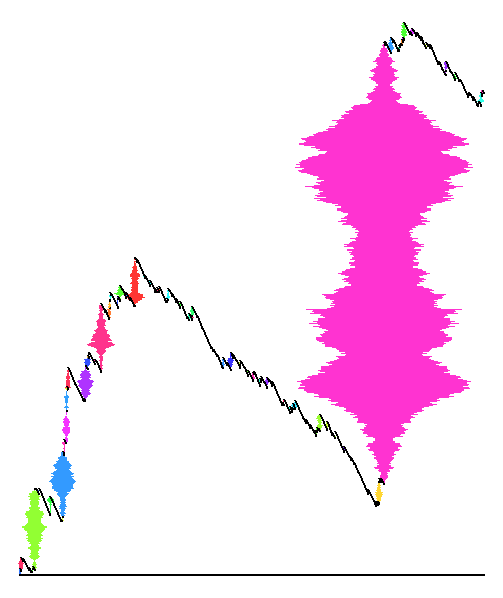}
 \end{minipage}
 }
 \caption{(A) Simulation of a \Stable$(1.3)$ scaffolding with spindles, run until hitting a negative level, superposed with the graph of its infimum process. (B) The associated reflected scaffolding with the same spindles, as in \eqref{eq:reflected_scaff}. (C) A single clade of the reflected process. (D) An enlarged (6-fold vertical, 30-fold horizontal) plot of the left end of this clade; note how there is no leftmost spindle but rather an accumulation of small spindles.\label{fig:reflected}}
\end{figure}

\begin{equation}\label{eq:reflected_scaff}
 \uX(t) := \bX(t)-\inf_{u\le t}\bX(u), \qquad t\ge0,
\end{equation}
the \emph{scaffolding process reflected at the infimum process}, or simply the \emph{reflected scaffolding process}. 
See Figure \ref{fig:reflected}. 
 Since $\uX$ is a strong Markov process \cite[Proposition VI.1]{BertoinLevy},  It\^{o}'s theory of excursions applies to $\uX$ \cite[Chapters IV and VI]{BertoinLevy}, from which we record the following two results. 
Consider the \em first passage process\em
\begin{equation}
 T^{-y} := \inf\left\{t\ge0\colon \bX(t) < -y\right\} \qquad \text{for} \qquad y\ge 0.
\end{equation}
\begin{proposition}[{\cite[Theorem VII.1]{BertoinLevy}}]\label{prop:hitting_time:subord}
 The process $(T^{-y},\,y\geq 0)$ is a \Stable[1/(1+\alpha)] subordinator. Its Laplace exponent is the inverse $\phi_\alpha=\psi_{\alpha}^{-1}$ of the Laplace exponent $\psi_\alpha$ of 
$\bX$:
 \begin{equation}\label{eq:hitting_time:Laplace}
 \begin{split}
  \EV\left[ \exp\left(-q T^{-y}\right) \right] &= \exp\left(-y\phi_{\alpha}(q)\right), \quad \text{where}\\
  \phi_{\alpha}(q) &= \left(2^{\alpha}\Gamma(1+\alpha)q\right)^{1/(1+\alpha)}, \quad q\ge 0.
 \end{split}
 \end{equation}
\end{proposition}

%
\begin{lemma}[{\cite[Theorem IV.10]{BertoinLevy}}]\label{lem:min_exc}
 Define a point process on $\mathcal{D}$ by
 \begin{equation}
  \ue := \sum_{y\ge 0\colon T^{(-y)-}<T^{-y}} 
  	\Dirac{y,\ShiftRestrict{\uX}{\left[T^{(-y)-},T^{-y}\right)}},
  \end{equation}
  where $T^{(-y)-} := \sup_{z<y}T^{-z}$ for $y>0$. Then $\ue$ is a $\PRM({\tt Leb}\otimes \mSxcA{\perp})$, where $\mSxcA{\perp}$ is known as the It\^o measure of excursions of the reflected process $\uX$.
\end{lemma}


\subsection{The It\^o measure of reflected clades}

Here, we extend the excursion theory, as we did for excursions away from fixed levels in \cite[Section~4.4]{Paper1-1}, 
to define an It\^o measure $\overline{\nu}_{\perp{\rm cld}}^{(\alpha)}$ associated with clades of $(\mathbf{V},\uX)$: 
point measures of spindles corresponding to the excursions of $\uX$. 
%
%

Define a point process of clades
\begin{equation}\label{eq:imm_PPP}
 \uF :=\sum_{y\ge 0\colon T^{(-y)-}<T^{-y}}
 \Dirac{y,\ShiftRestrict{\mathbf{V}}{\left[T^{(-y)-},T^{-y}\right)},\ShiftRestrict{\uX}{\left[T^{(-y)-},T^{-y}\right)}}. 
\end{equation}
The following statement follows readily from the marking property of \PRM s and the existence of the limits \eqref{eq:scaff_def} uniformly on all compact intervals. 

\begin{proposition}\label{prop:mark_jumps_min_cld}
  The point measure $\uF$ is a $\PRM\big({\tt Leb}\otimes \overline{\nu}_{\perp{\rm cld}}^{(\alpha)}\big)$  on $[0,\infty)\times\mathcal{V}\times\mathcal{D}$, 
  where $\overline{\nu}_{\perp{\rm cld}}^{(\alpha)}(dV,dg)=\overline{\mu}_g(dV)\mSxcA{\perp}(dg)$,
  the kernel $g\mapsto\overline{\mu}_g$ is as in Lemma \ref{prop:mark_jumps_clade} and $\mSxcA{\perp}$ is the It\^o measure of excursions of the reflected scaffolding $\uXA$. 
  Moreover, a.s.\ for all points $(V,X)$ of $\uF$, we have $X\!=\!\xi_V$.
\end{proposition}

See Figure \ref{fig:reflected} for a plot of the largest of a collection of clades sampled from a simulated approximation of $\overline{\nu}_{\perp{\rm cld}}^{(0.3)}$. 
In the rest of this section, we shall establish some relevant properties of the It\^o measure $\overline{\nu}_{\perp{\rm cld}}^{(\alpha)}$. 
We first present a scaling invariance property due to the $\Stable(1\!+\!\alpha)$-scaling property of $\bX$, which we express in terms of the scaling operation
$(c\odot_{\rm stb}^{1+\alpha}g)(t)=cg(t/c^{1+\alpha})$, for $c>0$, $t\in\mathbb{R}$, $g\in\mathcal{D}$. 

\begin{lemma}[Self-similarity of $\overline{\nu}_{\perp{\rm cld}}^{(\alpha)}$]\label{lem:min_cld:scaling}
	For $c>0$, we have
	\begin{equation}
	c\overline{\nu}_{\perp{\rm cld}}^{(\alpha)}((c\!\scaleHA\! A)\!\times\!(c\!\odot_{\rm stb}^{1+\alpha}\!B)) = \overline{\nu}_{\perp{\rm cld}}^{(\alpha)}(A\!\times\! B) \ \text{ for }A\!\in\!\Sigma(\mathcal{V}),\;B\!\in\!\Sigma(\mathcal{D}).\label{eq:min_cld:scaling}
	\end{equation}
\end{lemma}
\begin{proof}
 Let us record the scaling invariance properties of $\mBxcA$ and \StableA; see e.g.\ \cite[Lemmas 2.9 and 4.3]{Paper1-1}:  
 \begin{gather*}
  c^{1+\alpha}\mBxcA( c\scaleB B) = \mBxcA(B)\mbox{ and }(\bX (c^{1+\alpha} t) ,t\!\ge\! 0) \stackrel{d}{=} ( c \bX (t) ,t\!\ge\! 0). 
 \end{gather*}
 The claim now follows from Proposition~\ref{prop:mark_jumps_min_cld}. 
\end{proof}

Recall the notation $\life^+$ and $\len$ from Section \ref{sec:clade_stat}.

\begin{proposition}\label{prop:min_cld:stats}
	\begin{enumerate}[label=(\roman*), ref=(\roman*)]
		\item $\displaystyle \overline{\nu}_{\perp{\rm cld}}^{(\alpha)}\big\{ \life^+ > z \big\} = \alpha z^{-1}, \quad z>0.$\label{item:MCS:max}
		\vspace{4pt}
		\item $\displaystyle \overline{\nu}_{\perp{\rm cld}}^{(\alpha)}\big\{\len > x\big\} = \frac{\left(2^{\alpha}\Gamma(1+\alpha)\right)^{1/(1+\alpha)}}
		{\Gamma(\alpha/(1+\alpha)) }
		x^{-1/(1+\alpha)}, \quad x>0$.\label{item:MCS:len}
		\vspace{4pt}
		\item $\overline{\nu}_{\perp{\rm cld}}^{(\alpha)}\big\{(V,g)\in\mathcal{V}\times\mathcal{D} \colon V(\{0\}\times\Exc\times[0,1]) > 0\big\} = 0$.\label{item:mcl:first_jump}
	\end{enumerate}
\end{proposition}

Less formally, \ref{item:mcl:first_jump} states that clades of the reflected process $(\mathbf{V},\uX)$, corresponding to the atoms of $\uF$, do not have spindles at time zero, and excursions of $\uX$ do not begin with jumps. 
This proposition is essentially a consequence of known results on excursions of the reflected process $\uX$. 
For completeness, we include a proof in Appendix \ref{sec:min_cld_pf}.



\subsection{Path-continuity and Markov-like properties under $\overline{\nu}_{\perp{\rm cld}}^{(\alpha)}$}

\begin{proposition}\label{lem:min_cld:skewer}
	Let $(\mathbf{v},\mathbf{x})\sim\overline{\nu}_{\perp{\rm cld}}^{(\alpha)}\big(\,\cdot\,\big|\,\life^+>y\big)$. 
	Then $\mathbf{x}=\xi_\mathbf{v}$ a.s.. 
	Moreover, $(\sskewer (y,\mathbf{v},\mathbf{x}),\,y\ge 0)$ defined as in Definition~\ref{def:superskewer}
           is  a well-defined $\cC ([0,\infty), (\mathcal{M}^a,d_\mathcal{M}))$-valued random variable.  
	Furthermore, it is a $d_{\rm TV}$-path-continuous excursion away from $0\in\mathcal{M}^a$.
\end{proposition}
\begin{proof}
	We may assume that $\mathbf{v}$ is the clade corresponding to the first excursion $\mathbf{x}$ above the minimum of $\bX$ with height $\life^+(\mathbf{x}) > y$. 
           Then we can write $\mathbf{v} = \ShiftRestrict{\mathbf{V}}{[T', T'')}$ for a pair of a.s.\@ finite random times $T',T''$. 
	We observe from  Proposition~\ref{prop:min_cld:stats}\ref{item:mcl:first_jump} that 
	\begin{equation}\label{eq:clade_shift}
	\begin{split}
	&\overline{\sskewer}(\restrict{\mathbf{V}}{[0, T'')},\restrict{\bX}{[0, T'')} - \bX(T'))\\
	&\ \ = \overline{\sskewer} (\restrict{\mathbf{V}}{[0, T')},\restrict{\bX}{[0, T')} - \bX(T'))+ \overline{\sskewer} (\mathbf{v},\mathbf{x}).
	\end{split}
	\end{equation}
	From Proposition \ref{prop:PRM:cont}, the first two superskewer processes in this formula are a.s.\ $d_{\rm TV}$-continuous. 
	The $d_{\rm TV}$-continuity of $\overline{\sskewer}(\mathbf{v},\mathbf{x})$ follows. 
	Moreover, it follows from Proposition~\ref{prop:min_cld:stats}\ref{item:mcl:first_jump} that $\sskewer(0,\mathbf{v},\mathbf{x}) = 0$. 
\end{proof}

\begin{lemma}[Mid-spindle Markov property (MSMP) for $\overline{\nu}_{\perp{\rm cld}}^{(\alpha)}$]\label{lem:min_cld:mid_spindle}
 $\!\!\!$Fix $y\!>\!0$. Consider $(\mathbf{v},\mathbf{x})$ with law $\overline{\nu}_{\perp{\rm cld}}^{(\alpha)}\big(\,\cdot\,\big|\,\life^+ > y\big)$. Let $(T,f_T,x_T)$ denote the point of the first spindle in $\bv$ that crosses level $y$, and let $\hat f^y_{T}$ and $\check f^y_{T}$ denote its broken components. 
 Finally, let $m^y(\bv,\mathbf{x}) := \check f^y_T(y-\mathbf{x}(T-)) = \hat f^y_T(0)$. Given $m^y(\bv,\mathbf{x})$ and $x_T$, the process $\shiftrestrict{\bv}{[0,T)}+\DiracBig{T,\check f^y_{T},x_T}$ is conditionally independent of $\shiftrestrict{\bv}{(T,\infty)}+\DiracBig{0,\hat f^y_{T},x_T}$. 
 Moreover, under the conditional law, 
 $(\shiftrestrict{\bv}{(T,\infty)}, \hat f^y_{T})$ has the law of  
 $(\shiftrestrict{\mathbf{V}'}{[0,\tau]}, f')$, where $f'\sim{\tt BESQ}_{m^y(\bv,\mathbf{x})}(-2\alpha)$ is independent of $\mathbf{V}'\sim\PRM[{\tt Leb}\otimes\nu_{\tt BESQ}^{(-2\alpha)}\otimes\Unif]$, and 
 $\tau=\inf\{t\ge 0\colon \zeta(f')+\xiA_{\mathbf{V}'}(t)=-y\}$. 
\end{lemma}


We prove this in Appendix \ref{sec:min_cld_pf}.
It is an analogue to Proposition \ref{MSMP}, which established the Markovian decoupling at the first passage of level $y$ 
by $\mathbf{X}=\xi_{\mathbf{V}}$, $\mathbf{V}\sim{\tt PRM}({\tt Leb}\otimes\nu_{\tt BESQ}^{(-2\alpha)}\otimes\Unif)$, 
instead of $(\mathbf{v},\mathbf{x})\sim\overline{\nu}_{\perp{\rm cld}}^{(\alpha)}\big(\,\cdot\,\big|\,\life^+ > y\big)$. 
%
%
%
%
%
%
\begin{corollary}[Markov-like property of $\overline{\nu}_{\perp{\rm cld}}^{(\alpha)}$]\label{cor:min_cld:mid_spindle}
	Fix $y\!>\!0$ and consider $(\mathbf{v},\mathbf{x})\!\sim\!\overline{\nu}_{\perp{\rm cld}}^{(\alpha)}\big(\cdot\,\big|\,\life^+\!>\!y\big)$. 
	Then the upper point measure $G^{\ge y}_0(\mathbf{v},\mathbf{x})$ of clades is conditionally independent of the lower cutoff 
	process ${\textsc{cutoff}}^{\le y}(\mathbf{v},\mathbf{x})$ 
	given $\sskewer(y,\mathbf{v},\mathbf{x})=\pi$ and has regular conditional distribution $\mathbf{Q}^{\alpha,0}_{\pi}$.
\end{corollary}

\begin{proof}
	Using Lemma~\ref{lem:min_cld:mid_spindle} and the notation in its statement, the conditional distribution that we wish to characterize is the same as the law of $G_0^{\ge 0}(V',\xi_{V'})$ given $\sskewer(0,V',\xi_{V'})=\pi$, where $V'= \Dirac{0,f',x'}+ \shiftrestrict{\mathbf{V}'}{[0,\tau]}$. 
	By the Markov-like property of Proposition \ref{prop:a0:markovlike}, we know that the latter conditional law is $\mathbf{Q}^{\alpha,0}_{\pi}$. 
	This proves the corollary.
\end{proof}

\subsection{Entrance law of $\overline{\sskewer}$ under $\overline{\nu}_{\perp\rm cld}^{(\alpha)}$}
\newcommand{\PRMLBAU}{\ensuremath{\PRM\big(\Leb\otimes\mBxcA\otimes\Unif\big)}}

We first state a limit representation of the
underlying excursions of the $\StableA$ scaffolding.

\begin{lemma}\label{lem:min_cld:cnvgc}
 Fix $y>0$. For $a\in(0,y)$, denote by $\mathbf{x}^a$ a ${\tt Stable}(1+\alpha)$ process started at $a$, absorbed at $0$ and conditioned to reach level $y$. Then we have the following weak convergence of measures on Skorokhod space:
 \begin{equation}\label{eq:min_cld:cnvgc}
  \mathbf{P}\left\{\mathbf{x}^a\in\,\cdot\,\right\} \to \mSxcA{\perp}\big(\,\cdot\;\big|\;\life^+>y\big)\quad \text{as} \quad a\downto 0.
 \end{equation}
\end{lemma}

Similar results can be found in the literature, e.g.\ \cite[Corollaire 3]{Chaumont1994a} when conditioning to exceed an excursion length threshold rather than an excursion height threshold. Since we have been unable to find a reference for Lemma \ref{lem:min_cld:cnvgc}, we provide a proof in Appendix \ref{sec:min_cld_pf}.\pagebreak

\begin{proposition}\label{prop:min_cld:transn}
 Fix $y\!>\!0$ and let $(\mathbf{v},\mathbf{x})\!\sim\!\overline{\nu}_{\perp{\rm cld}}^{(\alpha)}\big(\cdot\big|\,\life^+\!>\!y\big)$. Then
 \begin{enumerate}[label=(\roman*), ref=(\roman*)]
  \item $m^y(\mathbf{v},\mathbf{x})\sim \GammaDist[1-\alpha,1/2y]$;\label{item:mct:LMB}
  \item $\sskewer(y,\mathbf{v},\mathbf{x}) \stackrel{d}{=} B^y\ol\Pi$, where $B^y\sim \ExpDist[1/2y]$ is independent of $\ol\Pi \sim{\tt PDRM}(\alpha,0)$.\label{item:mct:transn}
 \end{enumerate}
\end{proposition}
Together, Propositions \ref{prop:min_cld:stats}\ref{item:MCS:max} and \ref{prop:min_cld:transn}, Corollary \ref{cor:min_cld:mid_spindle} and Theorem \ref{thm:a0:const}, yield an entrance law description of $\overline{\sskewer}$ under $\overline{\nu}_{\perp\rm cld}^{(\alpha)}$. Specifically, for each $y>0$,
the point measure $\sskewer(y,\cdot)$ is non-zero at rate $\alpha/y$, distributed as an $\ExpDist[1/2y]$ multiple of an independent ${\tt PDRM}(\alpha,0)$, and $(\sskewer(z,\,\cdot\,),\,z\ge y)$ is an ${\tt SSSP}(\alpha,0)$. Here is a consequence.

\begin{corollary}\label{cor:besq0exc} We have 
  $\overline{\nu}_{\perp\rm cld}^{(\alpha)}\Big\{\big\|\overline{\sskewer}\big\|\in\cdot\Big\}=2\alpha\nu_{\tt BESQ}^{(0)}$ 
  for all $\alpha\in(0,1)$, where $\nu_{\tt BESQ}^{(0)}$ is the Pitman--Yor excursion measure of ${\tt BESQ}(0)$
  with normalisation $\nu_{\tt BESQ}^{(0)}\{\zeta>y\}=1/2y$, $y>0$.
\end{corollary}
\begin{proof} This follows from Proposition \ref{prop:0:mass} and the above entrance law discussion, together with the path-continuity noted in
  Proposition \ref{lem:min_cld:skewer}.
\end{proof}

\begin{proof}[Proof of Proposition \ref{prop:min_cld:transn}]
  Let $(\bv,\bx)\sim \umCladeAbar(\,\cdot\,\mid \life^+>y)$. For $k\ge 1$ let $\bx_k$ denote a \StableA\ L\'evy process started from $1/k$ and conditioned to exceed $y$ before being absorbed upon hitting 0. Let $\bv_k$ denote a point measure of spindles and allelic types associated with $\bx_k$ in the natural manner described in Lemma \ref{prop:mark_jumps_clade}, with each jump, including the initial jump up to level $1/k$, marked conditionally independently by a uniform type label and a \BESQ\ spindle with lifetime equal to the jump height (the marking on the first jump is unimportant for this proof, except that $\xi_{\bv_k}$ must equal $\bx_k$).
  From \eqref{eq:min_cld:cnvgc}, we may assume that these processes are coupled so that $\bx_k$ converges to $\bx$ in probability in the Skorokhod metric, as $k$ increases.

 \ref{item:mct:LMB} 
  We write ${\rm len}(g):=\sup\{t\ge 0\colon g(t)\neq 0\}$ for $g\in\mathcal{D}$, and let 
$\mathcal{D}_{\rm exc}:=\{g\in\mathcal{D}\colon g|_{(-\infty,0)}=0,\; g|_{(0,{\rm len}(g))}\neq 0\}$ denote the set of c\`adl\`ag excursion paths.  For an excursion in the set $\{g\in\mathcal{D}_{\rm exc}\colon \zeta^+(g) > y\}$, let $J^{y-}(g)$ and $J^{y}(g)$ denote the values of the excursion immediately before and after the first time $T^{\ge y}(g)$ that $g$ crosses level $y$, i.e.\ $J^{y-}(g) := g(T^{\ge y}(g)-)$ and $J^{y}(g) := g(T^{\ge y}(g))$. 
%
 For $\delta > 0$, we define
 \begin{equation*}
  S_\delta := \{g\in\mathcal{D}_{\rm exc} \colon \zeta^+(g) > y,\ J^{y-}(g) < y-\delta < y+\delta < J^y(g)\}.
 \end{equation*}
 For every $\delta$, $S_\delta$ is open in the Skorokhod topology on $\mathcal{D}_{\rm exc}$ and $J^{y-}$ and $J^y$ are Skorokhod-continuous on it. 
 It follows from \eqref{eq:min_cld:cnvgc} that, conditional on $\{\bx\in S_\delta\}$, we get
 $$(J^{y-}(\bx_{k}),J^y(\bx_{k})) \stackrel{P}{\longrightarrow} (J^{y-}(\bx),J^y(\bx))\quad \text{as }k\to\infty.$$
 For every $\epsilon > 0$ there is some $\delta>0$ sufficiently small so that $\Pr\{\bx\in S_\delta\} > 1-\epsilon$; thus, this limit holds without conditioning.
 
 Let $\bff$ (resp.\ $\bff_k$, $k\ge1$) denote the leftmost spindle of $\bv$ (resp.\ $\bv_k$) that survives to level $y$. By Proposition \ref{prop:mark_jumps_min_cld}, given $\bx$, this spindle $\bff$ has conditional law $\mBxcA(\,\cdot\; |\; \life = J^y(\bx)-J^{y-}(\bx))$. By \cite[Proposition 4.9]{Paper1-1}, which is stated for clades without type labels but adapts without modification to the present setting, for $k\ge \lceil 1/y\rceil$ the leftmost spindle of $\bv_k$ to cross level $y$ has conditional law $\bff_k \sim \mBxcA(\,\cdot\; |\; \life = J^y(\bx_k)-J^{y-}(\bx_k))$ given $\bx_k$.
 
 The leftmost spindle masses of $\bv$ and the $\bv_k$ at level $y$ are
 $$m^y(\bv,\bx) = \bff\big(y-J^{y-}(\bx)\big) \quad\text{and}\quad m^y(\bv_k,\bx_k) = \bff_k(y-J^{y-}(\bx_k)).$$ 
 Since $y-J^{y-}(\bx_k)$ is converging weakly to $y-J^{y-}(\bx)$ while $\bff_k$ is converging weakly to $\bff$, we conclude that these leftmost masses are converging weakly. From \cite[Lemma A.4]{Paper1-2}, for $c>0$,
 \begin{equation}
  \Pr\{m^y(\bv_k,\bx_k)\!\in\! dc\}\!=\!\frac{\alpha 2^\alpha c^{-1-\alpha}}{\Gamma(1-\alpha)} \frac{e^{-c/2y}\! -\! e^{-c/2(y-k^{-1})}}{(y\!-\!k^{-1})^{-\alpha}-y^{-\alpha}}dc
  	\to\frac{(2y)^{\alpha-1}e^{-c/2y}}{c^\alpha \Gamma(1-\alpha)}dc
 \end{equation}
 as $k$ increases, by L'H\^opital's Rule. Thus, $m^y(\bv,\bx) \sim \GammaDist(1\!-\!\alpha,1/2y)$, as claimed.
 
 \ref{item:mct:transn} We apply the MSMP of Lemma \ref{lem:min_cld:mid_spindle} to $(\bv,\bx)$ at level $y$. Then, given $\restrict{\bv}{[0,T^{\ge y}]}$, the point measure $(\shiftrestrict{\bv}{(T^{\ge y},\infty)}$ is conditionally a \PRMLBAU\ killed when its scaffolding hits level $-J^y(\bx)$. 
 This same situation has been addressed in Proposition \ref{prop:clade:trans} and previously in \cite[proof of Proposition 3.4]{Paper1-2} with the conclusion that, if $\lambda := \sskewer(\bv,\bx,y) - m^y(\bv,\bx)\delta_x$ denotes the superskewer minus the atom corresponding to the leftmost spindle at level $y$, then $G := \lambda([0,1]) \sim \GammaDist[\alpha,1/2y]$ and $\ol\lambda := \lambda/G \sim \PDRM(\alpha,\alpha)$. Moreover, $\ol\lambda$, $G$, and $m^y(\bv,\bx)$ are independent of each other. Now, the claim follows from well-known descriptions of the $\PoiDir(\alpha,\theta)$ distributions. In particular, if one takes a $\PoiDir(\alpha,0)$-distributed sequence, removes one of the entries chosen as a size-biased random pick, then renormalizes the remainder of the sequence, then the removed entry is independent of the renormalized remaining sequence, and the two have respective laws $\BetaDist(1-\alpha,\alpha)$ (the law of $m^y(\bv,\bx) / (m^y(\bv,\bx)+G)$) and $\PoiDir(\alpha,\alpha)$ (the law of ranked masses of $\lambda/G$) \cite[Theorem 3.2, Definition 3.3]{CSP}.
\end{proof}


\section{Self-similar $(\alpha,\theta)$-superprocesses}\label{sec:theta}

Recall that we outlined two approaches to self-similar $(\alpha,\theta)$ superprocesses, ${\tt SSSP}(\alpha,\theta)$, in the introduction: 
the first is in terms of kernels that we claim, in Theorem \ref{thm:sp}, give rise to a Hunt process; 
the second is in terms of point measures that we claim, in Theorem \ref{thm:at:const}, have a Markovian superskewer process. 
As in the special case $\theta=0$ that we established in Section \ref{sec:alphazero}, 
we will study the point measure construction and compute its semi-group to connect the two approaches.

\subsection{Marginal distributions of the superskewer process}\label{sec:theta:def}
 
Fix $\theta>0$. Recall the point measure construction in the statement of Theorem \ref{thm:at:const}: 
for any initial measure $\pi\in\mathcal{M}^a$, we construct a measure-valued process by adding two independent superskewer processes.
The first is associated with $\mathbf{F}\sim\mathbf{Q}_\pi^{\alpha,0}$,
which we know from Theorem \ref{thm:a0:const}, already proved, gives an ${\tt SSSP}_\pi(\alpha,0)$. 
The second is associated with $\cev{\mathbf{F}}\!\sim\!\mathbf{Q}_0^{\alpha,\theta}\!:=\!{\tt PRM}(\frac{\theta}{\alpha} {\tt Leb}\!\otimes\!\overline{\nu}_{\perp\rm cld}^{(\alpha)})$, on $[0,\infty)\!\times\!\mathcal{V}\!\times\!\mathcal{D}$, as
$$
\sskewer(y,\cev{\mathbf{F}})=\sum_{{\rm points}\;(z,V_z,X_z)\;{\rm of}\;\cev{\mathbf{F}}}\!\!\!\!\!\sskewer(y-z,V_z,X_z), \quad y\ge 0.
$$
We interpret each atom $(z,V_z,X_z)$ of $\cev{\mathbf{F}}$ as a sub-population -- a clade, in the biological sense -- that enters via immigration at level (time) $z$. The $\theta$ parameter is understood as the rate at which new immigration enters. 
For any $y>0$, the atomic measure $\sskewer(y,\cev{\mathbf{F}})$ is the superskewer at level $y$ of the population formed by clades that enter below level $y$. We stress that it comprises an infinite summation of superskewers of clades.
The following proposition gives the marginal distribution of the process 
$$\sskewerP\left(\cev{\mathbf{F}}\right):=\left(\sskewer\left(y,\cev{\mathbf{F}}\right),\,y\ge 0\right).$$ 
In particular, it confirms that the process takes values in the space $\cM^a$ of finite atomic measures. 
\begin{proposition}\label{prop:theta:entrance}
	Let $\cev{\mathbf{F}}\!\sim\!\PRM[\frac{\theta}{\alpha}{\tt Leb}\!\otimes\!\overline{\nu}_{\perp {\rm cld}}^{(\alpha)}]$. For any $y\!\ge \!0$, consider $G^y\!\sim\!\GammaDist[\theta,1/2y]$ and an independent random measure  $\overline{\Pi}_0\sim\PDRMAT$. Then 
	$\sskewer(y,\cev{\mathbf{F}}) \stackrel{d}{=} G^y\overline{\Pi}_0$. 
\end{proposition}
\begin{corollary}[Marginal distributions]\label{cor:at:marginals} In the setting of Theorem \ref{thm:at:const}, 
  $$\sskewer(y,\cev{\mathbf{F}})+\sskewer(y,\mathbf{F})\sim K^{\alpha,\theta}_y(\pi,\,\cdot\,)$$
  for each $y>0$, where $K^{\alpha,\theta}_y(\pi,\,\cdot\,)$ is the kernel of Theorem \ref{thm:sp}.
\end{corollary}
\begin{proof} Given the definition of $K_y^{\alpha,\theta}$ in Definition \ref{def:kernel:sp}, this follows straight from the proposition and the
  kernel $K_y^{\alpha,0}$ that we related to the superskewer construction in Theorem \ref{thm:a0:const}, as proved in Section \ref{sec:pfa0}.
\end{proof}

To prove the proposition, we require an intriguing multivariate distributional identity.

\begin{lemma}\label{lem:gamma_ident}
	Consider $G\sim\GammaDist[\theta,\rho]$ and suppose that for $n\geq 1$ we have $E_n\sim\ExpDist[\rho]$ and $B_n\sim\BetaDist[\theta,1]$, with all of these variables jointly independent. Then
	\begin{equation}\label{eq:gamma_ident}
	\left( E_n\cdot\prod_{i=1}^n B_i,\ n\ge 1 \right) 
	\stackrel{d}{=} 
	\left( G(1-B_n)\cdot\prod_{i=1}^{n-1} B_i,\ n\ge 1 \right).
	\end{equation}
	Moreover, $\sum_{n\ge1}E_n\cdot\prod_{i=1}^n B_i \sim \GammaDist[\theta,\rho]$.
\end{lemma}

If we take each coordinate separately, then the claim follows from standard Beta-Gamma calculus. The joint distributional identity is subtler. 
It can be read from \cite[Theorems 3--6]{Tavare1987} in the context of the limiting behaviour of birth processes with immigration, but it can 
also be quickly proved by more direct arguments, as we show here. 
\begin{proof}
	The second conclusion follows from the first, as the terms in the sequence on the right in \eqref{eq:gamma_ident} form a telescoping series, adding up to $G$. Thus, we need only verify \eqref{eq:gamma_ident}. Assume WLOG that $\rho=1$. Let $J_n := E_n\cdot\prod_{i=1}^n B_i$ and $L_n := G(1-B_n)\cdot\prod_{i=1}^{n-1} B_i$. It suffices to show that, for every $n\geq 1$ and $(r_1,\ldots,r_n)\in \BN^n$, we have $\EV[J_1^{r_1}\cdots J_n^{r_n}] = \EV[L_1^{r_1}\cdots L_n^{r_n}]$.
	
	For $j\in [n]$, let $r_{j:n}$ denote $r_j+r_{j+1}+\cdots+r_n$. Then
	\begin{equation*}
	\begin{split}
	\EV[J_1^{r_1}\cdots J_n^{r_n}] &= \EV[B_1^{r_{1:n}}B_2^{r_{2:n}}\cdots B_n^{r_n}E_1^{r_1}\cdots E_n^{r_n}]\\
	&= \left(\frac{\theta}{\theta+r_{1:n}}\right)\left(\frac{\theta}{\theta+r_{2:n}}\right)\cdots\left(\frac{\theta}{\theta+r_n}\right) r_1!\cdots r_n!,
	\end{split}
	\end{equation*}
	and
	\begin{equation*}
	\begin{split}
	&\EV[L_1^{r_1}\cdots L_n^{r_n}] = \EV[(1-B_1)^{r_1}B_1^{r_{2:n}}(1-B_2)^{r_2}B_2^{r_{3:n}}\cdots (1-B_n)^{r_n}G^{r_{1:n}}]\\
	&\ \ = \theta\!\left(\frac{\Gamma(\theta\!+\!r_{2:n})r_1!}{\Gamma(\theta\!+\!r_{1:n}\!+\!1)}\right)\!\theta\!\left(\frac{\Gamma(\theta\!+\!r_{3:n})r_2!}{\Gamma(\theta\!+\!r_{2:n}\!+\!1)}\right)\cdots\theta\!\left(\frac{\Gamma(\theta)r_n!}{\Gamma(\theta\!+\!r_{n}\!+\!1)}\right)\! \frac{\Gamma(r_{1:n}\!+\!\theta)}{\Gamma(\theta)}.
	\end{split}
	\end{equation*}
	Clearly, these two products are equal, as desired.
\end{proof}

\begin{proof}[Proof of Proposition \ref{prop:theta:entrance}]
	Consider the process
	\begin{equation}
	K^y(s) := 
	\int_{[0,s]\times\mathcal{V}\times\mathcal{D}} \cf\!\left\{\life^+(V) - r > y \right\}\cev{\mathbf{F}}(dr,dV,dX),
	\quad s\in [0,y),
	\end{equation}
	counting the clades that enter via immigration below level $s$ and survive to level $y$. 
	Using Proposition~\ref{prop:min_cld:stats}\ref{item:MCS:max}, we deduce that $ (K^y(s), s\in[0,y))$ is an inhomogeneous Poisson process with intensity
	\begin{equation*}
	\frac{\theta}{\alpha}\umCladeAbar\left\{\life^+ > y - s\right\}ds = \theta\left(y - s\right)^{-1}ds \qquad \text{for }s\in \left[0,y\right). 
	\end{equation*}
	Let $S_0 := 0$ and for $i\ge 1$ let $S_i$ denote the time of the $i^{\text{th}}$ jump of $K^y$. Then almost surely, for $i\ge1$,  
	$$\Pr(S_i > s\; |\; S_{i-1}) = \exp \!\left(\! -\!\int_{S_{i-1}}^s \!\!\!\theta (y \!-\! r)^{-1}dr \!\right) = \left(\!\frac{y\!-\!s}{y \!-\! S_{i-1}}\! \right)^{\!\!\theta}\!\!,\quad s\in [S_{i-1},y].$$
	Thus, setting
	$$B_i := \frac{y-S_i}{y - S_{i-1}} \qquad \text{for each }i\ge1,$$
	the sequence $(B_i,\,i\ge 1)$ is i.i.d.\ \BetaDist[\theta,1]. Let $(V_i,X_i)$, $i\ge1$, denote the sequence of clades corresponding to the $S_i$, so $\cev{\mathbf{F}}$ has a point at each $(S_i,V_i,X_i)$. By the \PRM\ definition of $\cev{\mathbf{F}}$, the pairs $((V_i,X_i),\,i\ge1)$ are conditionally independent given $(S_j,\,j\ge1)$, with respective conditional distributions $\umCladeAbar\{\,\cdot \mid \life^+ + S_i > y\}$. 
	Let
	\begin{equation*}
	M_i := \big\|\sskewer\left(y\!-\!S_i,V_{i},X_i\right)\big\|, \quad  
	\overline{\Pi}_i := \frac{1}{M_i}\scaleI\sskewer\left(y\!-\!S_i,V_{i},X_i\right).
	\end{equation*}
	By Proposition \ref{prop:min_cld:transn}, conditionally given $S_i$, we get 
	\[
	M_i\sim\ExpDist[1/2(y-S_i)] \quad\text{and}\quad \overline{\Pi}_i\sim\PDRM[\alpha,0] ,
	\] 
	with the two being independent. Since the distribution of $\overline{\Pi}_i$ does not depend on $S_i$, and since products of the $(B_j)$ variables are telescoping, we get a nicer characterization by setting
	$$E_i := M_i \Bigg(\prod_{j=1}^i B_j\Bigg)^{\!\!-1} \sim \ExpDist[1/2y].$$
	Indeed, the sequences $(B_i)$, $(E_i)$, and $(\overline{\Pi}_i)$ are each i.i.d.\ and are jointly independent of each other. Now,
	\begin{equation}\label{eq:level_y_dist}
	\sskewer(y,\cev{\mathbf{F}})
	= \sum_{i\ge 1} \Bigg(E_i\prod_{j=1}^i B_j\Bigg)\overline{\Pi}_i.
	\end{equation}
	Note that the scaling factors in this concatenation are the terms on the left hand side in \eqref{eq:gamma_ident}.
	
	Consider $\overline{\Pi}^\prime\sim\PDRM[\alpha,\theta]$, $\overline{\Pi}\sim\PDRM[\alpha,0]$, and $B\sim\BetaDist[\theta,1]$ independent of each other. We appeal to a classical decomposition,
	\begin{equation}\label{eq:at_a0}
	\overline{\Pi}^\prime \stackrel{d}{=} B \overline{\Pi}^\prime + (1-B)\overline{\Pi},
	\end{equation}
	e.g., from \cite[Proposition 3.16]{CSP}. Iterating \eqref{eq:at_a0} yields
	\begin{equation}
	\sum_{i\ge 1} \Bigg((1-B_i)\prod_{j=1}^{i-1}B_j\Bigg)\overline{\Pi}_i \sim \PDRM[\alpha,\theta].
	\end{equation}
	Scaling each term by a common factor of $G\sim\GammaDist[\theta,1/2y]$, taken to be independent of the other variables, yields a concatenation of a sequence of independent \PDRM[\alpha,0], scaled by the terms of the sequence on the right hand side in \eqref{eq:gamma_ident}. Thus, by \eqref{prop:min_cld:stats} and Lemma \ref{lem:gamma_ident},
	\begin{equation*}
		\sskewer(y,\cev{\mathbf{F}})\stackrel{d}{=} \sum_{i\ge 1} \Bigg(G(1-B_i)\prod_{j=1}^{i-1}B_j\Bigg)\overline{\Pi}_i,
	\end{equation*}
	which has the claimed distribution.
\end{proof}

\subsection{Path-continuity and Markov property}

\begin{figure}[t]
 \centering
 \includegraphics[width=8.12cm , height=6cm]{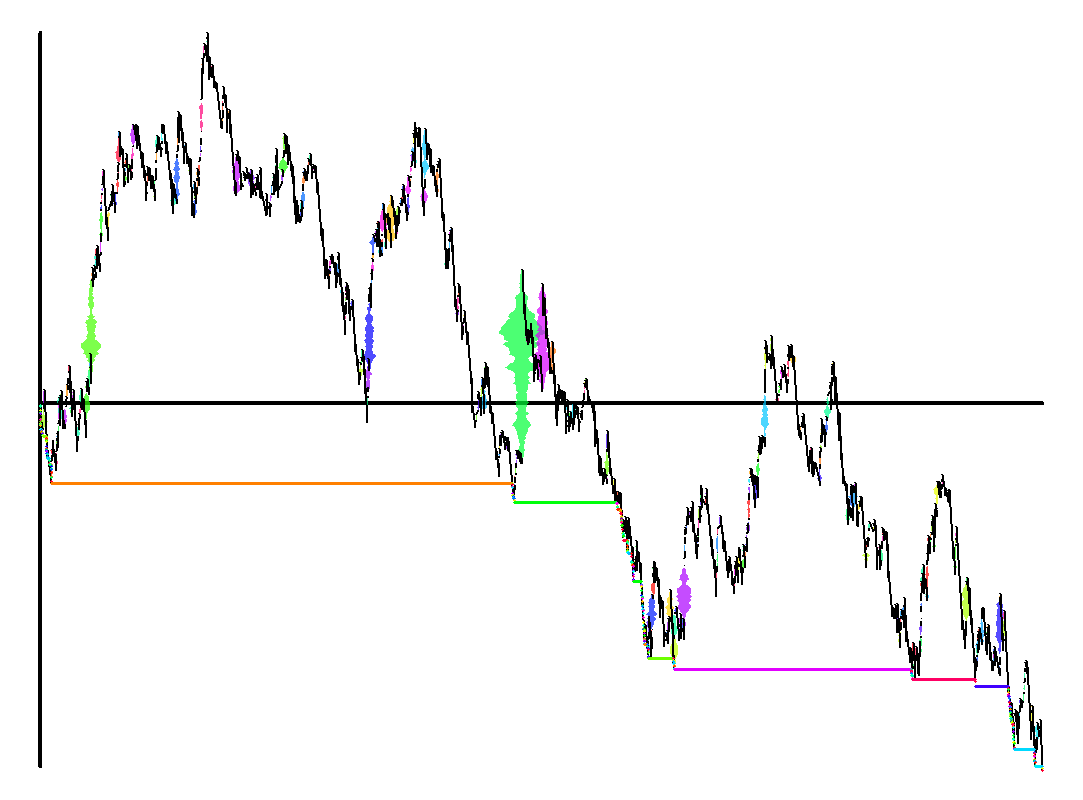} \ \ \ \raisebox{.08cm}{\includegraphics[width=3.88cm , height=4.065cm]{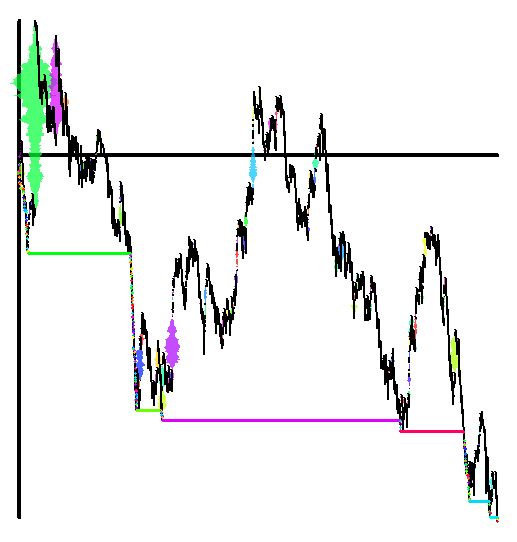}}\ 
 \caption{Simulated thinning of scaffolding and spindles as in the proof of Proposition \ref{prop:cevF:path-continuity}, with $(\alpha,\theta) = (0.6,0.4)$. Left: $(\mathbf{V},\bX)$. Right: $(\mathbf{V}_{\theta},\bX_{\theta})$. Corresponding clades of the reflected process are indicated by colored underlining (unrelated to spindle types).\label{fig:theta_thinning}}
\end{figure}
\begin{proposition}\label{prop:cevF:path-continuity}
	Let $\theta\ge 0$ and $\cev{\mathbf{F}}\sim \PRM[\frac{\theta}{\alpha}{\tt Leb}\otimes\umCladeAbar]$. 
	Define $\cev{\pi}^y := \sskewer(y, \cev{\mathbf{F}})$, $y\ge 0$. 
        Then the process $(\cev{\pi}^y,y\ge 0)$ is $d_{\rm TV}$-path-continuous.
\end{proposition}
We immediately deduce the following statement from Propositions~\ref{prop:0:continuity} and \ref{prop:cevF:path-continuity}. 
\begin{corollary}\label{cor:at:continuity}  In the setting of Theorem \ref{thm:at:const}, with $\mathbf{F}\!\sim\!\mathbf{Q}^{\alpha,0}_\pi$, $\pi\!\in\!\cM^a$,
  $\big(\sskewer(y,\cev{\mathbf{F}})+\sskewer(y,\mathbf{F}),\,y\ge 0\big)$ is $d_{\rm TV}$-path-continuous starting from $\pi$. 
\end{corollary}
We denote the distribution of $\big(\sskewer(y,\cev{\mathbf{F}})+\sskewer(y,\mathbf{F}),\,y\ge 0\big)$ on $\mathcal{C}([0,\infty),(\cM^a,d_\mathcal{M}))$, by $\mathbb{Q}_\pi^{\alpha,\theta}$, for each $\pi\in\mathcal{M}^a$ and note the following consequence of 
the additivity of \PRM s and other point measures.
\begin{corollary}[Additivity property]\label{cor:theta:additive}
	Let $\lambda_1, \lambda_2\in \cM^a$, with distinct atom locations. 
  For independent $(\pi_1^y,y\!\ge \!0)\sim\mathbb{Q}_{\lambda_1}^{\alpha,\theta_1}$ 
  and $(\pi_2^y,y\!\ge\! 0)\sim\mathbb{Q}_{\lambda_2}^{\alpha,\theta_2}$, 
  we have $(\pi_1^y+\pi_2^y,y\!\ge\! 0)\sim\mathbb{Q}_{\lambda_1+\lambda_2}^{\alpha,\theta_1+\theta_2}$.
\end{corollary}
\begin{proof}[Proof of Proposition \ref{prop:cevF:path-continuity}]
        Fix $y_0\!>\!0$. To show $d_{\rm TV}$-path-continuity on $[0,y_0]$, 
	first consider the special case $\theta\!\in\!(0,\alpha]$. 
	Our starting point is $\mathbf{V}\!\sim\!{\tt PRM}({\tt Leb}\otimes\nu_{\tt BESQ}^{(-2\alpha)}\otimes\Unif)$ with scaffolding $\bX\!=\!\xi_\mathbf{V}$, 
	and the associated ${\tt PRM}$ $\uF$ of clades of the reflected process defined in \eqref{eq:imm_PPP}, which we restrict to $[0,y_0]\times\mathcal{V}\times\mathcal{D}$.  
	By thinning $\uF$ retaining each clade with probability $\theta/\alpha$, 
	we obtain a ${\tt PRM}(\frac{\theta}{\alpha}{\tt Leb}\otimes\umCladeAbar)$, which we denote by $\cev{\mathbf{F}}_{\theta}$. 
	We concatenate the clades of $\cev{\mathbf{F}}_{\theta}$ to a point measure $\mathbf{V}_\theta$ 
	and a scaffolding $\bX_\theta$ in which the excursions of the reflected process are still at the same levels as in $\bX$; see Figure \ref{fig:theta_thinning}. 
	By a space-time shift, we see that 
	$$(\sskewer(y,\mathbf{V}_\theta,y_0+\bX_\theta),y\in[0,y_0]) \stackrel{d}{=}(\cev{\pi}^y,y\in[0,y_0]),$$
	where $(\cev{\pi}^y,\,y\ge 0)$ is as in the statement of the proposition.
	
	Let $\gamma\!\in\!(0,1/2)$. The $\BESQ[-2\alpha]$ excursions $Z_t$ marking $\Delta\bX_\theta(t)$ 
	have H\"older constants $D_\gamma^*(t)=\sup_{0\le x<y\le\Delta X(t)}|Z_t(y)-Z_t(x)|/|y-x|^\gamma<\infty$ a.s., by \cite[Corollary 36]{Paper0}.
	If furthermore $\gamma\!<\!1\!-\!\alpha$, then Lemma \ref{lem:holder} yields that the set of jump times of $\bX$ and hence $y_0+\bX_\theta$, by thinning, 
	may a.s. be partitioned into a sequence of ``piles'' of jumps $\{J_j^n,j\!\ge\! 1\}$, $n\!\ge\! 1$, 
	in such a way that the jump intervals $[\bX(J_j^n-),\bX(J_j^n)]$, $j\!\ge\! 1$, 
	are disjoint for each $n\!\ge\! 1$ and the H\"older constants $D_n=\sup_{j\ge 1}D_\gamma^*(J_j^n)$ are summable in $n\!\ge \!1$.
	
	Furthermore, if $m\alpha<\theta\le (m+1)\alpha$ for some $m\ge 1$, this conclusion applies with $\theta$ replaced by $\theta-m\alpha$, or with $\theta$ replaced by $\alpha$. 
	Taking the former and, independently, $m$ copies of the latter, we can merge the $m+1$ sequences of piles into a single sequence of piles whose H\"older constants are still summable.
	Furthermore, we can superpose the $m+1$ Poisson random measures of clades of the reflected process to construct a point measure $\mathbf{V}_\theta$ and scaffolding $y_0+\bX_\theta$ 
	such that the associated superskewer process is again distributed like $(\cev{\pi}^y,y\in[0,y_0])$, as in the case $\theta\in(0,\alpha]$. 
	
	The proof of Proposition~\ref{prop:PRM:cont} is easily adapted to prove $(\cev{\pi}^y,y\in[0,y_0])$ is a.s.\ H\"older-$\gamma$ in $(\cM^a,d_{\rm TV})$.
\end{proof}

To describe the Markov property we require additional notation. Recall the upper point measure of clades \eqref{eq:upper-clade} in a 
pair $(V,X)\!\in\!\mathcal{V}\!\times\!\mathcal{D}$. 
For a point measure $\cev F$ on $[0,\infty)\!\times\! \mathcal{V}\!\times\!\mathcal{D}$ we similarly define\vspace{-0.1cm}
\begin{equation}\label{eq:cevF:G0}
G_0^{\ge y} (\cev F) := \sum_{\text{points }(s,V_s,X_s)\text{ of }\cev F} \!\!\mathbf{1}\{s\in [0,y]\} G_0^{\ge y-s}(V_s,X_s),\vspace{-0.1cm}
\end{equation}
and, extending the notation of \eqref{eq:cutoff}, we define the lower cutoff process as
\begin{equation}\label{eq:cevF:cutoffL}
{\textsc{cutoff}}^{\le y}(\cev F):= \!\sum_{\text{points }(s,V_s,X_s)\text{ of }\cev F} \!\!\mathbf{1}\{s\!\in\! [0,y]\}\Dirac{s,{\textsc{cutoff}}^{\le y-s}(V_s, X_s)}\!.\vspace{-0.1cm}
\end{equation}

The following result is analogous to Proposition~\ref{prop:a0:markovlike}. 

\begin{lemma}[Markov-like property]\label{lem:cevF:cutoff}
	Consider $\cev{\mathbf{F}} \sim \PRM[\frac{\theta}{\alpha}{\tt Leb}\otimes\umCladeAbar]$ on $[0,\infty)\times\mathcal{V}\times\mathcal{D}$ 
	and $y>0$. 
	Then given $\textsc{cutoff}^{\le y}(\cev\bF)$, we have the following conditional distribution: \vspace{-0.1cm}
	\begin{equation}\label{eq:lem:cevF:cutoff}
	\left( \Restrict{\cev{\mathbf{F}}}{(y,\infty)\times \mathcal{V}\times\mathcal{D}}
	, G^{\ge y}_0\big(\cev{\mathbf{F}}\big) \right) 
	\sim \PRM[\frac{\theta}{\alpha}\Restrict{{\tt Leb}}{(y,\infty)}\otimes\umCladeAbar]
	\!\otimes \mathbf{Q}^{\alpha,0}_{\pi^y},\vspace{-0.1cm}
	\end{equation}
	where $\pi^y= \sskewer(y, \cev{\mathbf{F}})$. 
\end{lemma}
\begin{proof}
	We begin by decomposing $\cev{\mathbf{F}}$ into a sum of three independent terms: the point process of clades that enter below level $y$ and survive up to that level, those that enter below level $y$ but do not survive to level $y$, and those that enter above level $y$;\vspace{-0.1cm}
	\begin{gather*}
	\cev\bF_1 := \Restrict{\cev{\mathbf{F}}}{A^{\ge y}}, \quad \cev\bF_2 := \restrict{\cev{\mathbf{F}}}{A^{<y}}, \quad \cev\bF_3 := \restrict{\cev{\mathbf{F}}}{(y,\infty)\times\mathcal{V}\times\mathcal{D}},\ \ \text{where}\\
	A^{\ge y} := \{(s,V,X)\in [0,y]\!\times\!\mathcal{V}\!\times\!\mathcal{D}\colon \zeta^+(V) + s \ge y \},\\
	A^{< y} := \{(s,V,X)\in [0,y]\!\times\!\mathcal{V}\!\times\!\mathcal{D}\colon \zeta^+(V) + s < y \}.\vspace{-0.1cm}
	\end{gather*}
	By the Poisson property, these are three independent \PRM s. The claimed conditional independence and distribution of the first coordinate in \eqref{eq:lem:cevF:cutoff} follow immediately from three observations: this first coordinate equals $\cev\bF_3$, we are conditioning on a function of $\cev\bF_1+\cev\bF_2$, and  $G^{\ge y}_0\big(\cev{\mathbf{F}}\big) = G^{\ge y}_0(\cev\bF_1)$.
	
	Note that, in particular, we are conditioning on $\textsc{cutoff}^{\le y}(\cev{\mathbf{F}})$, which equals $\textsc{cutoff}^{\le y}(\cev\bF_1) + \textsc{cutoff}^{\le y}(\cev\bF_2)$. Again, appealing to the independence given by the Poisson property, it remains only to show that given $\bG := \textsc{cutoff}^{\le y}(\cev\bF_1)$, we get $\mathbf{Q}^{\alpha,0}_{\pi^y}$ as the conditional distribution of
	 $G^{\ge y}_0(\cev{\mathbf{F}})$.
	
	Recall from the proof of Proposition \ref{prop:theta:entrance} the notation $((S_i,V_i,X_i),\,i\ge 1)$, for the sequence of points of 
	$\cev{\mathbf{F}}$ that give all the clades $(V_i,X_i)$ surviving past level $y$ in the increasing order of their immigration times $S_i$.
	In this notation,
	\begin{equation*}
	G^{\ge y}_0(\cev{\mathbf{F}}) = \sum_{i\ge 1} G_0^{\ge y-S_i}(V_i,X_i), \quad \bG = \sum_{i\ge 1} \DiracBig{S_i,\textsc{cutoff}^{\le y-S_i}(V_i,X_i)}.
	\end{equation*}
	The Poisson property and the Markov-like property of Corollary~\ref{cor:min_cld:mid_spindle} yield: (i) that $G_0^{\ge y-S_i}(V_i,X_i)$, 
	$i\ge 1$, are conditionally independent given $\bG$, and (ii) that each $G_0^{\ge y-S_i}(V_i,X_i)$ then has conditional law 
	$\mathbf{Q}^{\alpha,0}_{\pi^y_i}$, where $\pi^y_i = \sskewer(y-S_i,V_i,X_i)$. 
	Since $\pi^y = \sum_{i\ge 1} \pi^y_i$, and all types are distinct a.s., we deduce from Corollary \ref{cor:branchprop} that 
	$G^{\ge y}_0(\cev{\mathbf{F}})$ has the conditional law $\mathbf{Q}^{\alpha,0}_{\pi^y}$, given $\bG$. 
\end{proof} 

\begin{proposition}[Markov property] \label{prop:theta-MP}
	Let $\mu$ be a probability measure on $\cM^a$ and $(\pi^z,\,z\ge 0)\sim \BQ^{\alpha,\theta}_{\mu}$. 
	For any $y\ge 0$, the process $(\pi^{y+r},\,r\ge 0)$ is conditionally independent of $(\pi^z,\,z\le y)$ given $\pi^y=\pi$ and has 
	regular conditional distribution $\BQ^{\alpha,\theta}_{\pi}$. 
\end{proposition}
\begin{proof} Consider independent $\cev{\mathbf{F}}\sim\mathbf{Q}_0^{\alpha,\theta}$ and $\mathbf{F}\sim\mathbf{Q}_\pi^{\alpha,0}$. 
  Applying the superskewer process to the lower cutoff process and the upper point measure in 
  Lemma \ref{lem:cevF:cutoff} yields that $(\sskewer(y+r,\cev{\mathbf{F}}),\,r\ge 0)\sim\mathbb{Q}^{\alpha,\theta}_{\pi_1}$ is conditionally 
  independent of $(\sskewer(z,\cev{\mathbf{F}}),\,z\le y)$ given $\sskewer(y,\cev{\mathbf{F}})=\pi_1$. 
  
  Independently, the Markov property of $\sskewerP(\mathbf{F})$ in Proposition \ref{prop:0-MP} yields that
  $(\sskewer(y+r,\mathbf{F}),\,r\ge 0)\sim\mathbb{Q}^{\alpha,0}_{\pi_2}$ is conditionally 
  independent of $(\sskewer(z,\mathbf{F}),\,z\le y)$ given $\sskewer(y,\mathbf{F})=\pi_2$. 
  
  By Corollary \ref{cor:theta:additive}, as $\pi_1$ and $\pi_2$ a.s.\ have no shared types, the conditional distribution of the post-$y$ superskewer process 
  $(\sskewer(y+r,\cev{\mathbf{F}})+\sskewer(y+r,\cev{\mathbf{F}}),\,r\ge 0)$ is $\mathbb{Q}^{\alpha,\theta}_{\pi}$, which only 
  depends on $\sskewer(y,\cev{\mathbf{F}})+\sskewer(y,\mathbf{F})=\pi_1+\pi_2=:\pi$. This completes the proof.  
\end{proof}

\begin{proposition}[Continuity in the initial state]\label{prop:initialcont:at}
        For a sequence $\pi_n \to \pi_{\infty}$ in  $(\cM^a,d_{\rm TV})$, there is the weak convergence 
	$\BQ^{\alpha,\theta}_{\pi_n} \to \BQ^{\alpha,\theta}_{\pi_{\infty}}$ 
	in the sense of finite-dimensional distributions on $(\cM^a,d_{\cM})$.  
\end{proposition}
\begin{proof} By Proposition \ref{prop:initial:sp} we have $\BQ^{\alpha,0}_{\pi_n} \to \BQ^{\alpha,0}_{\pi_{\infty}}$. 
  Trivially, also $\BQ^{\alpha,\theta}_{0} \to \BQ^{\alpha,\theta}_{0}$, and we conclude referring to Corollary \ref{cor:theta:additive} and 
  by noting that addition is a continuous operation on $\mathcal{C}([0,\infty),(\cM^a,d_{\cM}))$.
\end{proof}

Finally, we have all the ingredients to apply the proof of Proposition \ref{prop:0-SMP} to establish the strong Markov property under
$\mathbb{Q}_\mu^{\alpha,\theta}$.

\begin{proposition}[Strong Markov property] \label{prop:theta-SMP}
	For a probability measure $\mu$ on $\cM^a$, let $(\pi^z,\,z\ge 0)\sim \BQ^{\alpha,\theta}_{\mu}$. Denote 
	its right-continuous natural filtration by $(\cF^y,\, y\ge 0)$.  
	Let $Y$ be an a.s.\@ finite  $(\cF^y ,\,y\ge 0)$-stopping time. 
	Then given $\cF^Y$, the process $(\pi^{Y+y},\, y\ge 0)$ has conditional distribution 
	$\mathbb{Q}_{\pi^Y}^{\alpha,\theta}$. 
\end{proposition}

\subsection{Proofs of Theorems \ref{thm:sp}, \ref{thm:mass:sp}, \ref{thm:at:const} and Propositions \ref{prop:scaling:sp}--\ref{prop:branching}}\label{sec:pfat}

We now pull the threads together and prove the remaining results about ${\tt SSSP}(\alpha,\theta)$ stated in the introduction, i.e.\ the Hunt and path-continuity properties, transition kernels and the superskewer construction (Theorems \ref{thm:sp} and \ref{thm:at:const}, when $\theta>0$), the self-similarity (Proposition \ref{prop:scaling:sp}) and additivity property (Proposition \ref{prop:branching}) and the ${\tt BESQ}(2\theta)$ total mass process (Theorem \ref{thm:mass:sp}).

\begin{proof}[Proof of Theorem \ref{thm:at:const} and of the $\theta>0$ case of Theorem \ref{thm:sp}]
These results follow by the same arguments as those used to prove Theorem \ref{thm:a0:const} the $\theta=0$ case of Theorem \ref{thm:sp} in Section \ref{sec:pfa0}, replacing the intermediate results with corresponding results in the $\theta>0$ case, namely Proposition~\ref{cor:at:marginals}, Lemma \ref{lm:Lusin}, Corollary \ref{cor:at:continuity}, Proposition \ref{prop:theta-MP} and 
Proposition~\ref{prop:theta-SMP}. 
%
\end{proof}

\begin{proof}[Proof of Proposition \ref{prop:scaling:sp}]
	Let $\cev{\mathbf{F}}\sim\PRM[\frac{\theta}{\alpha}{\tt Leb}\otimes\overline{\nu}_{\perp {\rm cld}}^{(\alpha)}]$ and $\mathbf{F}\sim\mathbf{Q}_\pi^{\alpha,0}$ be independent. 
	Then we can write  
	$$\pi^y = \sskewer(y,\cev{\mathbf{F}})+\sskewer(y,\mathbf{F}),\quad y\ge 0.$$ 	
	Recall the scaling operator defined in \eqref{eq:min-cld:xform_def}. 
	We define 
	\[\cev{\mathbf{F}}_c= \sum_{{\rm points}\;(s,V,X)\;{\rm of}\;\cev{\mathbf{F}}} \delta(cs, c\scaleHA V, c\odot_{\rm stb}^{1+\alpha} X),\]
	where $c\odot_{\rm stb}^{1+\alpha} X= (cX(t/c^{1+\alpha}),\,t\ge 0)$.  
	Similarly we define $\mathbf{F}_c$. 
	Notice that, for any $V\in \mathcal{V}$ and $X\in \mathcal{D}$, we have the identity
	\[
	\sskewer(y, c\scaleHA V, c\odot_{\rm stb}^{1+\alpha} X) = c \cdot \sskewer(y/c, V, X), \quad y\ge 0. 
	\]
	It follows that 
	\[
	c\pi^{y/c} =  \sskewer(y, \cev{\mathbf{F}}_c)+\sskewer(y, \mathbf{F}_c),  \quad y\ge 0.
	\]
	By Lemma~\ref{lm:scaling1} and Lemma~\ref{lem:min_cld:scaling} respectively, we have $\mathbf{F}_c\sim \mathbf{Q}_{c\pi}^{\alpha,0}$ and $\cev{\mathbf{F}}_c\sim \PRM[\frac{\theta}{\alpha}{\tt Leb}\otimes\overline{\nu}_{\perp {\rm cld}}^{(\alpha)}]$, which implies that $(c\pi^{y/c},\,y\ge 0)\sim \BQ^{\alpha,\theta}_{c\pi}$. 
\end{proof}

\begin{proof}[Proof of Proposition \ref{prop:branching}]
 This now follows from Corollary \ref{cor:theta:additive}.
\end{proof}

\begin{proof}[Proof of Theorem \ref{thm:mass:sp}]
	Consider independent $(\pi_1^z,\,z\ge 0)\sim \mathbb{Q}^{\alpha,\theta}_0$ and $(\pi_2^z,\,z\ge 0)\sim \mathbb{Q}^{\alpha,0}_{\pi}$,
	as in Corollary \ref{cor:theta:additive}, for any $\pi\in\mathcal{M}^a$.  
	Then the process $(\pi^z = \pi_1^z + \pi_2^z,\, z\ge 0)$ is an  $\mathtt{SSSP}_{\pi}(\alpha,\theta)$. Denote 
	its right-continuous natural filtration by $(\cF^y,\, y\ge 0)$.  
	
	Proposition~\ref{prop:0:mass} shows that $(\|\pi_2^z\|,\,z\ge 0)$ is a $\BESQ_{\|\pi\|}(0)$. 
	For any fixed $y\ge 0$, by Proposition \ref{prop:theta:entrance} we deduce that $\|\pi_1^y\|\sim \mathtt{Gamma}(\theta, 1/2y)$, which is the marginal distribution of a $\BESQ_{0}(2\theta)$ at time $y$ (see e.g.\ \cite[Corollary XI.(1.4)]{RevuzYor}). 
	Using the additivity property of squared Bessel processes (see e.g.\ \cite[Theorem XI.(1.2)]{RevuzYor}), we deduce that $\|\pi^y\|$ has the marginal distribution of a $\BESQ_{\|\pi\|}(2\theta)$ at time $y$. 
	
	Next, by the marginal distributions described above and the Markov property of  $(\pi^z,\,z\ge 0)$, Proposition~\ref{prop:theta-MP}, we deduce that the total-mass process $(\|\pi^z\|,\,z\ge 0)$ possesses the Markov property with respect to the filtration $(\mathcal{F}^y, y\ge 0)$. 
	
	Finally, since the total-mass function is continuous on $(\cM^a, d_{\cM})$, we deduce by the path-continuity of an $\mathtt{SSSP}_{\pi}(\alpha, \theta)$ that the total-mass process has a.s.\ continuous paths. This completes the proof.  	
\end{proof}

\subsection{An emigration-immigration approach to ${\tt SSSP}_{b\delta(x)}(\alpha,0)$}

\begin{corollary}\label{cor:emimm}
 Fix $x\in [0,1]$. For $\mathbf{f}\!\sim\!{\tt BESQ}_b(-2\alpha)$ and $(\cev{\pi}^y,y\!\ge\! 0)\!\sim\!{\tt SSSP}_0(\alpha,\alpha)$ independent, let $\pi^y\!:=\!\mathbf{f}(y)\delta(x)\!+\!\cev{\pi}^y$, $0\!\le\! y\!\le\!\zeta(\mathbf{f})$, and conditionally given $(\cev\pi^y,y\!\in\![0,\zeta(\mathbf{f})])$ with $\cev{\pi}^{\zeta(\mathbf{f})}\!=\!\lambda$, let $(\pi^{\zeta(\mathbf{f})+z},z\!\ge\! 0)\sim{\tt SSSP}_\lambda(\alpha,0)$. Then $(\pi^y,y\!\ge\! 0)\!\sim\!{\tt SSSP}_{b\delta(x)}(\alpha,0)$.
\end{corollary}
See \cite[Proposition 3.15]{Paper1-2} for a related result for interval partition evolutions. Recalling from Theorem \ref{thm:mass:sp} that ${\tt SSSP}(\alpha,\alpha)$ has ${\tt BESQ}(2\alpha)$ total mass, we interpret ${\tt BESQ}(-2\alpha)$ here as a process with
emigration of mass at rate $-2\alpha$. This is compensated by immigration at rate $2\alpha$ provided by ${\tt SSSP}(\alpha,\alpha)$ during the period $[0,\zeta(\mathbf{f})]$ of emigration. In fact, this corollary generalizes to a connection between ${\tt SSSP}(\alpha,\theta)$ and ${\tt SSSP}(\alpha,\theta\!-\!\alpha)$ for all $\theta\ge \alpha$; we leave the details to the reader.
\begin{proof}
 From independent $\mathbf{f}\sim{\tt BESQ}_b(-2\alpha)$ and $\mathbf{V}\sim{\tt PRM}({\tt Leb}\otimes\nu_{\tt BESQ}^{(-2\alpha)}\otimes\Unif)$, we can construct both 
  $\widehat{\mathbf{V}}:=\delta(0,\mathbf{f},x)+\mathbf{V}|_{[0,T^{-\zeta(\mathbf{f})}]}\sim\mathbf{Q}_{b,x}^{(\alpha)}$ 
  as in \eqref{eq:clade_law}
  and $\mathbf{F}_{\perp}\sim{\tt PRM}({\tt Leb}\otimes\overline{\nu}_{\perp\rm cld}^{(\alpha)})$ 
  as in \eqref{eq:imm_PPP} and Proposition \ref{prop:mark_jumps_min_cld}. We also let $\cev{\mathbf{F}}^\prime=\sum_{\text{points}\;(z,V_z,X_z)\in[0,\zeta(\mathbf{f})]\times\mathcal{V}\times\mathcal{D}\text{\;of\;}\mathbf{F}_{\perp}}\delta(\zeta(\mathbf{f})-z,V_z,X_z)$. 
  
  By Theorem \ref{thm:at:const}, 
  $(\sskewer(y,\cev{\mathbf{F}}^\prime),\,y\in[0,\zeta(\mathbf{f})])$ is distributed as an ${\tt SSSP}_0(\alpha,\alpha)$ 
  stopped at an independent time $Y\stackrel{d}{=}\zeta(\mathbf{f})$. 
  On the other hand, $\mathbf{f}\delta(x)+\overline{\sskewer}(\cev{\mathbf{F}}^{\prime})=\overline{\sskewer}(\widehat{\mathbf{V}},\xi_{\widehat{\mathbf{V}}})\sim{\tt SSSP}_{b\delta(x)}(\alpha,0)$.  
  Finally, by the strong Markov property of $(\pi^y,\,y\ge0) := \overline{\sskewer}(\widehat{\mathbf{V}},\xi_{\widehat{\mathbf{V}}})$ at the stopping time $\life(\mathbf{f})$, we get 
  $(\sskewer(\life(\mathbf{f})\!+\!y,\cev{\mathbf{F}}^\prime),\,y\ge0) \sim {\tt SSSP}_{\lambda}(\alpha,0)$ conditionally given $(\pi^y,\,y\in[0,\life(\mathbf{f})])$ with $\lambda = \pi^{\life(\mathbf{f})}$.
\end{proof}





\section{$(\alpha,\theta)$-Fleming-Viot processes} \label{sec:FV}

Before we explicitly turn to the study of ${\tt FV}(\alpha,\theta)$ with ${\tt PDRM}(\alpha,\theta)$ stationary distribution in Section \ref{sec:FVsubsec},  
by time-changing ${\tt SSSP}(\alpha,\theta)$ and normalising to unit total mass, let us, in Section \ref{sec:pseudostat}, prepare the stationarity 
at the level of ${\tt SSSP}(\alpha,\theta)$, where we observe a certain decoupling of the non-stationary total mass evolution from 
stationary ${\tt PDRM}(\alpha,\theta)$ mass proportions. This was similarly observed in \cite[Theorem 1.5]{Paper1-2} for the type-0 and type-1 interval partition evolutions
that relate to the cases $\theta=\alpha$ and $\theta=0$ here. 

\subsection{Pseudo-stationarity of self-similar $(\alpha,\theta)$-superprocesses}\label{sec:pseudostat}


\begin{theorem}[Pseudo-stationarity for $\mathtt{SSSP}(\alpha,\theta)$]\label{thm:pseudostat}
	Let $\theta\!\ge\! 0$. 
	Consider an independent pair $(Z,\ol\Pi)$, where $\ol\Pi\!\sim\!\PDRM[\alpha,\theta]$ and $Z\!=\!(Z(y),y\!\ge\! 0)$ is a \BESQ[2\theta] with an arbitrary initial distribution. Let $(\pi^y,\,y\geq 0)$ be an $\mathtt{SSSP}(\alpha,\theta)$ with $\pi^0\stackrel{d}{=}Z(0)\ol\Pi$. Then for each fixed $y\ge 0$ we have $\pi^y \stackrel{d}{=} Z(y)\ol\Pi$.
\end{theorem}
 
In Theorem \ref{thm:pseudostat:strong-Y} we will generalize this result from fixed times to certain stopping times. For now, to prove Theorem~\ref{thm:pseudostat}, we first consider a special case for the law of $Z(0)$. 

\begin{proposition}\label{prop:pseudostat:gamma}
	In the setting of Theorem \ref{thm:pseudostat} with $\theta > 0$, suppose that $Z(0)\sim \GammaDist[\theta,\rho]$ for some $\rho\in(0,\infty)$. Then, for every fixed $y\ge 0$, we have $\pi^y \stackrel{d}{=} (2y\rho+1)Z(0)\scaleI\ol\Pi$.
\end{proposition}

\begin{proof}
	
	Let $(\lambda^z,\,z\ge0)$ be an $\mathtt{SSSP}(\alpha,\theta)$ starting from $0$. 
	Fix $y\ge 0$. 
	By Proposition \ref{prop:theta:entrance}, we have $\lambda^{1/2\rho}\stackrel{d}{=}Z(0)\ol\Pi$ and $\lambda^{y+1/2\rho}\stackrel{d}{=}(2y\rho +1)Z(0)\ol\Pi$, using the fact that $(2y\rho +1)Z(0) \sim\GammaDist[\theta,\rho/(2yp+1)]$. 
	
	Since it follows from the simple Markov property that the process $(\widetilde{\lambda}^z:=\lambda^{(1/2\rho)+z}, z\ge 0)$ is an $\mathtt{SSSP}(\alpha,\theta)$ starting from $\lambda^{1/2\rho}\stackrel{d}{=}Z(0)\ol\Pi$, we have 
	$\pi^y \stackrel{d}{=} \widetilde{\lambda}^y \stackrel{d}{=}(2y\rho +1)Z(0)\ol\Pi$.  
	This is the desired statement. 
\end{proof}

\begin{proof}[Proof of Theorem~\ref{thm:pseudostat}]
	First fix $\theta>0$. 
	For every $b\ge 0$, denote by $\mu_b^{\alpha,\theta}$ the distribution of $b\overline{\Pi}$ on $(\mathcal{M}^a,d_\mathcal{M})$. 
	For every bounded continuous function $f$ on $(\cM^a, d_{\cM})$ with $f(0)=0$ and $y\ge 0$, 
	Proposition~\ref{prop:pseudostat:gamma} yields \vspace{-0.1cm}
	\begin{equation*}
	\begin{split}
	& \int_{0}^{\infty}e^{-\rho b}  \Gamma(\theta)^{-1} \rho^{\theta} b^{\theta-1}  
	\mathbb{Q}_{\mu_b^{\alpha,\theta}}^{\alpha,\theta}\left[f\!\left(\pi^y\right)\right] d b \\
	&=  \int_0^{\infty}
	\Gamma(\theta)^{-1} \rho^{\theta} b^{\theta-1}  e^{-\rho b} 
	\mathbf{E}\left[f\!\left((2 y \rho +1)b  \ol\Pi\right)\right] d b \\
	&=   \int_0^{\infty}
	\Gamma(\theta)^{-1} \frac{\rho^{\theta}}{(2 y \rho +1)^{\theta}} c^{\theta-1}  e^{-\rho c/(2y\rho+1)} 
	\mathbf{E}\left[f \!\left(c \ol\Pi\right)\right] d c.\vspace{-0.1cm}
	\end{split}
	\end{equation*}	
	For every $b\ge 0$, let $(Z_b(y), y\ge 0)$ be a $\BESQ_b(2\theta)$, independent of $\ol\Pi$.  
	It is known \cite[Corollary~XI.(1.4)]{RevuzYor} that $Z_0(y)\sim \mathtt{Gamma}(\theta, 1/2y)$ for every $y> 0$. 
	Using this fact and the Markov property of $\BESQ(2\theta)$,  we deduce the identity\vspace{-0.1cm}
	\begin{equation*}
	\begin{split}
	& \int_0^{\infty}
	\Gamma(\theta)^{-1} \frac{\rho^{\theta}}{(2 y \rho +1)^{\theta}} c^{\theta-1}  e^{-\rho c/(2y\rho+1)} 
	\mathbf{E}[f (c \ol\Pi)] dc\\
	&= \mathbf{E}\Big[\mathbf{E}\left[f\!\left(\left.Z_0(y+1/2\rho) \ol\Pi\right) \right| Z_0(1/2\rho)\right]\Big]  \\
	&= 	\int_{0}^{\infty}  e^{-\rho b} \Gamma(\theta)^{-1} \rho^{\theta} b^{\theta-1} 
	 \mathbf{E}[f (Z_b(y) \ol\Pi)] db. \vspace{-0.1cm}
		\end{split}
	\end{equation*}  
	By the uniqueness theorem for Laplace transforms, we find that\vspace{-0.1cm}
	\[
	\mathbb{Q}_{\mu_b^{\alpha,\theta}}^{\alpha,\theta}[f (\pi^y)] = \mathbf{E}[f (Z_b(y) \ol\Pi)], \quad \text{for Lebesgue-a.e.\ } b>0.\vspace{-0.1cm} 
	\]
	Since the map $b\mapsto b\ol\Pi$ is $d_{\rm TV}$-continuous, we have by Proposition~\ref{prop:initialcont:at} that the map $b\mapsto \mathbb{Q}_{\mu_b^{\alpha,\theta}}^{\alpha,\theta}[f (\pi^y)]$
	is also continuous. So we conclude that 
	$\mathbb{Q}_{\mu_b^{\alpha,\theta}}^{\alpha,\theta}[f (\pi^y)] = \mathbf{E}[f (Z_b(y) \ol\Pi)]$ for every $b\ge 0$. 
	This identifies the distribution of $\pi^y$ because bounded continuous functions separate points in $(\cM^a, d_{\cM})$, due to its Lusin property. 
	This completes the proof for $\theta>0$.
	
	For any sequence $\theta_k\!\downarrow\!\theta_0\!:=\!0$, consider independent $\overline{\Pi}^\prime_k\!\sim\!{\tt PDRM}(\alpha,\theta_k)$,
	$k\!\ge\! 0$, also independent of a sequence $B_k\!\sim\!{\tt Beta}(\theta_k,1)$, $k\!\ge\! 1$, with $B_k\!\downarrow\! 0\!=:\!B_0$ a.s.. By 
	\eqref{eq:at_a0}, we have $\overline{\Pi}_k:=B_k\overline{\Pi}^\prime_k+(1-B_k)\overline{\Pi}_0^\prime\sim{\tt PDRM}(\alpha,\theta_k)$, 
	$k\ge 0$, and $d_{\rm TV}(\overline{\Pi}_k,\overline{\Pi}_0)\rightarrow 0$ a.s.. For $Z^{(k)}_b\sim{\tt BESQ}_b(2\theta_k)$, $k\ge 0$, we also have $Z^{(k)}_b(y)\rightarrow Z_b^{(0)}(y)$ in
	distribution, so the established time-$y$ distribution for $\theta_k>0$ converges to the claimed time-$y$ distribution in the case $\theta=0$,
	weakly in $(\mathcal{M}^a,d_\mathcal{M})$. On the other hand, for all continuous 
	$\phi\colon[0,1]\rightarrow[0,\infty)$
	\begin{align*}
	  &\mathbb{Q}_{\mu_b^{\alpha,\theta_k}}^{\alpha,\theta_k}\bigg[\exp\bigg(-\int_{[0,1]}\phi(x)\pi^y(dx)\bigg)\bigg]\\
	  &=\mathbb{Q}_{\mu_b^{\alpha,\theta_k}}^{\alpha,0}\bigg[\exp\bigg(-\int_{[0,1]}\phi(x)\pi^y(dx)\bigg)\bigg]
	      \mathbb{Q}_{0}^{\alpha,\theta_k}\bigg[\exp\bigg(-\int_{[0,1]}\phi(x)\pi^y(dx)\bigg)\bigg]\\
	  &\rightarrow \mathbb{Q}_{\mu_b^{\alpha,0}}^{\alpha,0}\bigg[\exp\bigg(-\int_{[0,1]}\phi(x)\pi^y(dx)\bigg)\bigg],
        \end{align*}
        by Proposition \ref{prop:initial:sp} for the first factor, using $d_{\rm TV}(b\overline{\Pi}_k,b\overline{\Pi}_0)\rightarrow 0$ a.s., and by the monotone convergence theorem (suitably coupled by thinning 
        $\cev{\mathbf{F}}_{\theta_k}\sim{\tt PRM}(\frac{\theta_k}{\alpha}{\tt Leb}\!\otimes\!\overline{\nu}_{\perp\rm cld}^{(\alpha)})$ 
        as $\theta_k\!\downarrow\! 0$) for the second factor. By the uniqueness theorem for Laplace functionals, $\mathbb{Q}_{\mu_b^{\alpha,0}}^{\alpha,0}(\pi^y\in\cdot\,)$ is the claimed time-$y$
        distribution.
\end{proof}

	We can further strengthen the pseudo-stationarity of Theorem~\ref{thm:pseudostat}. For the purpose of the following results, we write $\|\boldsymbol{\pi}\| := (\|\pi^y\|,\,y\ge0)$.

\begin{proposition}\label{prop:pseudostat:strong-y}
		Let $\theta\ge 0$. 
	Consider an independent pair $(Z,\ol\Pi)$, where $\ol\Pi\sim\PDRM[\alpha,\theta]$ and $Z=(Z(y),y\ge 0)$ is a \BESQ[2\theta] with an arbitrary initial distribution. Let $(\pi^y,\,y\geq 0)$ be an $\mathtt{SSSP}(\alpha,\theta)$ with $\pi^0\stackrel{d}{=}Z(0)\ol\Pi$. 
	Then for any $y\ge 0$ and non-negative measurable functions $h \colon (\cM_1^a,d_{\cM} )\to [0,\infty)$ and $\eta\colon \mathcal{C}([0,\infty), [0,\infty))\to [0,\infty)$,  we have
	\[
	\mathbf{E}\!\left[
	\eta(\|\boldsymbol{\pi}\|) \, \mathbf{1}\!\left\{ \pi^{y}\!\ne\! 0 \right\}\,
	h\!\left(   
	\left\| \pi^{y} \right\|^{-1} \pi^{y}
	\right) 
	\right]
	= \mathbf{E}\left[\eta(Z)\, \mathbf{1}\!\left\{ Z(y)\!\ne\! 0 \right\}\right] \mathbf{E} [h(\ol\Pi)].
\]
\end{proposition} 
\begin{proof} We adapt the proof of \cite[Lemma 4.7]{Paper1-2}. Denote by $\mu_b$ the distribution of $b\overline{\Pi}$ and consider 
  $0\!\le\! y_1\!<\!\cdots\!<\!y_n\!=\!y$ and $f_1,\ldots,f_n\colon[0,\infty)\!\rightarrow\![0,\infty)$. 
  Then the Markov property and pseudo-stationarity of $(\pi^y,y\!\ge\! 0)$ at $y_1$ yield 
  \begin{align*}
   &\mathbf{E}\Bigg[\prod_{j=1}^nf_j\big(\|\pi^{y_j}\|\big)\mathbf{1}\!\left\{ \pi^{y}\!\ne\! 0 \right\}h\!\left(\left\| \pi^{y} \right\|^{-1} \pi^{y}\right)\Bigg]\\ 
   &=\mathbf{E}\Bigg[f_1\big(\|\pi^{y_1}\|\big)\mathbb{Q}_{\pi^{y_1}}^{\alpha,\theta}\!\Bigg[\prod_{j=2}^n\!f_j\big(\|\pi^{y_j-y_1}\|\big)\mathbf{1}\!\left\{ \pi^{y-y_1}\!\ne\! 0 \right\}h\!\left(\left\| \pi^{y-y_1} \right\|^{-1} \!\pi^{y-y_1}\right)\!\Bigg]\Bigg]\\
    &=\mathbf{E}\Bigg[f_1\big(Z(y_1)\big)\mathbb{Q}_{\mu_{Z(y_1)}}^{\alpha,\theta}\!\Bigg[\prod_{j=2}^n\!f_j\big(\|\pi^{y_j-y_1}\|\big)\mathbf{1}\!\!\left\{ \pi^{y-y_1}\!\ne 0\! \right\}\!h\!\left(\!\left\| \pi^{y-y_1} \right\|^{-1}\! \pi^{y-y_1}\!\right)\!\Bigg]\Bigg]\!.
  \end{align*}
  For $n=1$, we have $y=y_1$, so $\pi^{y-y_1}=\pi^0\sim b\overline{\Pi}$ under $\mathbb{Q}_{\mu_b}^{\alpha,\theta}$, for all $b\ge 0$, so 
  $$\mathbb{Q}_{\mu_{b}}^{\alpha,\theta}\!\left[\mathbf{1}\!\left\{ \pi^{y-y_1}\!\ne\! 0 \right\}h\!\left(\!\left\| \pi^{y-y_1} \right\|^{-1}\! \pi^{y-y_1}\!\right)\right]=\mathbf{1}\!\left\{ b\!\ne\! 0 \right\}\mathbf{E}[h(\overline{\Pi})],$$
  as required, and for $n\ge 2$ we obtain inductively that
  \begin{align*}
   &\mathbf{E}\Bigg[\prod_{j=1}^nf_j\big(\|\pi^{y_j}\|\big)\mathbf{1}\!\left\{ \pi^{y}\ne 0 \right\}h\!\left(\!\left\| \pi^{y} \right\|^{-1} \pi^{y}\!\right)\!\Bigg]\\ 
   &=\mathbf{E}\Bigg[f_1\big(Z(y_1)\big)\,\mathbb{Q}_{\mu_{Z(y_1)}}^{\alpha,\theta}\!\Bigg[\prod_{j=2}^nf_j\big(\|\pi^{y_j-y_1}\|\big)\mathbf{1}\!\left\{\|\pi^{y-y_1}\|\!\ne\! 0 \right\}\!\Bigg]\mathbf{E}[h(\overline{\Pi})]\Bigg]\\
   &=\mathbf{E}\Bigg[\prod_{j=1}^nf_j\big(Z(y_j)\big)\mathbf{1}\!\left\{Z(y)\ne 0 \right\}\!\Bigg]\mathbf{E}[h(\overline{\Pi})],
   \end{align*}
   where the last step applied the Markov property of the ${\tt BESQ}(2\theta)$ total mass process at $y_1$. The argument is easily adapted by
   further applications of the Markov properties at $y=y_n$ to extend the product by further terms $f_j\big(\|\pi^{y_j}\|\big)$ for $y_j>y$, 
   $n+1\le j\le n+m$. The application of a monotone class theorem completes the proof. 
  \end{proof}

\begin{theorem}\label{thm:pseudostat:strong-Y}
 Let $(\pi^y,y\!\geq\! 0)$ be an $\mathtt{SSSP}(\alpha,\theta)$ with $\theta\!\ge\! 0$ and  $\pi^0\stackrel{d}{=}B\ol\Pi$, where $B$ is an arbitrary non-negative random variable and $\ol\Pi\sim\PDRM[\alpha,\theta]$ is independent of $B$. 
	Let $Y$ be an a.s.\ finite $\sigma(\|\pi^y\|, y\ge 0)$-measurable random time. 
	Then for any measurable functions $h \colon (\cM_1^a,d_{\cM} )\to [0,\infty)$ and $\eta\colon \mathcal{C}([0,\infty), [0,\infty))\to [0,\infty)$,  we have
	\[
	\mathbf{E}\left[
	\eta(\|\boldsymbol{\pi}\|) \, \mathbf{1}\!\left\{ \pi^{Y}\!\!\ne\! 0 \right\}\,
	h\!\left(   
	\left\| \pi^{Y} \right\|^{-1} \pi^{Y}
	\right) 
	\right]
	= \mathbf{E}\left[\eta(\|\boldsymbol{\pi}\|) \, \mathbf{1}\!\left\{ \pi^{Y}\!\!\ne\! 0 \right\}\right] \mathbf{E} [h(\ol\Pi)].
\]
\end{theorem}
\begin{proof}
We use the standard dyadic approximation. Let $Y_n = 2^{-n} \lceil 2^n Y\rceil$, which is a.s.\ finite and decreases to $Y$ as $n\to \infty$. 
Then applying Proposition~\ref{prop:pseudostat:strong-y}, for any bounded measurable function $\eta$, bounded continuous function $h$ and $k\ge 1$, we have  
\begin{align*}
&	\mathbf{E}\left[\eta(\|\boldsymbol{\pi}\|) \, \mathbf{1}\!\left\{ \pi^{Y_n}\ne 0; Y_n = k2^{-n} \right\}\,
			h\!\left(  \left\| \pi^{Y_n} \right\|^{-1} \pi^{Y_n}\right) 	\right]\\
&= 	\mathbf{E}\left[\eta(\|\boldsymbol{\pi}\|) \,  \mathbf{1}\!\left\{ Y_n = k2^{-n} \right\}\,\mathbf{1}\!\left\{ \pi^{k2^{-n}}\ne 0\right\}
			h\!\left(  \left\| \pi^{k2^{-n}} \right\|^{-1} \pi^{k2^{-n}}\right) 	\right]\\
&=\mathbf{E}\left[\eta(\|\boldsymbol{\pi}\|) \,  \mathbf{1}\!\left\{ Y_n = k2^{-n} \right\}\,\mathbf{1}\!\left\{ \pi^{Y_n}\ne 0\right\}\right]
			\mathbf{E}\left[h\!\left(\ol\Pi\right)\right].
\end{align*}
Summing over $k\ge 0$ leads to 
$$	\mathbf{E}\!\left[\eta(\|\boldsymbol{\pi}\|)  \mathbf{1}\!\left\{ \pi^{Y_n}\!\ne\! 0 \right\}
			h\!\left(  \left\| \pi^{Y_n} \right\|^{-1}\! \pi^{Y_n}\!\right) 	\right]
=\mathbf{E}\!\left[\eta(\|\boldsymbol{\pi}\|)  \mathbf{1}\!\left\{ \pi^{Y_n}\!\ne\! 0\right\}\right]
			\mathbf{E}\!\left[h\!\left(\ol\Pi\right)\right]\!.
$$
Letting $n\to \infty$, we deduce by the path-continuity of $(\pi^y)$ and the dominated convergence theorem that the last identity still holds when $Y_n$ is replaced by $Y$. 
This proves the assertion in the theorem for bounded measurable $\eta$ and bounded continuous function $h$. We complete the proof by a 
monotone class theorem. 
\end{proof}

\subsection{De-Poissonization and $(\alpha, \theta)$-Fleming--Viot processes}\label{sec:FVsubsec}

$\!$Fix $\alpha\in(0,1)$ and $\theta\ge 0$. Let $\boldsymbol{\pi}:= (\pi^y,\,y\ge 0)$ be an ${\tt SSSP}_\pi(\alpha,\theta)$ for some $\pi\in\mathcal{M}^a\setminus\{0\}$. 
Recall the time-change function defined by \eqref{eq:tau-pi}: 
\begin{equation*}
\rho_{\boldsymbol{\pi}}(u):= \inf \left\{ y\ge 0\colon \int_0^y \|\pi^z\|^{-1} d z>u \right\}, \qquad u\ge 0.
\end{equation*}

Since by Theorem~\ref{thm:mass:sp} the total mass of $\boldsymbol{\pi}$ evolves according to $\BESQ(2\theta)$,   
by rewriting known results on squared Bessel processes in our setting, we have the following lemma. 
\begin{lemma}[{\cite[p.\ 314-5]{GoinYor03}}]\label{lem:tau-beta-prop}
	The time-change function $\rho_{\boldsymbol{\pi}}$  is continuous and strictly increasing.  Moreover, with the usual convention $\inf\emptyset=\infty$, 
	\begin{equation*}
	\lim_{u\uparrow \infty} \rho_{\boldsymbol{\pi}}(u) = \inf \{y>0 \colon \pi^y = 0 \}. 
	\end{equation*}
The limit is a.s.\ finite if $\theta\in [0,1)$ and a.s.\ infinite if $\theta\ge 1$. 
\end{lemma}

We define an $\cM_1^a$-valued process $\overline{\boldsymbol{\pi}}:= (\ol{\pi}^u,\,u\ge 0)$ via the following so-called \emph{de-Poissonization}: 
\[
\ol{\pi}^u:= \left\| \pi^{\rho_{\boldsymbol{\pi}}(u)} \right\|^{-1} \pi^{\rho_{\boldsymbol{\pi}}(u)},\qquad u\ge 0. 
\]
Recall from Definition~\ref{def:FV} that we have referred to the de-Poissonized process $\overline{\boldsymbol{\pi}}$ as an \emph{$(\alpha, \theta)$-Fleming--Viot process}, or ${\tt FV}(\alpha,\theta)$. 
We have by Lemma~\ref{lem:tau-beta-prop} and the path-continuity of an ${\tt SSSP}_\pi(\alpha,\theta)$ that the sample paths of an ${\tt FV}(\alpha,\theta)$  are still continuous. 

For any $c>0$, let $\boldsymbol{\pi}_c= (\pi_c^y=c \pi^{y/c},\, y\ge 0)$. For every $u\ge 0$, since $\rho_{\boldsymbol{\pi}_c}(u)= c \rho_{\boldsymbol{\pi}}(u)$, we have the equality between the de-Poissonized processes 
\[
\left\| \pi^{\rho_{\boldsymbol{\pi}}(u)} \right\|^{-1} \pi^{\rho_{\boldsymbol{\pi}}(u)}
= \left\| \pi_c^{\rho_{ \boldsymbol{\pi}_c}(u)} \right\|^{-1}  \pi_c^{\rho_{\boldsymbol{\pi}_c}(u)}. 
\]
Therefore, for any $\boldsymbol{\lambda}:= (\lambda^y, y\ge 0) \sim {\tt SSSP}(\alpha,\theta)$ with $\lambda^0 \stackrel{d}{=} c\pi^0$ for some constant $c>0$, then their associated de-Poissonized processes have the same distribution. 
Indeed, this results from the following two facts: due to the self-similarity, Proposition~\ref{prop:scaling:sp}, $\boldsymbol{\lambda}$ and the rescaled process $\boldsymbol{\pi}_c$ have the same law and so do their de-Poissonized processes; on the other hand, we have seen that $\boldsymbol{\pi}$ has the same de-Poissonized process as $\boldsymbol{\pi}_c$.  
As a result, the law of a $\mathtt{FV}(\alpha,\theta)$ is characterized by its initial value. 
For any probability measure $\ol\mu$ on $\cM_1^a$, we can denote by  $\ol\BQ^{\alpha,\theta}_{\ol\mu}$ the law on $\cC([0,\infty),(\cM_1^a,d_{\cM}))$ of the de-Poissonized process of $\boldsymbol{\pi}\sim \BQ^{\alpha,\theta}_{\ol\mu}$. 

\begin{proposition}[Strong Markov property]\label{prop:FV:SMP}
	For a probability measure $\ol\mu$ on $\cM_1^a$, let $(\ol\pi^u,\,u\ge 0)\sim \ol\BQ^{\alpha,\theta}_{\ol\mu}$. Denote 
	its right-continuous natural filtration by $(\ol\cF^u,\, u\ge 0)$.  
	Let $U$ be an a.s.\@ finite $(\ol\cF^u ,\,u\ge 0)$-stopping time. 
	Then given $\ol\cF^U$, the process $(\ol\pi^{U+u},\, u\ge 0)$ has conditional distribution 
	$\ol{\mathbb{Q}}_{\ol\pi^U}^{\alpha,\theta}$. 
\end{proposition}
\begin{proof}
Let $\boldsymbol{\pi}=(\pi^y ,y\ge 0)\sim \BQ^{\alpha,\theta}_{\ol\mu}$ with its right-continuous natural filtration denoted by $(\cF^y ,\,y\ge 0)$. We may assume that $(\ol\pi^u,\,u\ge 0)$ is the de-Poissonized process of $\boldsymbol{\pi}$ and then we have $\ol\cF^u \subseteq \cF^{\rho_{\boldsymbol{\pi}}(u)}$. Since $\rho_{\boldsymbol{\pi}}$ is $(\ol\cF^u)$-adapted, continuous and strictly increasing, we have by \cite[Proposition 7.9]{Kallenberg} that $Y:= \rho_{\boldsymbol{\pi}}(U)$ is an $(\cF^y)$-stopping time and $\ol\cF^{U}\subseteq\cF^Y$. 
Write $\boldsymbol{\pi}'= (\pi^{Y+y}, y\ge 0)$ for the shifted process. 
Then for $u\ge 0$, we have \vspace{-0.1cm}
\[
\ol\pi^{U+u} =\|\pi^{\rho_{\boldsymbol{\pi}}(U+u)} \|^{-1} \pi^{\rho_{\boldsymbol{\pi}}(U+u)} =\|\pi^{Y+ \rho_{\boldsymbol{\pi}'}(u)} \|^{-1} \pi^{Y+ \rho_{\boldsymbol{\pi}'}(u)}.\vspace{-0.1cm} 
\]
This is to say, $(\ol\pi^{U+u},\, u\ge 0)$ is the de-Poissonized process of $\boldsymbol{\pi}'$. 
Now the desired result follows from the strong Markov property of $\boldsymbol{\pi}$, Proposition~\ref{prop:theta-SMP}. 
\end{proof}

\begin{proof}[Proof of Theorem~\ref{thm:dP-sp}] 
	 As in the proof of Theorem~\ref{thm:sp}, we shall check the following properties to prove that it is a path-continuous Hunt process.  
	The space $(\mathcal{M}^a_1, d_{\cM})$ is Lusin by Lemma~\ref{lm:Lusin}. 
	By Lemma~\ref{lem:tau-beta-prop} and the path-continuity of an ${\tt SSSP}(\alpha,\theta)$, the sample paths of the time-changed process, the ${\tt FV}(\alpha,\theta)$,  are still continuous. 
	Moreover, by \cite[Theorem 2.6]{Helland78}, our time-change operations are continuous maps from Skorokhod space to itself. We deduce that the semi-group of the time-changed process is still measurable.  
	Finally, Proposition~\ref{prop:FV:SMP} gives the required strong Markov property.  
	
	We finally verify the stationary distribution. 
For $\ol\Pi\sim \PDRM(\alpha,\theta)$, let $\boldsymbol{\pi}:= (\pi^y,\, y\ge 0)$ be an ${\tt SSSP}(\alpha,\theta)$ starting from $\overline{\Pi}$ and $(\ol\pi^u,\,u\ge 0)$ its de-Poissonized process.  
Applying Theorem~\ref{thm:pseudostat:strong-Y} to $\eta=1$ and $Y= \rho_{\boldsymbol{\pi}}(u)$ with $u\ge 0$, we have, for every measurable function $h\colon \cM_1^a \to [0,\infty)$,\vspace{-0.1cm}  
	\[
   \mathbf{E} [h(\ol\pi^u)]=\mathbf{E}\left[\mathbf{1}\left\{ \pi^{\rho_{\boldsymbol{\pi}}(u)}\ne 0 \right\}
					h \left( \left\| \pi^{\rho_{\boldsymbol{\pi}}(u)} \right\|^{-1} \pi^{\rho_{\boldsymbol{\pi}}(u)}\right) \right]
	=  \mathbf{E} [h(\ol\Pi)],\vspace{-0.1cm}
	\]
	as $\mathbf{P}(\pi^{\rho_{\boldsymbol{\pi}}(u)}\!\neq\! 0)=1$ by Lemma \ref{lem:tau-beta-prop}. This is the desired claim.  		
\end{proof}

\section{Properties of ${\tt FV}(\alpha,\theta)$ and ${\tt SSSP}(\alpha,\theta)$}\label{sec:properties}

In this section we discuss and prove the three further properties of ${\tt FV}(\alpha,\theta)$ stated in Section \ref{intro:shiga}. We prove
Theorems \ref{thm:property1}--\ref{thm:property2}, respectively, in Subsections \ref{sec:property1}--\ref{sec:property2} and develop the
connections to Shiga \cite{Shiga1990} in Subsection \ref{sec:property3}.

\subsection{Exceptional times with finitely many atoms}\label{sec:property1}

Recall notation $N(\pi)=\#\{x\!\in\![0,1]\colon\pi(\{x\})\!>\!0\}$ for the number of atoms of $\pi\!\in\!\mathcal{M}^a$ and let $1\!\le\! n\!<\!\infty$. Theorem \ref{thm:property1}
claims that there are, with positive probability, exceptional times at which a ${\tt FV}(\alpha,\theta)$ has only $n$ atoms if and only if 
$\theta+n\alpha<1$. This may be surprising since there are infinitely many atoms with probability 1 under the stationary distribution ${\tt PDRM}(\alpha,\theta)$. Indeed, the relevant part of the transition kernel of the underlying ${\tt SSSP}(\alpha,\theta)$ adds in scaled 
${\tt PDRM}(\alpha,\alpha)$ components, which each have infinitely many atoms with probability 1. 
 
\begin{proof}[Proof of Theorem \ref{thm:property1}] 
  Since the desired property is not affected by either the continuous de-Poissonization time-change $u\mapsto\rho_{\boldsymbol{\pi}}(u)$ or 
  by the normalization to unit mass, it suffices to establish that ${\tt SSSP}(\alpha,\theta)$ visits $A_n:=\{\pi\!\in\!\mathcal{M}^a\colon N(\pi)\!=\!n\}$
  before visiting zero with positive probability if and only if $\theta+n\alpha<1$.
  Therefore, let $\pi\in\mathcal{M}^a\setminus\{0\}$, $\mathbf{F}_\pi\sim\mathbf{Q}_{\pi}^{\alpha,0}$ and $\cev{\mathbf{F}}_\theta\sim{\tt PRM}(\frac{\theta}{\alpha}{\tt Leb}\otimes\overline{\nu}_{\perp\rm cld}^{(\alpha)})$ so that
  $$\boldsymbol{\pi}=(\pi^y,\,y\ge 0)=\sskewerP(\cev{\mathbf{F}}_\theta)+\sskewerP(\mathbf{F}_\pi)\sim{\tt SSSP}_\pi(\alpha,\theta).$$
  Recall from Theorem \ref{thm:mass:sp} that $(\|\pi^y\|,\,y\ge 0)\sim{\tt BESQ}_{\|\pi\|}(2\theta)$. In particular, total mass visits 0 if and only if $\theta<1$, by Lemma \ref{lem:tau-beta-prop}. 
  To study times $y$ when $\pi^y\in A_n$, we first suppose $\pi^0=\pi\in A_n$, so that 
  $\mathbf{F}_\pi=\sum_{1\le i\le n}\delta(\mathbf{V}_i,\mathbf{X}_i)$. 
  Then the contribution of $\cev{\mathbf{F}}_\theta$ to the total mass  is a ${\tt BESQ}_0(2\theta)$, while each
  $(\mathbf{V}_i,\mathbf{X}_i)$ has a left-most spindle $\mathbf{f}_i$, and by Corollary \ref{cor:emimm}, the remaining
  spindles contribute a ${\tt BESQ}_0(2\alpha)$ to total mass during $(0,\zeta(\mathbf{f}_i))$. In order for $\pi^y\in A_n$ at $y\in\bigcap_{1\le i\le n}(0,\zeta(\mathbf{f}_i))$, we need the 
  sum of a ${\tt BESQ}_0(2\theta)$ and $n$ independent ${\tt BESQ}_0(2\alpha)$, i.e.\ a ${\tt BESQ}_0(2(\theta+n\alpha))$, to vanish, which happens with 
  positive probability if and only if $\theta+n\alpha<1$. 
  
  Now suppose we have more or fewer than $n$ atoms initially. If more, there is positive probability that any choice of $n$ atoms has lifetimes
  $\zeta(\mathbf{f}_i)$, $1\le i\le n$, greater than all levels $\zeta_j^+$ when the mass of $(\mathbf{V}_j,\mathbf{X}_j)$, $j>n$, vanishes. On this event, the previous argument 
  applies. If there are fewer than $n$ atoms initially, then $\mathbf{P}\{\pi^\epsilon\in A_\infty\}>0$ for all $\epsilon>0$, and by the Markov property, the previous argument 
  applies. This completes the proof when $\theta+n\alpha<1$.
  
  For $\theta+n\alpha\ge 1$, it remains to show that $A_n$ cannot be visited with positive probability even on 
  events not yet considered. To this end, consider the set $A_n^{(1/m)}\subseteq A_n$ that requires all atoms to be of
  size strictly greater than $1/m$. If we do not have $n$ atoms of this size, $A_n^{(1/m)}$ cannot be visited before the stopping time when our measure first includes $n$ 
  such large atoms, and we can apply the strong Markov property. Then the previous argument shows that the subset $A_n^{(1/m)}$ of $A_n$ is not visited almost surely while none of these atoms have reached the end of their lifetime. At the stopping time when the first of them vanishes, we 
  can keep repeating the argument. The differences between the relevant stopping times are stochastically bounded below by the law of the time of the first absorption among $n$ independent $\BESQ_{1/m}(-2\alpha)$ diffusions. Thus, these stopping times tend to infinity a.s., so $A_n^{(1/m)}$ and hence 
  $A_n=\bigcup_{m\ge 1}A_n^{(1/m)}$ are not visited a.s., if $\theta+n\alpha\ge 1$.
\end{proof}

If there are times with precisely $n$ atoms when $\theta+n\alpha<1$, then by the argument above, locally, the set of such times has the structure of the zero set of $\BESQ(2(\theta+n\alpha))$, which is well-known to be the range of a stable subordinator of index $1-(\theta+n\alpha)$ and therefore has Hausdorff dimension $1-(\theta+n\alpha)$. See, for example,  \cite{BertoinSubord}.

\subsection{$\alpha$-diversity}\label{sec:property2}

Theorem \ref{thm:property2} claims that $(\overline{\pi}^u,\,u\ge 0)\sim{\tt FV}_\pi(\alpha,\theta)$ has
  $\alpha$-diversities $\IPLT(\overline{\pi}^u):=\Gamma(1-\alpha)\lim_{h\downarrow 0}h^\alpha\#\{x\in[0,1]\colon\overline{\pi}^u(\{x\})>h\}$ that evolve continuously.
  We develop this via a number of intermediate results about ${\tt SSSP}_\pi(\alpha,0)$ and ${\tt SSSP}_0(\alpha,\theta)$, beginning by strengthening the
  continuity of the superskewer process $\widetilde{\mathbf{V}}$, a ${\tt PRM}\left({\tt Leb}\otimes\nu_{\tt BESQ}^{(-2\alpha)}\otimes\Unif\right)$ stopped at a random time $T\in(0,\infty)$, as studied in Proposition~\ref{prop:PRM:cont}.

\begin{proposition}\label{prop:PRM:diversity} Let $\widetilde{\mathbf{V}}$ be as in Proposition~\ref{prop:PRM:cont} and  $\widetilde{\mathbf{X}}=\xi_{\widetilde{\mathbf{V}}}$. 
	For every $y\ge 0$, define 
	$\pi^y= \sskewer(y,\widetilde{\mathbf{V}},\widetilde{\mathbf{X}})$. 
	Then  a.s.\ the real-valued process $(\IPLT(\pi^y),\, y\ge 0)$ is well-defined and is H\"older-$\gamma$ for every $\gamma\in (0,\alpha/2)$. 
\end{proposition}
\begin{proof}
	\cite[Proposition 3.8]{Paper1-1} shows that the skewer of $\widetilde{\mathbf{N}}=\varphi(\widetilde{\mathbf{V}})$  a.s.\ has total $\alpha$-diversity at every level and is H\"older-$\gamma$ for every $\gamma\in (0,\alpha/2)$. 
	But since interval lengths $f_t(y-\widetilde{\mathbf{X}}(t-))$ of $\skewer\big(y,\widetilde{\mathbf{N}},\widetilde{\mathbf{X}}\big)$ are also the atom sizes of $\sskewer\big(y,\widetilde{\mathbf{V}},\widetilde{\mathbf{X}}\big)$, 
	their (total) diversities coincide. 
\end{proof}

\begin{proposition}\label{prop:0:diversity}
	Let $\pi\in \mathcal{M}^a$ and $\bF\sim \bQ^{\alpha,0}_{\pi}$.  
	Then a.s.\ the real-valued process $(\IPLT(\sskewer(y, \bF)), y> 0)$ is well-defined and is H\"older-$\gamma$ for every $\gamma\in (0,\alpha/2)$. 
	If, in addition, the initial measure $\pi$ has $\alpha$-diversity $\IPLT(\pi)$, then 
	$\lim_{y\downarrow 0}\IPLT(\sskewer(y, \bF))= \IPLT(\pi)$ a.s..
\end{proposition}
\begin{proof}
        Like Proposition~\ref{prop:PRM:diversity}, this is really an assertion about the process of ranked atom sizes. The continuity of (total) 
        $\alpha$-diversity of interval partition valued processes with corresponding interval lengths was noted in 
        \cite[Corollary 1.5]{Paper1-1}. More specifically, results including the H\"older continuity can be found in
        \cite[Corollary 6.19 and Proposition 6.11]{Paper1-1} respectively.
\end{proof}

\begin{proposition}\label{contdiv:gen} The $\alpha$-diversity process of an ${\tt SSSP}_0(\alpha,\theta)$ is well-defined and path-continuous.
\end{proposition}
\begin{proof}
         We use notation $(\cev{\pi}^y,y\!\ge\! 0)\!\sim\!{\tt SSSP}_0(\alpha,\theta)$ of Proposition \ref{prop:cevF:path-continuity}. More precisely, for every $0\!\le\! a\!<\!b\!\le\! \infty$, we define for $\cev{\bF}_\theta\sim{\tt PRM}(\frac{\theta}{\alpha}{\tt Leb}\otimes\umCladeAbar)$
	\[\cev{\pi}_{[a,b)}^y:= \sum_{\text{points } (s,V_s,X_s) \text { of } \cev\bF_{\theta},\, s\in [a,b)} \!\!\!\!\!\sskewer\left(y-|s|  ,V_s,X_s\right), \quad y\ge 0.\] 
	That is, only those clades entering between levels $a$ and $b$ count for the process $\cev{\boldsymbol{\pi}}_{[a,b)}$. In particular, $\cev{\pi}_{[a,b)}^y= 0$ for all $y\le a$. 
	For every $z\ge 0$, let 
	\begin{align*}
	\overline{D}_z^y &:= \Gamma(1-\alpha)\limsup_{h\downarrow 0} h^{\alpha} \#\{ x\in [0,1]\colon \cev{\pi}_{[z,\infty)}^y(\{x\})>h \}, \quad &y\ge 0, \\
	\underline{D}_z^y &:= \Gamma(1-\alpha)\,\liminf_{h\downarrow 0}\,h^{\alpha} \#\{ x\in [0,1]\colon \cev{\pi}_{[z,\infty)}^y(\{x\})>h \}, \quad &y\ge 0. 
	\end{align*}
	Then the $\alpha$-diversity $\IPLT(\cev{\pi}^y)$ exists at level $y$ if and only if $\overline{D}_0^y= \underline{D}_0^y$. 
	Moreover, we then have $\IPLT(\cev{\pi}^y)= \overline{D}_0^y= \underline{D}_0^y$.
	
	We now fix $y_0>0$ and control uniformly in level $[0,y_0]$ the contributions of newly entered clades. 
	Specifically, the following holds almost surely.\vspace{-0.1cm}
	\begin{center}\begin{minipage}{0.85\linewidth} 
			{\bf Claim 1:} Almost surely, for all $\epsilon>0$, there exists $\delta^\prime>0$   (that depends on $\epsilon$ and the realization) such that 
			\begin{equation}\label{eqn:lemma1}
			\sup_{y\in (z, z+\delta^\prime]} \overline{D}_z^y <\epsilon/3, \quad\text{ for every}~ z\in [0,y_0].
			\end{equation}
	\end{minipage}\end{center}
	To prove this, it suffices to prove that for any fixed $\epsilon>0$ there exists $\delta^\prime>0$ a.s. such that \eqref{eqn:lemma1} holds. 
	Then we deduce the desired property by taking intersection of the almost sure events for each $\epsilon \in \{ 1/n, n\in \mathbb{N}\}$.
	
	We first consider the case $\theta=\alpha$. In this case, we are in the setting of Proposition~\ref{prop:PRM:diversity},  so the existence and continuity of $\alpha$-diversity process $(\IPLT(\cev{\pi}^y),\,y\ge 0)$ has been justified.   
	Then for each $z\ge 0$, the same conclusion holds for the process $\cev{\boldsymbol{\pi}}_z:=(\cev{\pi}_{[z,\infty)}^{z+y}, y\ge 0)$. 
	Fix $\epsilon>0$. Let $r_0=0$. Inductively, by continuity of $y\mapsto\IPLT(\cev{\pi}_{r_n}^{y})$, we have for all $n\ge0$
	\begin{align*}
	r_{n+1}:=&\; r_{n} + \inf\bigg\{x\ge 0\colon\sup_{y\in(0,2x)}\IPLT(\cev{\pi}_{[r_{n},r_{n}+x)}^{r_n+y})\ge\epsilon/9\bigg\}\\
	\ge&\; r_{n} + \inf\bigg\{x\ge 0\colon\sup_{y\in(0,2x)}\IPLT (\cev{\pi}_{r_{n}}^y) \ge\epsilon/9\bigg\}.
	\end{align*} 
	Since also $\cev{\boldsymbol{\pi}}_{r_n}$ is independent of $(r_0,\ldots,r_n)$, the process $(r_n,n\ge 0)$ is a random walk with i.i.d. increments $\delta_{n}=r_{n+1}-r_{n}>0$ a.s.. 
	In particular, $k=\inf\{n\ge 1\colon r_n>y_0\}<\infty$ and $\delta=\min\{\delta_0,\ldots,\delta_k\}>0$ a.s.. 
	Now for all $z\!\in\![0,y_0]$ and $n$ such that $r_{n}\!\le\! z\!<\!r_{n+1}$, we have $z+\delta\!<\!r_{n+1}+\delta\!\le \!r_{n+2}$ and
	\[\sup_{y\in (z, z+\delta]}\! \overline{D}_z^y
	\le\! \sup_{y\in (r_{n}, r_{n+2})}\!\overline{D}_{r_{n}}^y
	=\! \sup_{y\in (r_{n},r_{n+2})}\!\left(\! \IPLT(\cev{\pi}_{[r_{n},r_{n+1})}^y)+ \IPLT(\cev{\pi}_{[r_{n+1},r_{n+2})}^y)\!\right)\!,
	\]
	is at most $2\epsilon/9\!<\!\epsilon/3$. This ends the proof of Claim 1 if $\theta\!=\! \alpha$. The general case $\theta\!\ne\! \alpha$ follows by thinning and superposition as in the proof of Proposition \ref{prop:cevF:path-continuity}, since $\overline{D}_z^y$ decreases when thinning and
	is (sub)additive in superpositions.\smallskip
	
	Next, we control each clade and observe:\vspace{-0.1cm}
	\begin{center}\begin{minipage}{0.85\linewidth} 
			{\bf Claim 2:} Almost surely, for all points $(s,V_s,X_s)$ of $\cev\bF_\theta$ with\linebreak $s\!\in\![0,y_0)$, the process $y\!\mapsto\!\IPLT(\sskewer(y,V_s,X_s))$ is continuous.\vspace{-0.1cm}
	\end{minipage}\end{center}  
	This follows straight from \eqref{eq:clade_shift}, Proposition~\ref{prop:PRM:diversity}, and standard properties of Poisson random measures, which also yield the following.\vspace{-0.1cm} 
	\begin{center}\begin{minipage}{0.85\linewidth} 
			{\bf Claim 3:} Almost surely, for all $\delta^\prime>0$, there are at most finitely many points $(s,V_s,X_s)$ of $\cev\bF_\theta$ with $s\in[0,y_0)$ and $\zeta(V_s)>\delta^\prime$.\vspace{-0.1cm}
	\end{minipage}\end{center}  
	For the remainder of this proof, we will argue on the intersection of the three almost sure events, on which Claims 1--3 hold. 
	Fix any $\epsilon>0$ and take $\delta^\prime>0$ so that \eqref{eqn:lemma1} holds. 
	By Claims 2 and 3, the evolution of ``long-living'' clades
	$$\gamma^y=\!\! \sum_{\text{points } (s,V_s,X_s) \text { of } \cev\bF_{\theta},\, s\in [0,y_0)\, \zeta(V_s)>\delta^\prime}\!\!\!\!\!\!\sskewer\left(y\!-\!s  ,V_s,X_s\right)\!,\ \ y\!\in\![0,y_0],$$
	has continuously evolving diversity. 
	In particular, there is $\delta\in(0,\delta^\prime]$ such that for all $x,z\!\in\![0,y_0]$ with $|x\!-\!z|\!<\!\delta$, we have $\left|\IPLT(\gamma^x)\!-\!\IPLT(\gamma^z)\right|\!<\!\varepsilon/3$,
	and hence, by Claim 1, 
	\begin{equation}\label{eqn:cont}
	\left|\overline{D}_0^x-\underline{D}_0^z\right|
	\le\overline{D}_{\max\{x-\delta^\prime,0\}}^x+\overline{D}_{\max\{z-\delta^\prime,0\}}^z+\left|\IPLT(\gamma^x)-\IPLT(\gamma^z)\right|<\varepsilon.
	\end{equation}   
	Since $\varepsilon$ was arbitrary, the choice $x=z$ 
	yields $\overline{D}_0^z=\underline{D}_0^z$ and hence the existence of $\IPLT(\cev{\pi}^z)=\overline{D}_0^z=\underline{D}_0^z$ for all $z\in[0,y_0]$. In particular, the 
	LHS in \eqref{eqn:cont} is $|\IPLT(\cev{\pi}^x)-\IPLT(\cev{\pi}^z)|$ and we conclude the continuity of $z\mapsto\IPLT(\cev{\pi}^z)$.  
\end{proof}

\begin{proof}[Proof of Theorem \ref{thm:property2}] By the construction of ${\tt FV}_\pi(\alpha,\theta)$ processes from 
  ${\tt SSSP}_\pi(\alpha,\theta)$ processes via a continuous time-change and normalisation by 
  the continuous total mass process, ${\tt FV}_\pi(\alpha,\theta)$ inherits the continuity of the diversity process from 
  ${\tt SSSP}_\pi(\alpha,\theta)$. Since, by Proposition \ref{prop:branching}, an ${\tt SSSP}_\pi(\alpha,\theta)$ can be constructed by adding independent 
  ${\tt SSSP}_\pi(\alpha,0)$ and ${\tt SSSP}_0(\alpha,\theta)$, the continuity of its $\alpha$-diversity process follows from Propositions \ref{prop:0:diversity} and \ref{contdiv:gen}, with the subtleties about the $\alpha$-diversity at time 0 obtained in Proposition \ref{prop:0:diversity}
\end{proof}

\subsection{Coupling ${\tt FV}(\alpha,\theta)$ and ${\tt FV}(0,\theta)$}\label{sec:property3}

Shiga \cite[(3.12) and Theorem 3.6]{Shiga1990} gave a Poissonian construction for a large class of measure-valued processes. 
Let us discuss it for a $\BESQ[0]$ excursion law $\nu_{\tt BESQ}^{(0)}$ in the sense of Pitman and Yor \cite{PitmYor82}:
for $\pi\!\in\!\mathcal{M}$, consider independent $\mathbf{W}_\pi\!\sim\!{\tt PRM}(\nu_{\tt BESQ}^{(0)}\otimes\pi)$ and 
$\cev{\mathbf{W}}\!\sim\!{\tt PRM}(2\theta{\tt Leb}\otimes\nu_{\tt BESQ}^{(0)}\otimes{\tt Unif})$. 
The process
\begin{equation}\label{Shigaconst}
 \pi^y_0(dx)=\int_{(0,y]\times\mathcal{E}}\!f(y\!-\!s)\delta(x)\cev{\mathbf{W}}(ds,df,dx)+\int_\mathcal{E}f(y)\delta(x)\mathbf{W}_\pi(df,dx),
\end{equation}
$y\!>\!0$, uniquely solves the martingale problem associated with the generator
$$LF(\lambda)=2\int_{[0,1]}\frac{\delta^2F(\lambda)}{\delta\lambda(x)^2}\lambda(dx)+2\theta\int_{[0,1]}\frac{\delta F(\lambda)}{\delta\lambda(x)}dx,$$
on a domain of functions $F$ of the form $F(\lambda)=g(\langle\phi_1,\lambda\rangle,\ldots,\langle\phi_k,\lambda\rangle)$ for some bounded measurable 
$\phi_i\colon[0,1]\rightarrow\mathbb{R}$ and bounded twice continuously differentiable functions $g\colon\mathbb{R}^k\rightarrow\mathbb{R}$, $k\ge 1$. Here, for such $F$
$$\frac{\delta F(\lambda)}{\delta\lambda(x)}:=\sum_{i=1}^k\phi_i(x)\partial_ig(\langle\phi_1,\lambda\rangle,\ldots,\langle\phi_k,\lambda\rangle),\quad\lambda\in\mathcal{M},$$
and $\delta^2 F(\lambda)/\delta\lambda(x)^2$ is defined similarly. 
Then $(\overline{\pi}^u_0,\,u\ge 0)$ defined via de-Poissonization as Definition \ref{def:FV} is the Fleming--Viot process of Ethier and Kurtz 
\cite{EthiKurt87}, see also \cite[Theorem$\,$8.1]{EthiKurt93}, with stationary distribution ${\tt PDRM}(0,\theta)$, which we denote by 
${\tt FV}(0,\theta)$. We also refer to $(\pi_0^y,\,y\ge 0)$ as ${\tt SSSP}(0,\theta)$. 

Recall from \cite[Theorem (4.1)]{PitmYor82} that $Z\!=\!(Z_y,y\!\ge\! 0)\!\sim\!{\tt BESQ}_a(0)$ can be constructed from $\mathbf{M}\!\sim\!{\tt PRM}(a\nu_{\tt BESQ}^{(0)})$ as $Z_y\!=\!\sum_{{\rm points\,}\!f{\rm\,of\,}\mathbf{M}}f(y)\!=\!\int_\mathcal{E}f(y)\mathbf{M}(df)$. We \linebreak denote by $\kappa(g,dM)$ a regular conditional distribution of $\mathbf{P}(\mathbf{M}\!\in\! dM\,|\,Z\!=\!g)$. 
In \eqref{Shigaconst}, Shiga's construction builds $\BESQ_a(0)$ mass evolutions for atoms of initial mass $a$ in this manner, via $\mathbf{W}_\pi$.

To obtain a more precise connection between Shiga's framework and our framework, consider the following maps on $\mathcal{V}\times\mathcal{D}$:
\begin{itemize}
  \item $\varphi_{\rm mass}(V,X):=\big\|\overline{\sskewer}(V,X)\big\|$, which is  
  $\mathcal{E}$-valued except on a subset that is $\overline{\nu}_{\perp\rm cld}^{(\alpha)}$-null and $\mathbf{Q}_{b,x}^{(\alpha)}$-null for
  all $b>0$, $x\in[0,1]$;
  \item $\varphi_{\rm type}(V,X):=x$ if $V$ has a unique initial point of the form $b\delta(0,f,x)$; this is well-defined $\mathbf{Q}_{b,x}^{(\alpha)}$-almost surely for all $b>0$, $x\in[0,1]$.
\end{itemize}


In the left panel of Figure \ref{fig:theta_thinning} we see a descending scaffolding and spindles with clades of the reflected process underlined with distinct colors. One may think of those underline colors as representing independent $\Unif$ types with which the clades of $\cev{\mathbf{F}}\sim\mathbf{Q}_0^{\alpha,\theta}$ are marked. We obtain a point measure $\cev{\mathbf{W}}$ as in \eqref{Shigaconst} by replacing each clade of $\cev{\mathbf{F}}$ by its total mass evolution, marked by this independent type.

\begin{theorem}\label{thm:shiga} Fix $\pi\!\in\!\mathcal{M}^a$. Let $\mathbf{F}_\pi\!\sim\!\mathbf{Q}_\pi^{\alpha,0}$ and 
  $\cev{\mathbf{F}}\!\sim\!\mathbf{Q}_0^{\alpha,\theta}$ be independent point measures. Define $\cev{\mathbf{W}}$ by mapping $\cev{\mathbf{F}}$ via 
  $(y,\!V,X)\!\mapsto\!(y,\varphi_{\rm mass}(V,X))$ and marking by independent $\Unif$ types. Define $\mathbf{W}_\pi^\circ$ by 
  mapping each point of $\mathbf{F}_\pi$ via $(V,X)\!\mapsto\! (g,x)\!:=\!(\varphi_{\rm mass}(V,X),\varphi_{\rm type}(V,X))$, then marking $(g,x)$ by the kernel $\kappa(g,dM)$.
  Finally superpose the points $(g_i,x_i,M_i)$ of $\mathbf{W}_\pi^\circ$ to a point measure 
  $\mathbf{W}_\pi:=\sum_{{\rm points\;}(g,x,M){\rm\;of\;}\mathbf{W}_\pi^\circ}\sum_{{\rm points\;}f{\;\rm of\;}M}\delta(f,x).$ 
  Then $\mathbf{W}_\pi \sim {\tt PRM}(\nu_{\tt BESQ}^{(0)}\otimes\pi)$ and, independently, 
  $\cev{\mathbf{W}} \sim {\tt PRM}(2\theta{\tt Leb}\otimes\nu_{\tt BESQ}^{(0)}\otimes{\tt Unif})$.
\end{theorem}
\begin{proof} By Corollary \ref{cor:besq0exc}, $2\alpha\nu_{\tt BESQ}^{(0)}$ is the push forward of $\overline{\nu}_{\perp\rm cld}^{(\alpha)}$ via $\varphi_{\rm mass}$. By Proposition \ref{prop:0:mass}, each point $(V_i,X_i)\sim\mathbf{Q}_{b_i,x_i}^{(\alpha)}$ of $\mathbf{F}_\pi$ is mapped to a 
  ${\tt BESQ}_{b_i}(0)$, independently. Therefore, this proof is completed by standard mapping, marking and superposition of Poisson random
  measures.
\end{proof}

In order to state the following corollary, it is helpful to bring outcomes $\omega\in\Omega$ into our notation for superprocesses, e.g.\ $(\pi_\omega^y,\,y\ge0)$ would be a sample path of an {\tt SSSP} defined on a probability space $(\Omega,\mathcal{F},\Pr)$.

\begin{corollary}\label{cor:shiga1}
 Fix $\alpha\in (0,1)$ and $\pi\in\mathcal{M}^a$. There exists a coupling of $(\pi^y,\,y\ge0)\sim {\tt SSSP}_\pi(\alpha,\theta)$ and $(\lambda^y,\,y\ge0) \sim {\tt SSSP}_\pi(0,\theta)$ on a probability space $(\Omega,\mathcal{F},\mathbf{P})$ and a measurable function $h \colon \Omega\times [0,1]\to [0,1]$ such that
 \begin{equation*}
  \lambda_\omega^y = \pi_\omega^y\circ h_\omega^{-1}\quad \text{where }h_\omega(u) := h(\omega,u),\quad \text{for all }\omega\in\Omega,\ y\ge0,
 \end{equation*}
 i.e.\ $\lambda^y$ is the $h$-pushforward of $\pi^y$. In particular, $\|\lambda^y\| = \|\pi^y\|$ for all $y\ge0$.
\end{corollary}
\begin{proof}
 We apply the map $\overline{\sskewer}$ to the coupled point measures of Theorem \ref{thm:shiga} and take $h_\omega$ to be the function that maps all types in each clade of $\cev{\mathbf{F}}(\omega)$ or $\mathbf{F}_\pi(\omega)$ to the type of the corresponding point in $\cev{\mathbf{W}}(\omega)$ or $\mathbf{W}_\pi^\circ(\omega)$, respectively, while sending the remainder of $[0,1]$ arbitrarily to 0.
\end{proof}

We recall a Poisson--Dirichlet identity \cite[(5.26)]{CSP}. Let $(A'_i,\,i\ge1)\sim\PoiDir(0,\theta)$ and independently for each $i$, let $(B_{ij},\,j\ge1) \sim \PoiDir(\alpha,0)$. Independent of these variables, let $(U'_i,\,i\ge1)$ and, for each $i$, $(U_{ij},\,j\ge1)$, be mutually independent i.i.d.\ \Unif\ sequences. Then, setting $A_{ij} := A'_iB_{ij}$,
\begin{equation}\label{eq:a0_frag}
 \Pi' := \sum_{i\ge1} A'_i\delta(U'_i) \sim \PDRM(0,\theta)\text{\ \ and\ \ } \Pi := \sum_{i\ge1}\sum_{j\ge1} A_{ij}\delta(U_{ij}) \sim \PDRM(\alpha,\theta).
\end{equation}
The random measure $\Pi$ is called an \emph{$(\alpha,0)$-fragmentation} of $\Pi'$.

\begin{corollary}
 For any $\alpha\!\in\!(0,1)$, there is a coupling of jointly stationary ${\tt FV}(\alpha,\theta)$ and ${\tt FV}(0,\theta)$, with the joint stationary law that the ${\tt PDRM}(\alpha,\theta)$ is an $(\alpha,0)$-fragmentation of the ${\tt PDRM}(0,\theta)$.
\end{corollary}
\begin{proof}
 We consider initial measures $\Pi'$ and $\Pi$ as in \eqref{eq:a0_frag}. 
  We will modify the coupling of Theorem \ref{thm:shiga} by noting that each point $(V_{ij},X_{ij})$ of $\mathbf{F}_{\Pi}$ can be further 
  associated with the corresponding type $U_i^\prime$ in $\Pi^\prime$. 
  Specifically, by mapping $(V_{ij},X_{ij},U_i^\prime)$ to $(g,x):=(\varphi_{\rm mass}(V_{ij},X_{ij}),U_i^\prime)$, all
  initial mass $A_i^\prime=\sum_{j\ge 1}A_{ij}$ will be positioned at $U_i^\prime$. 
  With $\cev{\mathbf{F}}$ as in Theorem \ref{thm:shiga}, the proof of that theorem yields
  $\mathbf{\mathbf{W}}_{\Pi^\prime}$ and $\cev{\mathbf{W}}$ as required for ${\tt SSSP}(0,\theta)$ starting from $\Pi^\prime$.
  
  As in Corollary \ref{cor:shiga1}, the total mass processes and so the de-Poissonization time-changes of the associated 
  ${\tt SSSP}(\alpha,\theta)$ and ${\tt SSSP}(\alpha,0)$ coincide. This establishes a coupling of stationary ${\tt FV}(\alpha,\theta)$ and
  ${\tt FV}(0,\theta)$. The joint stationarity follows from \eqref{eq:level_y_dist} and the arguments in Section \ref{sec:pseudostat}.  
\end{proof}

Shiga's construction of \cite[(3.12) and Theorem 3.6]{Shiga1990} requires that type mass evolutions -- spindles, in our terminology -- be continuous-state branching processes (CSBPs), i.e.\ without emigration, while our $\BESQ(-2\alpha)$ spindle masses can be viewed as CSBPs with emigration \cite{PW18}. Consequently, our ${\tt SSSP}(\alpha,\theta)$ are outside the class of processes considered by Shiga.




\appendix

\section{Proofs of properties in Section 4 
}\label{sec:min_cld_pf}
\newcommand{\bV}{\mathbf{V}}

In this appendix, 
we complete the proofs of Proposition~\ref{prop:min_cld:stats}, Lemma~\ref{lem:min_cld:mid_spindle}, and Lemma~\ref{lem:min_cld:cnvgc}.

\begin{proof}[Proof of Proposition~\ref{prop:min_cld:stats}]
	\ref{item:MCS:len} 
	The measure $\umCladeAbar\big\{\len \in \cdot\big\}$ is just the L\'evy measure of the subordinator in Proposition~\ref{prop:hitting_time:subord}, therefore, we find its density by the following identity:   
	\begin{equation*}
	\begin{split}
	\int_0^{\infty} (1-e^{-q x}) \umCladeAbar&\big\{\len \in dx\big\}=\left(2^{\alpha}\Gamma(1+\alpha)q\right)^{1/(1+\alpha)}\\
	&= \int_0^{\infty} (1-e^{-q x}) \frac{\left(2^{\alpha}\Gamma(1+\alpha)\right)^{1/(1+\alpha)}}
	{\Gamma\big(\frac{\alpha}{1+\alpha}\big) (1+\alpha) } x^{-1-1/(1+\alpha)} dx. 
	\end{split}
	\end{equation*}
	
	\ref{item:MCS:max}	
	It is known that \cite[Page 222]{AvrKypPis04} 
	\[
	\umCladeAbar\big\{ \life^+ > z \big\} = \nu_{\perp\rm stb}^{1+\alpha}\big\{g\in\mathcal{D}\colon\textstyle\sup_{t\in\mathbb{R}} g(t)>z\big\}=W'(z)/W(z),
	\]
	where $W(z)$ is the scale function of the stable process with Laplace exponent $\psi_{\alpha}$ of \eqref{eq:scaff:laplace}, satisfying 
	$
	\int_0^{\infty}\! e^{-c z} W(z) dz\!=\!\big(\psi_{\alpha}(c)\big)^{-1}\!\! =\! 2^{\alpha} \Gamma(1\!+\!\alpha) c^{-(1+\alpha)}. 
	$ 
	Then 
	\begin{equation}\label{scalefn}
	 W(z) = 2^{\alpha} z^{\alpha}, 
	 \end{equation}
	 and we deduce the desired formula. 
	 
	 \ref{item:mcl:first_jump} Similar arguments appear around \cite[Propositions A.2, A.3]{Paper1-1}. It suffices to show that our \StableA\ scaffolding, $\bX$, does not jump away from a historic minimum, i.e.\ there is a.s.\ no $t >0$ for which $\bX(t-) = \inf\{\bX(s)\colon s<t\}$ and $\bX(t-)<\bX(t)$. And by the L\'evy process properties of $\bX$, it suffices to prove this for $t\in (0,1)$.
	 
	 By the strong Markov property, there is a.s.\ no $t \in (0,1)$ at which $\bX(t-)<\bX(t)$ and $\bX(t) = \sup\{\bX(s)\colon s\in (t,1]\}$, since each of the countably many jumps of $\bX$ can be captured by a stopping time. Now, the desired property follows by invariance under increment reversal.
\end{proof}

\begin{lemma}\label{lem:mcl:overshoot}
 Take $0<y<z$ and $(\bv,\bx)\sim\umCladeAbar\big(\,\cdot\; |\; \life^+ > y \big)$. Then the probability that the first jump of $\bx$ across level $y$ also exceeds $z$ is bounded above by $y/z$.
\end{lemma}

\begin{proof}
By Proposition~\ref{prop:min_cld:stats}\ref{item:MCS:max}, $\umCladeAbar(\life^+ > y \mid \life^+ > z)= z/y$. This bounds the probability that the first jump of $\bx$ across $y$ also crosses $z$.
\end{proof}

\begin{proof}[Proof of Lemma \ref{lem:min_cld:mid_spindle}] (Mid-spindle Markov property for $\overline{\nu}_{\perp\rm cld}^{(\alpha)}$).
Let $\bV\sim \PRMLBAU$ and let $\uX$ denote the associated scaffolding reflected at its minimum, as in \eqref{eq:reflected_scaff}. Let $T = T^{\ge y} := \inf\{t\!\ge\!0\colon \uX(t)\ge y\}$. The point measure $\bV$ a.s.\ has a point $(T,f_T,x_T)$; let $\check f^y_T$ and $\hat f^y_T$ denote the components of the spindle $f_T$ broken around level $y$ in $\uX$, as in Figure~\ref{fig:min_cld_y}. Let $m^y(\bV,\bX_{\perp}) := \hat f^y_T(0)$ denote the mass of the broken spindles at the break. Let
 $$G := \sup\{t<T\colon \uX(t) = 0\},\quad D := \inf\{t>T\colon \uX(t) = 0\},$$
 so $\bv' := \shiftrestrict{\bV}{[G,D]}$ has the same distribution as $\bv$ in the lemma statement.
 
 The proof of the mid-spindle Markov property  \cite[Lemma~4.13]{Paper1-1} of $(\bN,\bX)$ modified by (i) including type labels with the point measure $\bV$ in place of $\bN:=\varphi(\bV)\sim\PRMLBA$, as in the proof of Proposition~\ref{MSMP}, and (ii) substituting $\uX$ in place of $\bX := \xi(\bN)$, yields here that
 \begin{equation}\label{eq:std_MSMP}
 \begin{array}{c}
  \ShiftRestrict{\bV}{(T,\infty)} + \DiracBig{0,\hat f_T^y,x_T}\text{\ is conditionally independent of}\\[5pt]
  \ShiftRestrict{\bV}{[0,T)} + \Dirac{T,\check f_T^y,x_T}\text{\ given }\big(m^y(\bV,\bX_\perp),x_T\big),
 \end{array}
 \end{equation}
 with $\hat f$ being conditionally distributed as a \BESQ[-2\alpha] started at $m^y(\bV,\bX_\perp)$ and $\ShiftRestrict{\bV}{(T,\infty)}\stackrel{d}{=}\bV$ being conditionally (and unconditionally) independent of $(\hat f,x_T)$.
 
 Let $T' := T-G$. The random time $G$ is a function of the measure in the second line of \eqref{eq:std_MSMP}, and $D-T$ is the stopping time for the measure in the first line at which its associated scaffolding first hits $-y$. Thus, noting that $m^y(\bv',\bx') = m^y(\bV,\bX_\perp)$,
 \begin{equation*}
 \begin{array}{c}
  \ShiftRestrict{\bv'}{(T',\infty)} + \DiracBig{0,\hat f_T^y,x_T}\text{\ is conditionally independent of}\\[5pt] 
  \ShiftRestrict{\bv'}{[0,T')} + \Dirac{T',\check f_T^y,x_T}\text{\ given }(m^y(\bv',\bx'),x_T),
 \end{array}
 \end{equation*}
 with the desired regular conditional distribution.
\end{proof}

\begin{proof}[Proof of Lemma \ref{lem:min_cld:cnvgc}] 
  Let $X$ be \Stable[1\!+\!\alpha] and $\mathbb{P}_a$ the law on $\mathcal{D}_{\rm exc}$ of $X$ 
  starting from $a$ and absorbed when first reaching 0. Recall notation $\len(X)=\sup\{t\ge 0\colon X(t)\neq 0\}$ for the excursion length.
  Let $t>0$. By \cite[Corollary 3]{Chaumont1994a} and the Markov property
  $$\mathbb{E}_a[H\,|\,\len>t]\rightarrow\nu_{\perp\rm stb}^{1+\alpha}[H\,|\,\len>t]\qquad\mbox{as }a\downarrow 0,$$
  for all bounded continuous $H\colon\mathcal{D}\rightarrow\mathbb{R}$. Also, by \cite[Lemma 1]{Chaumont1997} and 
  \cite[Example in Section 4]{ChauDone2005},
  $$\mathbb{P}_a\{\len>t\}/\mathbb{P}_a\{\zeta^+>y\}\rightarrow\nu_{\perp\rm stb}^{1+\alpha}\{\len>t\}/\nu_{\perp\rm stb}^{1+\alpha}\{\zeta^+>y\}\qquad\mbox{as }a\downarrow 0.$$
  Now consider any open $A\subseteq\mathcal{D}_{\rm exc}$. 
  Then for every $t>0$, since $\zeta^+$ is continuous, we have by the Portmanteau Theorem,
  \begin{align*}
    \liminf_{a\downarrow 0}\mathbb{P}_a(A\,|\,\zeta^+>y)&\ge \liminf_{a\downarrow 0}\mathbb{P}_a\left(A\cap\{\len>t\}\left|\,\zeta^+>y\right.\right)\\
      &=\liminf_{a\downarrow 0}\frac{\mathbb{P}_a\{\len>t\}}{\mathbb{P}_a\{\zeta^+>y\}}\mathbb{P}_a\left(\left.A\cap\{\zeta^+>y\}\,\right|\,\len>t\right)\\
      &\ge\frac{\nu_{\perp\rm stb}^{1+\alpha}\{\len>t\}}{\nu_{\perp\rm stb}^{1+\alpha}\{\zeta^+>y\}}\nu_{\perp\rm stb}^{1+\alpha}\left(\left.A\cap\{\zeta^+>y\}\,\right|\,\len>t\right)\\
      &=\nu_{\perp\rm stb}^{1+\alpha}\left(A\cap\{\len>t\}\left|\,\zeta^+>y\right.\right).
  \end{align*}
  Letting $t\downarrow 0$, we find that $\liminf_{a\downarrow 0}\mathbb{P}_a(A\,|\,\zeta^+>y)\ge\nu_{\perp\rm stb}^{1+\alpha}(A\,|\,\zeta^+>y)$, and the conclusion follows from the Portmanteau Theorem. 
\end{proof}


\bibliographystyle{abbrv}
\bibliography{AldousDiffusion4}
\end{document}

%% file: ADMacros.tex
\theoremstyle{plain}

\newtheorem{theorem}{Theorem}[section]

\newtheorem{proposition}[theorem]{Proposition}
\newtheorem{corollary}[theorem]{Corollary}
\newtheorem{lemma}[theorem]{Lemma}

\theoremstyle{definition}
\newtheorem{definition}[theorem]{Definition}

\theoremstyle{remark}

\makeatletter
 \DeclareRobustCommand{\checkarg}{\@ifnextchar[{\@witharg}{}}
 \DeclareRobustCommand{\@witharg}[1][]{\ensuremath{\left(#1\right)}}
 
 \DeclareRobustCommand{\scaleGen}[1]{\@ifnextchar[{\@scalewithargs{#1}}{\odot^{}_{#1}}}
 \def\@scalewithargs#1[#2][#3]{#2 \odot^{}_{#1} #3}
 
 \DeclareRobustCommand{\blankBinOp}{\@ifnextchar[{\@blankbinopwithargs}{}}
 \def\@blankbinopwithargs[#1][#2]{#1 #2}
\makeatother

\def\Exc{\mathcal{E}}

\def\mBxcA{\nu_{\BESQ}^{(-2\alpha)}}	

\def\cC{\mathcal{C}}	

\def\cD{\mathcal{D}}

\def\mSxcA#1{\nu_{#1\textnormal{stb}}^{1+\alpha}}	

\def\uXA{\bX_{\perp}}	
\def\uX{\bX_{\perp}}	
\def\uF{\bF_{\perp}}	
\def\ue{\mathbf{e}_{\perp}}	

\def\cM{\mathcal{M}}

\def\cN{\mathcal{N}}

\def\cS{\mathcal{S}}

\def\len{\textnormal{len}}			
\def\life{\zeta}					

\def\skewer{\textsc{skewer}}		
\newcommand{\xiA}{\xi}

\def\Dirac#1{\delta\left( #1 \right)}
\def\DiracBig#1{\delta\big( #1 \big)}	

\def\scaleB{\scaleGen{\textnormal{spdl}}}		

\def\scaleHA{\odot^{\alpha}_{\textnormal{cld}}}		
\def\scaleI{\blankBinOp}	

\def\ShiftRestrict#1#2{#1\big|^{\leftarrow}_{#2}} 
\def\shiftrestrict#1#2{#1|^{\leftarrow}_{#2}}

\def\Restrict#1#2{#1\big|_{#2}}
\def\restrict#1#2{#1|_{#2}}

\def\bN{\mathbf{N}}			
\def\bF{\mathbf{F}}			
\def\bG{\mathbf{G}}			
\def\bX{\mathbf{X}}			

\def\bx{\mathbf{x}}
\def\bv{\mathbf{v}}

\newcommand{\IPLT}{\mathscr{D}}

\def\bff{\mathbf{f}}		

\newcommand{\ol}[1]{\widebar{#1}}

\def\BR{\mathbb{R}}				
\def\BN{\mathbb{N}}				
\def\BQ{\mathbb{Q}}				
\def\Leb{\textnormal{Leb}}		

\def\to{\rightarrow}
\def\downto{\downarrow}

\def\cf{\mathbf{1}}				
\def\Pr{\mathbf{P}}				
\def\bQ{\mathbf{Q}}				
\def\BQ{\mathbb{Q}}				
\def\EV{\mathbf{E}}				
\def\cF{\mathcal{F}}			

\def\cB{\mathcal{B}}	
\def\cM{\mathcal{M}}	

\def\distribfont#1{\texttt{\upshape #1}}

\def\Unif{\distribfont{Unif}}
\def\ExpDist{\distribfont{Exponential}\checkarg}
\def\GammaDist{\distribfont{Gamma}\checkarg}

\def\BetaDist{\distribfont{Beta}\checkarg}

\def\PoiDir{\distribfont{PD}\checkarg}
\def\PoiDirAT{\PoiDir[\alpha,\theta]}
\def\PRM{\distribfont{PRM}\checkarg}
\def\PRMLBA{\ensuremath{\distribfont{PRM}\left(\Leb\otimes\mBxcA\right)}}

\def\Stable{\distribfont{Stable}\checkarg}
\def\StableA{\ensuremath{\distribfont{Stable}(1\!+\!\alpha)}}
\def\BESQ{\distribfont{BESQ}\checkarg}

\def\PDRM{\distribfont{PDRM}\checkarg}
\def\PDRMAT{\PDRM[\alpha,\theta]}

\stackMath
\newcommand{\cev}[1]{\ThisStyle{
	\stackengine{0ex}{\SavedStyle#1}{\SavedStyle
		\scaleobj{0.63}{\leftharpoonup}}{O}{c}{F}{\useanchorwidth}{S}}}

\makeatletter
\let\save@mathaccent\mathaccent
\newcommand*\if@single[3]{%
  \setbox0\hbox{${\mathaccent"0362{#1}}^H$}%
  \setbox2\hbox{${\mathaccent"0362{\kern0pt#1}}^H$}%
  \ifdim\ht0=\ht2 #3\else #2\fi
  }
\newcommand*\rel@kern[1]{\kern#1\dimexpr\macc@kerna}
\newcommand{\widebar}{}
\DeclareRobustCommand*\widebar[1]{\@ifnextchar^{\wide@bar{#1}{0}}{\wide@bar{#1}{1}}}
\newcommand*\wide@bar[2]{\if@single{#1}{\wide@bar@{#1}{#2}{1}}{\wide@bar@{#1}{#2}{2}}}
\newcommand*\wide@bar@[3]{%
  \begingroup
  \def\mathaccent##1##2{%
    \let\mathaccent\save@mathaccent
    \if#32 \let\macc@nucleus\first@char \fi
    \setbox\z@\hbox{$\macc@style{\macc@nucleus}_{}$}%
    \setbox\tw@\hbox{$\macc@style{\macc@nucleus}{}_{}$}%
    \dimen@\wd\tw@
    \advance\dimen@-\wd\z@
    \divide\dimen@ 3
    \@tempdima\wd\tw@
    \advance\@tempdima-\scriptspace
    \divide\@tempdima 10
    \advance\dimen@-\@tempdima
    \ifdim\dimen@>\z@ \dimen@0pt\fi
    \rel@kern{0.6}\kern-\dimen@
    \if#31
      \overline{\rel@kern{-0.6}\kern\dimen@\macc@nucleus\rel@kern{0.4}\kern\dimen@}%
      \advance\dimen@0.4\dimexpr\macc@kerna
      \let\final@kern#2%
      \ifdim\dimen@<\z@ \let\final@kern1\fi
      \if\final@kern1 \kern-\dimen@\fi
    \else
      \overline{\rel@kern{-0.6}\kern\dimen@#1}%
    \fi
  }%
  \macc@depth\@ne
  \let\math@bgroup\@empty \let\math@egroup\macc@set@skewchar
  \mathsurround\z@ \frozen@everymath{\mathgroup\macc@group\relax}%
  \macc@set@skewchar\relax
  \let\mathaccentV\macc@nested@a
  \if#31
    \macc@nested@a\relax111{#1}%
  \else
    \def\gobble@till@marker##1\endmarker{}%
    \futurelet\first@char\gobble@till@marker#1\endmarker
    \ifcat\noexpand\first@char A\else
      \def\first@char{}%
    \fi
    \macc@nested@a\relax111{\first@char}%
  \fi
  \endgroup
}
\makeatother

%% file: Fig_JCCP_skewer_5b.pdf_t
\begin{picture}(0,0)%
\includegraphics[scale=0.96]{Fig_JCCP_skewer_5b.pdf}%
\end{picture}%
\setlength{\unitlength}{3978sp}
\begingroup\makeatletter\ifx\SetFigFont\undefined%
\gdef\SetFigFont#1#2#3#4#5{%
  \reset@font\fontsize{#1}{#2pt}%
  \fontfamily{#3}\fontseries{#4}\fontshape{#5}%
  \selectfont}%
\fi\endgroup%
\begin{picture}(5901,2233)(304,-2714)
\put(3646,-1556){\makebox(0,0)[b]{\smash{{\SetFigFont{10}{12.0}{\familydefault}{\mddefault}{\updefault}{\color[rgb]{0,0,0}$t$}%
}}}}
\put(406,-1276){\makebox(0,0)[rb]{\smash{{\SetFigFont{11}{13.2}{\familydefault}{\mddefault}{\updefault}{\color[rgb]{0,0,0}$y$}%
}}}}
\put(4501,-638){\makebox(0,0)[rb]{\smash{{\SetFigFont{10}{12.0}{\familydefault}{\mddefault}{\updefault}{\color[rgb]{0,0,0}$(N,X)$}%
}}}}
\put(4501,-811){\makebox(0,0)[rb]{\smash{{\SetFigFont{10}{12.0}{\familydefault}{\mddefault}{\updefault}{\color[rgb]{0,0,0}or $(V,X)$}%
}}}}
\put(3511,-1096){\makebox(0,0)[rb]{\smash{{\SetFigFont{11}{13.2}{\familydefault}{\mddefault}{\updefault}{\color[rgb]{0,0,0}$X(t)$}%
}}}}
\put(3781,-2086){\makebox(0,0)[lb]{\smash{{\SetFigFont{11}{13.2}{\familydefault}{\mddefault}{\updefault}{\color[rgb]{0,0,0}$X(t-)$}%
}}}}
\put(4775,-2396){\makebox(0,0)[lb]{\smash{{\SetFigFont{10}{12.0}{\familydefault}{\mddefault}{\updefault}{\color[rgb]{0,0,0}$0$}%
}}}}
\put(5393,-632){\makebox(0,0)[b]{\smash{{\SetFigFont{10}{12.0}{\familydefault}{\mddefault}{\updefault}{\color[rgb]{0,0,0}$(f_t(z),\,z\geq 0)$}%
}}}}
\put(5196,-2396){\makebox(0,0)[lb]{\smash{{\SetFigFont{10}{12.0}{\familydefault}{\mddefault}{\updefault}{\color[rgb]{0,0,0}$x_t$}%
}}}}
\put(6100,-2396){\makebox(0,0)[lb]{\smash{{\SetFigFont{10}{12.0}{\familydefault}{\mddefault}{\updefault}{\color[rgb]{0,0,0}$1$}%
}}}}
\put(3144,-2364){\makebox(0,0)[b]{\smash{{\SetFigFont{11}{13.2}{\familydefault}{\mddefault}{\updefault}{\color[rgb]{0,0,0}$f_t(y-X(t-))$}%
}}}}
\put(6043,-1440){\makebox(0,0)[lb]{\smash{{\SetFigFont{10}{12.0}{\familydefault}{\mddefault}{\updefault}{\color[rgb]{0,0,0}$z$}%
}}}}
\put(1807,-2645){\makebox(0,0)[rb]{\smash{{\SetFigFont{11}{13.2}{\familydefault}{\mddefault}{\updefault}{\color[rgb]{0,0,0}$\skewer(y,N,X)$}%
}}}}
\put(6190,-2645){\makebox(0,0)[rb]{\smash{{\SetFigFont{11}{13.2}{\familydefault}{\mddefault}{\updefault}{\color[rgb]{0,0,0}$\sskewer(y,V,X)$}%
}}}}
\end{picture}%

%% file: Fig_spindle_split.pdf_t
\begin{picture}(0,0)%
\includegraphics[scale=0.92]{Fig_spindle_split.pdf}%
\end{picture}%
\setlength{\unitlength}{3812sp}
\begingroup\makeatletter\ifx\SetFigFont\undefined%
\gdef\SetFigFont#1#2#3#4#5{%
  \reset@font\fontsize{#1}{#2pt}%
  \fontfamily{#3}\fontseries{#4}\fontshape{#5}%
  \selectfont}%
\fi\endgroup%
\begin{picture}(5958,1779)(619,-1810)
\put(3961,-1124){\makebox(0,0)[b]{\smash{{\SetFigFont{10}{12.0}{\familydefault}{\mddefault}{\updefault}{\color[rgb]{0,0,0}$t$}%
}}}}
\put(5658,-629){\makebox(0,0)[rb]{\smash{{\SetFigFont{10}{12.0}{\familydefault}{\mddefault}{\updefault}{\color[rgb]{0,0,0}$f_t$}%
}}}}
\put(6438,-455){\makebox(0,0)[rb]{\smash{{\SetFigFont{10}{12.0}{\familydefault}{\mddefault}{\updefault}{\color[rgb]{0,0,0}$\hat f_t^y$}%
}}}}
\put(6378,-1655){\makebox(0,0)[rb]{\smash{{\SetFigFont{10}{12.0}{\familydefault}{\mddefault}{\updefault}{\color[rgb]{0,0,0}$\check f_t^y$}%
}}}}
\put(4096,-1636){\makebox(0,0)[lb]{\smash{{\SetFigFont{11}{13.2}{\familydefault}{\mddefault}{\updefault}{\color[rgb]{0,0,0}$X(t-)$}%
}}}}
\put(3826,-646){\makebox(0,0)[rb]{\smash{{\SetFigFont{11}{13.2}{\familydefault}{\mddefault}{\updefault}{\color[rgb]{0,0,0}$X(t)$}%
}}}}
\put(721,-826){\makebox(0,0)[rb]{\smash{{\SetFigFont{11}{13.2}{\familydefault}{\mddefault}{\updefault}{\color[rgb]{0,0,0}$y$}%
}}}}
\end{picture}%

%% file: Fig_cutoff_FV.pdf_t
\begin{picture}(0,0)%
\includegraphics{Fig_cutoff_FV.pdf}%
\end{picture}%
\setlength{\unitlength}{4144sp}%
\begingroup\makeatletter\ifx\SetFigFont\undefined%
\gdef\SetFigFont#1#2#3#4#5{%
  \reset@font\fontsize{#1}{#2pt}%
  \fontfamily{#3}\fontseries{#4}\fontshape{#5}%
  \selectfont}%
\fi\endgroup%
\begin{picture}(6639,1737)(259,-1018)
\put(406,-151){\makebox(0,0)[rb]{\smash{{\SetFigFont{11}{13.2}{\familydefault}{\mddefault}{\updefault}{\color[rgb]{0,0,0}$y$}%
}}}}
\end{picture}%